\documentclass{amsart}

\usepackage{pgf,tikz,pgfplots}
\pgfplotsset{compat=1.15}
\usepackage{mathrsfs}
\usetikzlibrary{arrows}

\usepackage[T1]{fontenc}
\usepackage{lmodern}
\usepackage[colorlinks=true,urlcolor=blue, citecolor=red,linkcolor=blue,linktocpage,pdfpagelabels, bookmarksnumbered,bookmarksopen]{hyperref}
\usepackage[hyperpageref]{backref}
\usepackage{amsthm} 
\usepackage{latexsym,amsmath,amssymb}

\usepackage{accents}
\usepackage{esint}

\usepackage{soul}
\usepackage{mathtools} 
\usepackage{xparse} 
\usepackage[capitalize]{cleveref}

\usepackage[shortlabels]{enumitem}

\usepackage[textsize=small]{todonotes}
 \newcommand{\col}[1]{#1}           

\title[wave systems with antisymmetric potential]{On wave systems with antisymmetric potential in dimension \texorpdfstring{$\lowercase{d} \geq 4$}{d >= 4} and well-posedness for (half-)wave maps}

\newcommand{\N}{{\mathbb N}}

\renewcommand{\S}{{\mathbb S}}

\newtheorem{theorem}{Theorem}
\newtheorem{lemma}[theorem]{Lemma}
\newtheorem{corollary}[theorem]{Corollary}
\newtheorem{proposition}[theorem]{Proposition}

\theoremstyle{definition}
\newtheorem{definition}[theorem]{Definition}

\theoremstyle{remark}

\newcommand\curl{{\rm curl\,}}
\newcommand\lip{{\rm Lip\,}}

\newcommand\supp{{\rm supp\,}}

\newcommand{\R}{\mathbb{R}}

\newcommand{\Z}{\mathbb{Z}}

\newcommand{\brac}[1]{\left (#1 \right )}
\newcommand{\norm}[1]{\left \|#1 \right \|}
\newcommand{\scpr}[2]{\left \langle #1 , #2\right \rangle}
\newcommand{\abs}[1]{\left\lvert #1 \right \rvert}

\renewcommand{\vec}[1]{{\bf #1}}

\renewcommand{\i}{{\rm \bf i}}

\newcommand{\barint}{
\rule[.036in]{.12in}{.009in}\kern-.16in \displaystyle\int }

\newcommand{\barcal}{\text{$ \rule[.036in]{.11in}{.007in}\kern-.128in\int $}}

\def\mvint_#1{\mathchoice
          {\mathop{\vrule width 6pt height 3 pt depth -2.5pt
                  \kern -8pt \intop}\nolimits_{\kern -3pt #1}}%
          {\mathop{\vrule width 5pt height 3 pt depth -2.6pt
                  \kern -6pt \intop}\nolimits_{#1}}%
          {\mathop{\vrule width 5pt height 3 pt depth -2.6pt
                  \kern -6pt \intop}\nolimits_{#1}}%
          {\mathop{\vrule width 5pt height 3 pt depth -2.6pt
                  \kern -6pt \intop}\nolimits_{#1}}}

\numberwithin{theorem}{section} \numberwithin{equation}{section}

\newcommand{\lap}{\Delta }
\newcommand{\aleq}{\lesssim}

\newcommand{\aeq}{\approx}

\newcommand{\Rz}{\mathscr{R}}
\newcommand{\laps}[1]{(-\Delta)^{\frac{#1}{2}}}
\newcommand{\Ds}[1]{|\nabla|^{#1}}
\newcommand{\Dso}{|\nabla|}

\newcommand{\lapms}[1]{\Ds{-#1}}

\usepackage{scalerel}[2014/03/10]
\usepackage[usestackEOL]{stackengine}
\def\avint{\,\ThisStyle{\ensurestackMath{%
			\stackinset{c}{.2\LMpt}{c}{.5\LMpt}{\SavedStyle-}{\SavedStyle\phantom{\int}}}%
		\setbox0=\hbox{$\SavedStyle\int\,$}\kern-\wd0}\int}

\renewcommand{\div}{\operatorname{div}}

\let\latexchi\chi
\makeatletter
\renewcommand\chi{\@ifnextchar_\sub@chi\latexchi}
\newcommand{\sub@chi}[2]{
  \@ifnextchar^{\subsup@chi{#2}}{\latexchi^{}_{#2}}%
}
\newcommand{\subsup@chi}[3]{
  \latexchi_{#1}^{#3}%
}
\makeatother
\newcommand{\eps}{\varepsilon}

\author{Silvino Reyes Farina}
\email[Silvino Reyes Farina]{sir25@pitt.edu}
\author{Armin Schikorra}
\email[Armin Schikorra]{armin@pitt.edu}
\address{Department of Mathematics,
University of Pittsburgh,
301 Thackeray Hall,
Pittsburgh, PA 15260, USA}

\begin{document}
\begin{abstract}
We prove a priori estimates for wave systems of the type
\[
 \partial_{tt} \vec{u} - \lap \vec{u} = \Omega \cdot d\vec{u} + F(\vec{u}) \quad \text{in $\R^d \times \R$}
\]
where $d \geq 4$ and $\Omega$ is a suitable antisymmetric potential. We show that the assumptions on $\Omega$ are applicable to wave- and half-wave maps, the latter by means of the Krieger-Sire reduction. We thus obtain well-posedness of those equations for small initial data in $\dot{H}^{\frac{d}{2}}(\R^d)$.
\end{abstract}

.
\maketitle 
\tableofcontents

\section{Introduction}
In \cite{Riv2007} Rivi\`ere investigated the system 
\begin{equation}\label{eq:riviereeq}
 \lap \vec{u} = \Omega \cdot \nabla \vec{u} \quad \text{in $\R^2$}
\end{equation}
for maps $\vec{u}: \R^2 \to \R^N$, and showed that if $\Omega$ is an antisymmetric $L^2$-potential, namely if
\[
 \Omega_{ij} = -\Omega_{ji} \in L^2(\R^2,\R^2), \quad i,j \in \{1,\ldots,N\}
\]
then an adaptation of Uhlenbeck's Coulomb gauge \cite{Uhlenbeck1982} leads to compensation phenomena, an effect closely related to the moving frame technique by H\`elein \cite{Hel91,S10}. In particular, any $W^{1,2}$-solution $u$ to the above equation must be Lipschitz continuous. This has applications, among many other things, to the regularity theory of geometric equations such as the harmonic map equation. The methods have also been extended to half-harmonic maps in \cite{DaLioRiv11}, and several other elliptic settings, e.g. \cite{LR08,RS08,DL13,SharpTopping13,Seps15,Schikorra18,MS18,GWX23}.

Motivated by this elliptic theory we investigate the wave-version of \eqref{eq:riviereeq}. Namely we consider maps $\vec{u}: \R^d \times (-T,T) \to \R^N$ which solve
\begin{equation}\label{eq:generalwave}
 \begin{cases} 
 \partial_{tt} \vec{u} - \lap \vec{u} = \Omega \cdot d\vec{u} + F(\vec{u}) \quad &\text{in $\R^d \times (-T,T)$}\\
 \vec{u}(0) = u_0 \quad &\text{in $\R^d$} \\
 \partial_t\vec{u}(0) = v_0 \quad &\text{in $\R^d$}
 \end{cases}
\end{equation}

Here $\Omega = \brac{\Omega_{ij}^\alpha}_{\alpha \in \{0,\ldots,d\},\ i,j \in \{1,\ldots,N\} } \in so(N) \otimes \R^d$ is antisymmetric \[\Omega_{ij} =-\Omega_{ji} \in \R^d, \quad i,j \in \{1,\ldots,N\}\] and we denote
\[
\brac{\Omega \cdot d\vec{u}}^i = \sum_{\alpha=0}^d \sum_{j=1}^N \Omega_{ij}^\alpha \partial_\alpha u^j \quad i \in \{1,\ldots,N\}.
\]
We adapt the convention that the $0$-th derivative is the time derivative, i.e. $\partial_0 f = \partial_t f$.

We focus on the case where $\vec{u}$ maps into the sphere $\S^2 \subset \R^3$. Our main result is the following a priori estimate
\begin{theorem}\label{th:aprioriest}
There exists $s_0 > 1$ such that the following holds for any $s \in (1,s_0)$. There exists $\eps > 0$ with the following properties:

Whenever $\Omega$ is admissible in the sense of \Cref{def:admissibleomega} below and $F$ is admissible in the sense of \Cref{def:admissibleF}, and $u: \R^d \times [-T,T] \to {\S^2}$ is a smooth solution of \eqref{eq:generalwave} with initial data bounds
\[
 \|\vec{u_0} \|_{H^{\frac{d}{2}}(\R^d)} + \|\vec{\dot{u}_0} \|_{H^{\frac{d-2}{2}}(\R^d)}< \eps
\]
and
\begin{equation}\label{eq:ap:smallness}
 \sup_{t \in (-T,T)} \|\Ds{\frac{d-2}{2}} d\vec{u}(t)\|_{L^2(\R^d)}  < \eps,
\end{equation}
then if
\[
 \|\Ds{\frac{d-2}{2}} d\vec{u}\|_{C^0_t L^2(\R^d\times (-T,T))} + \|\Ds{s-1}d \vec{u}\|_{L^2_t L^{(\frac{2d}{2s-1},2)}_x(\R^d \times (-T,T))} < \infty
\]
we actually have
\begin{equation}\label{eq:apriori:est}
\begin{split}
 &\|\Ds{\frac{d-2}{2}} d\vec{u}\|_{L^\infty_t L^2(\R^d\times (-T,T))} + \|\Ds{s-1}d \vec{u}\|_{L^2_t L^{(\frac{2d}{2s-1},2)}_x(\R^d \times (-T,T))}\\
 \aleq&\|\Ds{\frac{d}{2}} \vec{u_0}\|_{L^2(\R^d)} + \|\Ds{\frac{d-2}{2}} \vec{\dot{u}_0}\|_{L^2(\R^d)}.
 \end{split}
\end{equation}
\end{theorem}

A priori estimates such as the one above, combined with a uniqueness results, can often be used to obtain an existence theory.

And indeed, as a consequence of \Cref{th:aprioriest} we recover in particular the well-posedness result from Shatah-Struwe \cite{SS02} for wave maps. Moreover, following the idea of Krieger--Sire \cite{KS18} to reduce the half-wave map equation to an equation structurally similar to the wave map equation, we also obtain the following well-posedness result for halfwave maps.

\begin{corollary}[Global Existence half-wave equation]\label{co:wellposedness}
Let $d \geq 4$. There exists $\eps > 0$, $s>1$ such that the following holds.

Assume that $\vec{u_0}: \R^d \to \S^2$ such that $\supp \brac{\vec{u_0}-\vec{q}}$ is compact for some constant vector $\vec{q} \in \S^2$, and we also assume
\begin{equation}\label{eq:smallnessassumptionwellposedness}
 \|\Ds{\frac{d}{2}} \vec{u_0}\|_{L^2(\R^d)} < \eps.
\end{equation}
Then there exists a unique solution $\vec{u}: \R^d \times (-\infty,\infty) \to \S^2$ such that
\[
 \begin{cases} \partial_t \vec{u} = \vec{u} \wedge \laps{1} \vec{u} \quad &\text{in $\R^d \times (-\infty,\infty)$}\\
  \vec{u}(0) = \vec{u_0}\quad &\text{in $\R^d$}
 \end{cases}
\]
with the estimate for all suitably small $s > 1$
\begin{equation}\label{eq:Wellposedest}
\begin{split}
\|\Ds{s-1}d \vec{u}\|_{L^2_t L^{(\frac{2d}{2s-1},2)}_x(\R^d \times (-\infty,\infty))} \aleq  \|\Ds{\frac{d}{2}} \vec{u_0}\|_{L^2(\R^d)}
\end{split}
\end{equation}
\end{corollary}
Halfwave maps appear in the physics literature as the continuum limit of Calagero-Moser spin systems, see \cite{LenzmannSok20} for details. Their mathematical study has been recently initiated in \cite{KS18,LS2018}, see also \cite{LenzmannPrimer18,GL18,Berntson_2020,SWZ21}.

To the best of our knowledge, \Cref{co:wellposedness} is new in dimension $d=4$. Uniqueness was proven for wave maps in \cite{SS02} and for halfwave maps in \cite{ERFS22}. In dimension $d \geq 5$ \cite{KS18}, and in dimension $d=4$ \cite{KK21}, a corresponding existence result for half-wave maps was obtained, but under stronger assumptions. Namely instead of the $H^{\frac{d}{2}}$-assumption \eqref{eq:smallnessassumptionwellposedness} the assumed one involving the Besov space $B^{\frac{d}{2},1}_2$. \cite{Liu2021} obtained the existence part of \Cref{co:wellposedness} for $d\geq 5$ with the condition \eqref{eq:smallnessassumptionwellposedness} not only in the sphere-case $\S^2$, but also for maps into hyperbolic space. \cite{Silvino24} obtained uniqueness results for the hyperbolic version of halfwave maps.
In \cite{Liu2023} existence in the most interesting case $d=1$ is obtained by parabolic approximation. In \cite{Marsden24} further well-posedness in dimension $d=3$ was obtained for small initial data in the critical Besov space with additional assumptions angular regularity and weighted decay of the derivative.

{\bf Outline.} In \Cref{s:admissibility} we discuss the admissibility conditions in \Cref{th:aprioriest} -- these are, besides the antisymmetry of $\Omega$, natural (albeit technical) growth conditions and a curl-condition on $\Omega$. In \Cref{s:gauge} we follow essentially the spirit of Rivi\`ere's \cite{Riv2007} and use Uhlenbeck's gauge to transform $\Omega$ into a divergence-free vectorfield, that thanks to the $curl$-condition grows quadratically. We then discuss the regularity dependency of the transformed $\Omega$ and the gauge, a major technical unpleasantness that takes inspiration from estimates on the Coulomb frame e.g. in Shatah-Struwe's \cite{SS02} for the wave-map equation into general manifolds. In \Cref{s:apriori} we then obtain the proof of \Cref{th:aprioriest} using a space-time Hodge decomposition, Strichartz estimates (hence the dimension $d \geq 4$), and combining all the previous estimates. In \Cref{s:halfwavewave} we show that wavemaps into spheres and halfwave maps fall into the form of our theorem. For the wavemaps this is a direct application of the observation in Rivi\`ere's \cite{Riv2007}; for the halfwave map equation this is more work but similar in spirit. Lastly, in \Cref{s:wellposed} we prove the well-posedness result for half-wave maps under small $H^{\frac{d}{2}}$-initial data, \Cref{co:wellposedness}.

{\bf Notation:}
We work on maps $\vec{u}: \R^d \to \R^N$ (although most of the time we have in mind $\vec{u}:\R^d \to \S^2 \subset \R^3$. Vectorial maps will be written bold-faced.

For presentational reasons we will divide indices of domain vectors in $\R^n$ in Latin script
\[
 a,b,c \in \{1,\ldots,d\},
\]
and Greek script (if time could be considered)
\[
 \alpha,\beta,\gamma \in \{0,1,\ldots,d\}.
\]
Coefficients of target vectors in $\R^N$ are denoted by
\[
 i,j,k, \ell \in \{1,\ldots,N\}.
\]
The vector of first order derivatives is
\[
 d = (\partial_t,\partial_1,\ldots,\partial_d).
\]

\subsection*{Acknowledgment} 
A.S. is an Alexander-von-Humboldt Fellow. S.R.F. and A.S. received funding from NSF Career DMS-2044898.

\section{Growth and compatibility conditions of \texorpdfstring{$\Omega$}{Omega} and \texorpdfstring{$F$}{F}}\label{s:admissibility}

In this section we discuss the admissibility conditions assumed in \Cref{th:aprioriest}.
They may look very technical, but the growth conditions in \Cref{def:FOmegagrowth} below essentially assume that $\Omega$ behaves growth-wise like $u du$ In \Cref{def:admissibleomega} we additionally assume a curl-type growth structure, namely $\curl (\Omega) \aeq |\nabla u|^2$. Observe this is the case if $\Omega = A(u) \nabla u$.

\begin{definition}[Growth conditions for $\Omega$]\label{def:FOmegagrowth}
Assume that for any $\vec{u}: \R^d \times [-T,T] \to {\S}^2$

we have
\begin{itemize}
\item For any $s \in [0,\frac{d-4}{2}]$, $p \in (1,\frac{d}{s+1})$ and $q \in [1,\infty]$
\begin{equation}\label{eq:Omegacond:growthpq}
\|\Ds{s} d\Omega\|_{L^{(p,q)}(\R^d)} \aleq (1+\|\Ds{\frac{d-2}{2}} d\vec{u}\|_{L^2(\R^d)}^\gamma)\, \|\Ds{s+1} d\vec{u}\|_{L^{(p,q)}(\R^d)};
\end{equation}
For $s \in [0,\frac{d-4}{2}]$, $p \in (1,\frac{d}{s})$ and $q \in [1,\infty]$
\begin{equation}\label{eq:Omegacond:growthpq2}
\|\Ds{s} \Omega\|_{L^{(p,q)}(\R^d)} \aleq (1+\|\Ds{\frac{d-2}{2}} d\vec{u}\|_{L^2(\R^d)}^\gamma)\, \|\Ds{s+1} d\vec{u}\|_{L^{(p,q)}(\R^d)}
\end{equation}

In particular we then have
\begin{equation}\label{eq:Omegacond:growth}
\|\Ds{\frac{d-4}{2}} d\Omega\|_{L^2(\R^d)} \aleq (1+\|\Ds{\frac{d-2}{2}} d\vec{u}\|_{L^2(\R^d)}^\gamma)\, \|\Ds{\frac{d-2}{2}} d\vec{u}\|_{L^2(\R^d)}
\end{equation}
and for $q \in [1,\infty]$
\begin{equation}\label{eq:Omegacond:growth2d}
\|\Omega\|_{L^{(2d,q)}(\R^d)} \aleq (1+\|\Ds{\frac{d-2}{2}} d\vec{u}\|_{L^2(\R^d)}^\gamma)\, \|d \vec{u}\|_{L^{(2d,q)}(\R^d)}
\end{equation}
\end{itemize}
\end{definition}

\begin{definition}[Perturbative Assumptions of $F$]\label{def:admissibleF}
\begin{itemize}
 \item For $t \in [1,\frac{d-2}{2}]$, $\sigma \in [0,1]$ such that $\frac{2d}{1+\sigma+2t} \in (1,\infty)$,
\begin{equation}\label{eq:Fcond:2}
 \|\Ds{t-1}  F(\vec{u})\|_{L^{\frac{2d}{1+\sigma+2t}}(\R^d)} \aleq   \brac{\|\Ds{\frac{d-2}{2}} d\vec{u}\|_{L^2(\R^d)}^{\gamma_1} + \|\Ds{\frac{d-2}{2}} d\vec{u}\|_{L^2(\R^d)}^{\gamma_2}}\|\Ds{t} d\vec{u}\|_{L^{\frac{2d}{1+\sigma+2t}}(\R^d)}
\end{equation}
and
\begin{equation}\label{eq:Fcond:2v2}
 \|F(\vec{u})\|_{L^{(\frac{2d}{3},2)}(\R^d)} \aleq   \brac{1 + \|\Ds{\frac{d-2}{2}} d\vec{u}\|_{L^2(\R^d)}^{\gamma_2}}\|d\vec{u}\|_{L^{(2d,2)}(\R^d)}
\end{equation}

 \item  There exists $\gamma_1, \gamma_2 > 0$ such that for any $\vec{u}: \R^d \to \S^2$, and any $s > 1$ we have the following estimate which essentially means $F$ is perturbative
\begin{equation}\label{eq:Fcond:3}
 \|\Ds{\frac{d-2}{2}} F(\vec{u})\|_{L^{2}(\R^d)} \aleq_s
  \brac{1 + \|\Ds{\frac{d-2}{2}} d\vec{u}\|_{L^2(\R^d)}^{\gamma}}
  \|\Ds{s-1} d\vec{u}\|_{L^{(\frac{2d}{2s-1},2)}(\R^d)}^2.
\end{equation}
\end{itemize}
\end{definition}

With the growth conditions on $\Omega$ and perturbative behavior of $F$ we can estimate repeated time derivatives. Namely we have
\begin{lemma}\label{la:dealingwithddu}
Assume $d \geq 4$ and that $F$ and $\Omega$ satisfy the assumptions of \Cref{def:FOmegagrowth}. Assume that 
$\vec{u}: \R^d \times [-T,T] \to \S^2$ solves
\[
 \partial_{tt} \vec{u} - \lap \vec{u} = \Omega \cdot d \vec{u} + F(\vec{u})
\]
Then for some $\gamma > 0$,
\begin{equation}\label{eq:dduest1}
 \|\Ds{\frac{d-4}{2}} dd\vec{u}\|_{L^2(\R^d)} \aleq (1+\|\Ds{\frac{d-2}{2}} d\vec{u}\|_{L^2(\R^d)}^\gamma) \|\Ds{\frac{d-2}{2}} d\vec{u}\|_{L^2(\R^d)}
\end{equation}
and for any $s\in[1,2]$
\begin{equation}\label{eq:dduest2}
\|\Ds{s-2} dd\vec{u}\|_{L^{(\frac{2d}{2s-1},2)}(\R^d)} \aleq  (1+\|\Ds{\frac{d-2}{2}} d\vec{u}\|_{L^2(\R^d)}^\gamma) \|\Ds{s-1}d\vec{u}\|_{L^{(\frac{2d}{2s-1},2)}(\R^d)}.
 \end{equation}

In particular, if \eqref{eq:ap:smallness} holds then so does
\begin{equation}\label{eq:ap:smallnessdd}
 \sup_{t \in (-T,T)} \|\Ds{\frac{d-4}{2}} dd\vec{u}(t)\|_{L^2(\R^d)}  < \eps,
\end{equation}

Moreover, there exists a uniform $\eps > 0$ such that if we assume additionally that 
\begin{equation}\label{eq:411Ssmallv1}
 \|\Ds{\frac{d-4}{2}}dd\vec{u}\|_{L^2(\R^d)} < \eps
\end{equation}
then for any $t \in [1,\frac{d-2}{2}]$, $\sigma \in [0,1]$ so that $\frac{2d}{1+\sigma+2t} \in (1,\infty)$ we have
\begin{equation}\label{eq:Dstm1estddvu:orig}
\begin{split}
 \|\Ds{t-1} dd\vec{u}\|_{L^{\frac{2d}{1+\sigma+2t}}(\R^d)}
 \aleq\|\Ds{t} d\vec{u}\|_{L^{\frac{2d}{1+\sigma+2t}}(\R^d)}.
\end{split}
 \end{equation}

\end{lemma}
\begin{proof}
\underline{As for \eqref{eq:dduest1}} we observe that using fractional Leibniz-rule \Cref{la:basicLeibniz}, \eqref{eq:Fcond:2} for $\sigma = 1$ and $t = \frac{d-2}{2}$, then Sobolev embedding
\[
\begin{split}
 \|\Ds{\frac{d-4}{2}} dd\vec{u}\|_{L^2(\R^d)} \aleq& \|\Ds{\frac{d-4}{2}} \brac{\partial_{tt}-\lap} \vec{u}\|_{L^2(\R^d)} +\|\Ds{\frac{d-2}{2}} d\vec{u}\|_{L^2(\R^d)}\\
 \aleq&\|\Ds{\frac{d-4}{2}} \brac{\Omega \cdot d\vec{u}}\|_{L^2(\R^d)} +\|\Ds{\frac{d-4}{2}} F\brac{\vec{u}}\|_{L^2(\R^d)} + \|\Ds{\frac{d-2}{2}} d\vec{u}\|_{L^2(\R^d)}\\
 \aleq&\max_{s \in [0,\frac{d-4}{2}]} \|\Ds{s} \Omega\|_{L^{\frac{d}{s+1}}(\R^d)} \|\Ds{\frac{d-4-2s}{2}} d\vec{u}\|_{L^{\frac{2d}{d-2s-2}}(\R^d)}\\
 &+
 \brac{\|\Ds{\frac{d-2}{2}} d\vec{u}\|_{L^2(\R^d)}^{\gamma_1} + \|\Ds{\frac{d-2}{2}} d\vec{u}\|_{L^2(\R^d)}^{\gamma_2}}\|\Ds{\frac{d-2}{2}} d\vec{u}\|_{L^{2}(\R^d)}\\
 &+ \|\Ds{\frac{d-2}{2}} d\vec{u}\|_{L^2(\R^d)}\\
 \aleq& \|\Ds{\frac{d-2}{2}} \Omega\|_{L^{2}(\R^d)} \|\Ds{\frac{d-2}{2}} d\vec{u}\|_{L^{2}(\R^d)}\\
 &+
 \brac{\|\Ds{\frac{d-2}{2}} d\vec{u}\|_{L^2(\R^d)}^{\gamma_1} + \|\Ds{\frac{d-2}{2}} d\vec{u}\|_{L^2(\R^d)}^{\gamma_2}}\|\Ds{\frac{d-2}{2}} d\vec{u}\|_{L^{2}(\R^d)}\\
 &+ \|\Ds{\frac{d-2}{2}} d\vec{u}\|_{L^2(\R^d)}\\
 \overset{\eqref{eq:Omegacond:growth}}{\aleq}&\brac{1+ \|\Ds{\frac{d-2}{2}} d\vec{u}\|_{L^2(\R^d)}^{\tilde{\gamma}}}\|\Ds{\frac{d-2}{2}} d\vec{u}\|_{L^{2}(\R^d)}\\
  \end{split}
\]
In the last line we also used Young's inequality $a^{\gamma_1} \leq 1+a^{\gamma_2}$ for $\gamma_1 < \gamma_2$, choosing $\tilde{\gamma} >0$ large enough.

\underline{As for \eqref{eq:dduest2}}, with a similar argument 
\[
\begin{split}
& \|\Ds{s-2} dd\vec{u}\|_{L^{(\frac{2d}{2s-1},2)}(\R^d)}\\
\aleq&  \|\Ds{s-2} \brac{\Omega \cdot d\vec{u}}\|_{L^{(\frac{2d}{2s-1},2)}(\R^d)} + \|\Ds{s-2} F(\vec{u})\|_{L^{(\frac{2d}{2s-1},2)}(\R^d)} +\|\Ds{s-1} d\vec{u}\|_{L^{(\frac{2d}{2s-1},2)}(\R^d)}\\
\overset{s \leq 2}{\aleq} & \|\Omega \cdot d\vec{u}\|_{L^{(\frac{2d}{3},2)}(\R^d)} + \| F(\vec{u})\|_{L^{\frac{2d}{3}}(\R^d)} +\|\Ds{s-1} d\vec{u}\|_{L^{(\frac{2d}{2s-1},2)}(\R^d)}\\
  \aleq & \|\Omega \|_{L^{(d,2)}(\R^d)}\, \|d\vec{u}\|_{L^{2d}(\R^d)} + \| F(\vec{u})\|_{L^{(\frac{2d}{3},2)}(\R^d)} +\|\Ds{s-1} d\vec{u}\|_{L^{(\frac{2d}{2s-1},2)}(\R^d)}\\
  \aleq & \|\Ds{\frac{d-2}{2}}\Omega \|_{L^{2}(\R^d)}\, \|d\vec{u}\|_{L^{2d}(\R^d)} + \| F(\vec{u})\|_{L^{(\frac{2d}{3},2)}(\R^d)} +\|\Ds{s-1} d\vec{u}\|_{L^{(\frac{2d}{2s-1},2)}(\R^d)}\\
  \overset{\eqref{eq:Omegacond:growth}}{\aleq}& \brac{1+\|\Ds{\frac{d-2}{2}} d\vec{u}\|_{L^2(\R^d)}^{1+\gamma} } \|d\vec{u}\|_{L^{2d}(\R^d)} + \| F(\vec{u})\|_{L^{(\frac{2d}{3},2)}(\R^d)} +\|\Ds{s-1} d\vec{u}\|_{L^{(\frac{2d}{2s-1},2)}(\R^d)}\\
 \end{split}
 \]
 We conclude with \eqref{eq:Fcond:2v2}.

 \underline{As for \eqref{eq:Dstm1estddvu:orig}}
  \[
\begin{split}
 &\|\Ds{t-1} dd\vec{u}\|_{L^{\frac{2d}{1+\sigma+2t}}(\R^d)}\\
 \aleq& \|\Ds{t} d\vec{u}\|_{L^{\frac{2d}{1+\sigma+2t}}(\R^d)} + \|\Ds{t-1}  (\partial_{tt} -\lap)\vec{u}\|_{L^{\frac{2d}{1+\sigma+2t}}(\R^d)}\\
  \aleq& \|\Ds{t} d\vec{u}\|_{L^{\frac{2d}{1+\sigma+2t}}(\R^d)}  + \|\Ds{t-1}  \brac{\Omega \cdot d\vec{u}}\|_{L^{\frac{2d}{1+\sigma+2t}}(\R^d)} +\|\Ds{t-1}  F(\vec{u})\|_{L^{\frac{2d}{1+\sigma+2t}}(\R^d)} \\
  \overset{\eqref{eq:Fcond:2}, \eqref{eq:411Ssmallv1}}{\aleq}&\|\Ds{t} d\vec{u}\|_{L^{\frac{2d}{1+\sigma+2t}}(\R^d)}  + \max_{\tilde{s} \in [0,t-1]}\|\Ds{s} \Omega\|_{L^{\frac{d}{1+s}}(\R^d)}\, \|\Ds{t-1-s} d\vec{u}\|_{L^{\frac{2d}{1+\sigma+2(t-1-s)}}(\R^d)}\\
  \aleq&\brac{1+\|\Ds{\frac{d-2}{2}} \Omega\|_{L^{2}(\R^d)}}\|\Ds{t} d\vec{u}\|_{L^{\frac{2d}{1+\sigma+2t}}(\R^d)}  \\
  \overset{\eqref{eq:Omegacond:growth}, \eqref{eq:411Ssmallv1}}{\aleq}&\|\Ds{t} d\vec{u}\|_{L^{\frac{2d}{1+\sigma+2t}}(\R^d)}.
\end{split}
 \]

\end{proof}

The next conditions essentially state that $\curl \Omega =0$ -- up to perturbative quantities, and that $\Omega_{ij}$ is essentially collinear with $\nabla u$, in the sense that $\Omega_{ij}^\alpha \partial_\beta u =\Omega_{ij}^\beta \partial_\alpha u$ up to perturbative quantities -- a condition we call cancellation.

\begin{definition}[Admissible $\Omega$]\label{def:admissibleomega}
For a map $\vec{u}: \R^d \times \R \to \S^{2}$ we assume that $\Omega = \Omega(\vec{u})$ satisfies the following conditions:

\begin{itemize}
\item Antisymmetry: $\Omega_{ij}^\alpha = -\Omega_{ji}^\alpha$.
\item $\Omega$ satisfies the growth conditions of \Cref{def:FOmegagrowth}.
\item $\Omega$ has a gradient-structure up to perturbative quantities, in the sense we have an estimate on its curl:
\[
 \partial_a \Omega^b - \partial_a \Omega^b = G_{1;a,b}(\vec{u})
\]
and
\[
 \partial_a \Omega^0 + \partial_t \Omega^a = G_{1;a,0}(\vec{u})
\]
where $G_1(u)$ satisfies the following estimates, which all would follow if we knew that $G_1(\vec{u})\aeq g(\vec{u})|d\vec{u}|^2$:

For some $\gamma > 0$:

If $s \in [1,\frac{d-2}{2}]$, $p \in (1,\frac{2d}{s+1})$, $q \in [1,\infty]$
\begin{equation}\label{eq:G1cond:1Lpq}
 \|\Ds{s-1} G_{1;\alpha,\beta}(\vec{u})\|_{L^{(p,q)}(\R^d)} \aleq (1+\|\Ds{\frac{d-2}{2}}d\vec{u}\|_{L^2(\R^d)}^\gamma)\ \|\Ds{\frac{d-2}{2}}d\vec{u}\|_{L^2(\R^d)}\, \|\Ds{s}d\vec{u}\|_{L^{(p,q)}(\R^d)}.
\end{equation}

We also have
\begin{equation}\label{eq:G1cond:1}
 \|\Ds{\frac{d-6}{2}}dG_{1;\alpha,\beta}(\vec{u})\|_{L^{2}(\R^d)} \aleq (1+\|\Ds{\frac{d-2}{2}}d\vec{u}\|_{L^2(\R^d)}^\gamma)\ \brac{\|\Ds{\frac{d-2}{2}}d\vec{u}\|_{L^2(\R^d)}^2+\|\Ds{\frac{d-4}{2}}dd\vec{u}\|_{L^2(\R^d)}^2}.
\end{equation}
and
\begin{equation}\label{eq:G1cond:3}
 \|dG_{1;\alpha,\beta}(\vec{u})\|_{L^{(\frac{2d}{5-\sigma},q)}(\R^d)} \aleq  \|dd \vec{u}\|_{L^{\frac{d}{2}}(\R^d)}\, \|d\vec{u}\|_{L^{(2d,q)}(\R^d)}.
\end{equation}
also for any $q \in [1,\infty]$, and whenever $(1-\sigma) \in [\frac{2}{d-2},1]$,
\begin{equation}\label{eq:G1cond:3sigma}
 \|dG_{1;\alpha,\beta}(\vec{u})\|_{L^{(\frac{2d}{5-\sigma},q)}(\R^d)} \aleq
 \|\Ds{-1}dd\vec{u}\|_{L^{2d}(\R^d)}^\sigma\ \|\Ds{\frac{d-4}{2}} dd\vec{u}\|_{L^2(\R^d)}^{1-\sigma} \, \|d\vec{u}\|_{L^{(2d,q)}(\R^d)}
%
\end{equation}

\begin{equation}\label{eq:G1cond:2}
 \|G_{1;\alpha,\beta}(\vec{u})\|_{L^{(d,1)}(\R^d)} \aleq \|d\vec{u}\|_{L^{(2d,2)}(\R^d)}^2.
\end{equation}

Whenever $(1-\sigma) \in [\frac{2(s-1)}{d-2},1]$ and $s \in [1, \frac{d-2}{2}]$
\begin{equation}\label{eq:G1cond:GN1spatial}
 \|\Ds{s-1}G_{1;\alpha,\beta}(\vec{u})\|_{L^{\frac{2d}{1+2s-\sigma}}(\R^d)} \aleq \|d\vec{u}\|_{L^{2d}(\R^d)}^{2-(1-\sigma)} \ \|\Ds{\frac{d-2}{2}} d\vec{u}\|_{L^2(\R^d)}^{1-\sigma}.
\end{equation}

Also, for $s \geq 2$ and $(1-\sigma) \in [\frac{2(s-1)}{d-2},1]$
\begin{equation}\label{eq:G1cond:GN1}
 \|\Ds{s-2}dG_{1;\alpha,\beta}(\vec{u})\|_{L^{\frac{2d}{1+2s-\sigma}}(\R^d)} \aleq \|\Ds{-1}d\vec{u}\|_{L^{2d}(\R^d)}^{1+\sigma} \ \|\Ds{\frac{d-4}{2}} dd\vec{u}\|_{L^2(\R^d)}^{1-\sigma}.
\end{equation}

\item We have the cancellation condition 
\begin{equation}
 \label{eq:omeacommdu}
 \begin{cases}
\sum_{j=1}^N \brac{\Omega^a_{ij} \partial_b u^{j} - \Omega^b_{ij} \partial_a u^j} = G_{2;a,b}(\vec{u}) \quad &\text{$a,b \in \{1,\ldots,N\}$}\\
\sum_{j=1}^N \brac{\Omega^a_{ij} \partial_t u^{j} + \Omega^0_{ij} \partial_a u^j} = G_{2;a,0}(\vec{u}) \quad &\text{$a \in \{1,\ldots,N\}$}\\
 \end{cases}
\end{equation}
and for some $\gamma_1 > 0$, with perturbative estimates
\begin{equation}\label{eq:G2cond:1}
 \|\Ds{\frac{d-4}{2}}dG_{2;\alpha,\beta}(\vec{u})\|_{L^{2}(\R^d)} \aleq \|\Ds{\frac{d-2}{2}} d \vec{u}\|_{L^2(\R^d)}^{\gamma_1}\, \|d \vec{u}\|_{L^{(2d,2)}(\R^d)}^2.
\end{equation}
\end{itemize}

\end{definition}

\section{Coulomb gauge estimates}\label{s:gauge}

The following is due to Uhlenbeck \cite{Uhlenbeck1982} and Rivi\`ere \cite[Lemma A.3]{Riv2007}, see also the presentation e.g. in  \cite[Theorem 4.2.]{DLS14}; the choice of $P$ is related to the Coulomb gauge and the moving frame used for harmonic maps, \cite{S10}.

For a given $\Omega: \R^d \to so(N) \otimes \R^d$ we set
\begin{equation}\label{eq:OmegaPalpha}
 \Omega_P^\alpha := \begin{cases}
                     \partial_t P P^T+ P\Omega^0\, P^T \quad& \alpha = 0\\
                     -\partial_\alpha P\, P^T + P \Omega^\alpha P^T  \quad & \alpha \in \{1,\ldots,N\}
                    \end{cases}
\end{equation}

\begin{lemma}\label{la:uhlenbeckgauge}
There exists $\eps > 0$ and $C > 0$ depending only on the dimensions $d$, $N$ such that the following holds.

Assume $\Omega_{ij}^a \in L^d(\R^d)$, $i,j \in \{1,\ldots,N\}$ and $a \in \{1,\ldots,d\}$ with the condition
\[
 \|\Omega\|_{L^d(\R^d)} < \eps.
\]
Then there exists $P \in W^{1,d}(\R^d,SO(N))$ satisfying 
\begin{equation}\label{eq:uhlenbeckPest}
 \|\nabla P\|_{L^d(\R^d)} \leq C\, \|\Omega\|_{L^d(\R^d)}
\end{equation}
and
\[
 \sum_{a=1}^d \partial_a \Omega_P^a = 0,
\]
\end{lemma}
It is important to observe that this gives (for now) no information about $\Omega^0$ whatsoever.

In the following we will always assume that $P$ is from \Cref{la:uhlenbeckgauge} and that $\Omega$ is admissible in the sense of \Cref{def:admissibleomega}; we will derive several useful estimates.

\begin{lemma}\label{la:spatialOmegaPLdest}
Assume that $\Omega$ is admissible in the sense of \Cref{def:admissibleomega} and \eqref{eq:dduest1} holds.

There exists a uniform $\eps > 0$ such that if
\begin{equation}\label{eq:411Ssmall}
 \|\Ds{\frac{d-4}{2}}dd\vec{u}\|_{L^2(\R^d)} < \eps
\end{equation}
then we have 
\begin{equation}\label{eq:411SS3}
 \|\Ds{\frac{d-2}{2}}\Omega_P^a\|_{L^2(\R^d)} \aleq \|\Ds{\frac{d-2}{2}}d\vec{u}\|_{L^2(\R^d)}^2
\end{equation}
and
\begin{equation}\label{eq:411SS2}
 \|\Ds{\frac{d}{2}}P\|_{L^2(\R^d)}  \aleq \|\Ds{\frac{d-2}{2}}d\vec{u}\|_{L^2(\R^d)}
\end{equation}
\end{lemma}
Let us remark that the smallness assumption \eqref{eq:411Ssmall} is not needed for this specific lemma, an estimate $\aleq 1$ would be enough. But for the overall estimates of the gauge, they are needed, see \eqref{eq:smallnessneeded}.
\begin{proof}
Denote by 
\[
 \Rz^\perp := \lapms{1}\curl.
\]
Let $s \in [1,\frac{d}{2}]$. We consider by a slight abuse of notation for $\Omega$ only the spatial part, then we have
\[
\begin{split}
&\|\Ds{s-1} \brac{\nabla P^T \, P + P \Omega P^T}\|_{L^{\frac{d}{s}}(\R^d)}\\
\aleq& \|\Ds{s-1}\underbrace{\Rz \cdot (\nabla P^T \, P + P \Omega P^T)}_{=0}\|_{L^{\frac{d}{s}}(\R^d)} + \|\Ds{s-1}\Rz^\perp  (\nabla P^T \, P + P \Omega P^T)\|_{L^{\frac{d}{s}}(\R^d)}\\
 \aleq&\|\Ds{s-1} [\Rz^\perp, P_{ij}](\nabla P_{k\ell})\|_{L^{\frac{d}{s}}(\R^d)}+\|\Ds{s-1}\brac{\Rz^\perp [P_{ij},\Omega_{k\ell}] P_{mn}}\|_{L^{\frac{d}{s}}(\R^d)} + \|\Ds{s-1} \brac{P\, \Rz^\perp \Omega\, P^T}\|_{L^{\frac{d}{s}}(\R^d)}\\
\aleq& \max_{t \in [0,s-1]} \|\Ds{\frac{1}{2}+t} P\|_{L^{\frac{d}{\frac{1}{2}+t}}(\R^d)} \|\Ds{\frac{1}{2}+s-1-t} P\|_{L^{\frac{d}{s-\frac{1}{2}-t}}(\R^d)}\\
&+\max_{t_1+t_2+t_3 =s-1} \|\Ds{t_1+\frac{1}{2}} P\|_{L^{\frac{d}{t_1+\frac{1}{2}}}(\R^d)}\, \|\Ds{t_2-\frac{1}{2}} \Omega\|_{L^{\frac{d}{t_2-\frac{1}{2}}}(\R^d)}\, \|\Ds{t_3} P\|_{L^{\frac{d}{t_3}}(\R^d)}\\
&+\max_{t_1+t_2+t_3 = s-1}\, \|\Ds{t_1} P\|_{L^{\frac{d}{t_1}}(\R^d)}\, \|\Ds{t_2} \Rz^\perp \Omega\|_{L^{\frac{d}{t_2+1}}(\R^d)}\, \|\Ds{t_3} P\|_{L^{\frac{d}{t_3}}(\R^d)}\\
\aleq&\|\Ds{s-\frac{1}{2}} P\|_{L^{\frac{d}{s-\frac{1}{2}}}(\R^d)}^2\\
&+\|\Ds{s-\frac{1}{2}} P\|_{L^{\frac{d}{s-\frac{1}{2}}}(\R^d)}\, \|\Ds{\frac{d-2}{2}} \Omega\|_{L^{2}(\R^d)}\, \brac{\|P\|_{L^\infty} + \|\Ds{s-\frac{1}{2}} P\|_{L^{\frac{d}{s-\frac{1}{2}}}(\R^d)}}\\
&+\brac{\|P\|_{L^\infty} + \|\Ds{s-\frac{1}{2}} P\|_{L^{\frac{d}{s-\frac{1}{2}}}(\R^d)}}\|\Ds{\frac{d-2}{2}} \Rz^\perp \Omega\|_{L^{2}(\R^d)}\, \\
 \end{split}
\]
We observe that in view of \Cref{def:admissibleomega},
\[
 \Rz^\perp \cdot (\Omega) = \lapms{1} \curl \Omega = \lapms{1} G(\vec{u})
\]
This and \eqref{eq:Omegacond:growth} implies 
\[
\begin{split}
&\|\Ds{s-1} \brac{\nabla P^T \, P + P \Omega P^T}\|_{L^{\frac{d}{s}}(\R^d)}\\
\aleq&\|\Ds{s-\frac{1}{2}} P\|_{L^{\frac{d}{s-\frac{1}{2}}}(\R^d)}^2\\
&+\|\Ds{s-\frac{1}{2}} P\|_{L^{\frac{d}{s-\frac{1}{2}}}(\R^d)}\, \brac{(1+\|\Ds{\frac{d-2}{2}} d\vec{u}\|_{L^2(\R^d)}^\sigma)\, \|\Ds{\frac{d-2}{2}} d\vec{u}\|_{L^2(\R^d)}}\, \brac{1+ \|\Ds{s-\frac{1}{2}} P\|_{L^{\frac{d}{s-\frac{1}{2}}}(\R^d)}}\\
&+\brac{1 + \|\Ds{s-\frac{1}{2}} P\|_{L^{\frac{d}{s-\frac{1}{2}}}(\R^d)}}\|\Ds{\frac{d-4}{2}} G(\vec{u})\|_{L^{2}(\R^d)}\, \\
 \end{split}
\]
With the help of \eqref{eq:G1cond:1} and \eqref{eq:411Ssmall} we arrive at 
\begin{equation}\label{eq:Pest:21241235}
\begin{split}
&\|\Ds{s-1} \brac{\Omega_P^a}\|_{L^{\frac{d}{s}}(\R^d)}\\
\aleq&\brac{1 + \|\Ds{s-\frac{1}{2}} P\|_{L^{\frac{d}{s-\frac{1}{2}}}(\R^d)}}  \brac{\|\Ds{s-\frac{1}{2}} P\|_{L^{\frac{d}{s-\frac{1}{2}}}(\R^d)}^2+\|\Ds{\frac{d-2}{2}} d\vec{u}\|_{L^2(\R^d)}^2}\\
 \end{split}
\end{equation}
On the other hand, for any $s > 1$
\[
\begin{split}
 &\|\Ds{s}P\|_{L^{\frac{d}{s}}(\R^d)}\\
 \aeq& \|\Ds{s-1}\nabla P\|_{L^{\frac{d}{s}}(\R^d)} =\|\Ds{s-1}(\nabla P\, P^T P)\|_{L^{\frac{d}{s}}(\R^d)}\\
 \aeq& \max_{\tilde{s} \in [0,s-1]}\|\Ds{s-1-\tilde{s}} P\|_{L^{\frac{d}{s-\tilde{s}-1}}(\R^d)}\, \|\Ds{\tilde{s}}(\nabla P\, P^T)\|_{L^{\frac{d}{\tilde{s}+1}}(\R^d)} \\
  \aleq& \brac{\|P\|_{L^\infty} + \|\Ds{s-1} P\|_{L^{\frac{d}{s-1}}(\R^d)}}\, \|\Ds{s-1}(\nabla P\, P^T)\|_{L^{\frac{d}{s}}(\R^d)} \\
  \aleq& \brac{1+ \|\|\Ds{s-1} P\|_{L^{\frac{d}{s-1}}}}\, \brac{\|\Ds{s-1}\Omega_P\|_{L^{\frac{d}{s}}(\R^d)} + \brac{\|P\|_{L^\infty} +\|\Ds{s-1} P\|_{L^{\frac{d}{s-1}}(\R^d)}}^2\, \|\Ds{s-1} \Omega\|_{L^{\frac{d}{s}}(\R^d)} }\\
  \aleq& \brac{1+ \|\|\Ds{s-1} P\|_{L^{\frac{d}{s-1}}}}\, \brac{\|\Ds{s-1}\Omega_P\|_{L^{\frac{d}{s}}(\R^d)} + \brac{\|P\|_{L^\infty} +\|\Ds{s-1} P\|_{L^{\frac{d}{s-1}}(\R^d)}}^2\, \|\Ds{\frac{d-2}{2}} \Omega\|_{L^{2}(\R^d)} }\\
  \overset{\eqref{eq:Omegacond:growth},\eqref{eq:411Ssmall}}{\aleq}& \brac{1+ \|\Ds{s-1} P\|_{L^{\frac{d}{s-1}}}}^3  \brac{ \|\Ds{s-1}\Omega_P\|_{L^{\frac{d}{s}}(\R^d)} +\|\Ds{\frac{d-2}{2}} d\vec{u}\|_{L^{2}(\R^d)}} \\
 \end{split}
\]
Applying \eqref{eq:Pest:21241235} we then have 
\begin{equation}\label{eq:Pest:21241236}
\begin{split}
&\|\Ds{s}P\|_{L^{\frac{d}{s}}(\R^d)}\\
\aleq& \brac{1+ \|\Ds{s-\frac{1}{2}} P\|_{L^{\frac{d}{s-\frac{1}{2}}}(\R^d)}^4}\, \brac{\|\Ds{s-\frac{1}{2}} P\|_{L^{\frac{d}{s-\frac{1}{2}}}(\R^d)}^2+\|\Ds{\frac{d-2}{2}} d\vec{u}\|_{L^2(\R^d)}^2+\|\Ds{\frac{d-2}{2}} d\vec{u}\|_{L^2(\R^d)}}\\
 \end{split}
\end{equation}
This implies that for any $s \in [1,\frac{d-2}{2}]$,
\begin{equation}\label{eq:Pest:21241237}
 \|\Ds{s}P\|_{L^{\frac{d}{s}}(\R^d)} \aleq \|\Ds{\frac{d-2}{2}} d\vec{u}\|_{L^2(\R^d)}.
\end{equation}
Indeed, for $s =\frac{3}{2}$ this follows  from \eqref{eq:uhlenbeckPest} and \eqref{eq:Omegacond:growth} and \eqref{eq:Pest:21241236}, for $s > \frac{3}{2}$ this follows by induction and \eqref{eq:Pest:21241236}.

This establishes \eqref{eq:411SS2}. Plugging \eqref{eq:Pest:21241237} into \eqref{eq:Pest:21241235}, we also have 
\[
\|\Ds{s-1} \brac{\Omega_P^a}\|_{L^{\frac{d}{s}}(\R^d)} \aleq\underbrace{\brac{1 + \|\Ds{\frac{d-2}{2}} d\vec{u}\|_{L^2(\R^d)}}}_{\overset{\eqref{eq:411Ssmall}}{\aleq} 2} \|\Ds{\frac{d-2}{2}} d\vec{u}\|_{L^2(\R^d)}^2\\
 \]
This implies \eqref{eq:411SS3}.

\end{proof}

Similarly we also observe the following estimates
\begin{lemma}\label{la:L2destP}
Under the assumptions of \Cref{la:uhlenbeckgauge}, for $\eps$ possibly smaller, we also have
\[
 \|\nabla P\|_{L^{(p,q)}(\R^d)} \aleq \|\Omega\|_{L^{(p,q)}(\R^d)}.
\]
and
\[
\max_{a \in \{1,\ldots,d\}}\|\Omega_P^a\|_{L^{(p,q)}(\R^d)} \aleq \|\Omega\|_{L^d(\R^d)} \|\Omega\|_{L^{(p,q)}(\R^d)} +  \|\Rz^\perp \Omega\|_{L^{(p,q)}(\R^d)}.
\]
\end{lemma}
\begin{proof}
We have
\[
 \div(P\nabla P^T + P \Omega P^T)=0 \quad \text{in $\R^d$}
\]
and thus
\[
 \Rz \cdot (P \nabla P^T + P \Omega P^T) = 0.
\]
Then
\[
\begin{split}
 \|\nabla P\|_{L^{(p,q)}(\R^d)} \aleq &\|\Rz^\perp (P \nabla P^T + P \Omega P^T)\|_{L^{(p,q)}(\R^d)} + \|\Omega\|_{L^{(p,q)}(\R^d)}\\
\aleq &[P]_{BMO}\|\nabla P\|_{L^{(p,q)}(\R^d)} + (1+[P]_{BMO}) \|\Omega\|_{L^{(p,q)}(\R^d)}\\
\aleq &\eps\, \|\nabla P\|_{L^{(p,q)}(\R^d)} + \|\Omega\|_{L^{(p,q)}(\R^d)}
 \end{split}
 \]
Choosing $\eps \ll 1$ we obtain the
\[
 \|\nabla P\|_{L^{(p,q)}(\R^d)} \aleq \|\Omega\|_{L^{(p,q)}(\R^d)}.
\]
Also,
\[
\begin{split}
 \|\Omega_P\|_{L^{(p,q)}(\R^d)} \aleq& \|\Rz^\perp( P\nabla P^T)\|_{L^{(p,q)}(\R^d)} + \|\Rz^\perp (P\Omega P^T)\|_{L^{(p,q)}(\R^d)}\\
 &[P]_{BMO} \|\nabla P\|_{L^{(p,q)}(\R^d)} + [P]_{BMO} \|\Omega\|_{L^{(p,q)}(\R^d)} +\|\Rz^\perp \Omega \|_{L^{(p,q)}(\R^d)}\\
 \aleq&\|\nabla P\|_{L^d(\R^d)}^2 + \|\nabla P\|_{L^d(\R^d)}\|\Omega\|_{L^{(p,q)}(\R^d)}+\|\Rz^\perp \Omega \|_{L^{(p,q)}(\R^d)}.
 \end{split}
 \]
With the previous estimate for $\nabla P$ we conclude.
\end{proof}

Next we show that \eqref{eq:411SS3} etc. holds also for $\Omega_P^0$.

\begin{lemma}\label{la:fullOmegaPLdest}
Under the assumptions of \Cref{la:spatialOmegaPLdest}, we actually have 
\begin{equation}\label{eq:411SSfull}
 \|\Ds{\frac{d-2}{2}} \Omega_P\|_{L^2(\R^d)}  \aleq \|\Ds{\frac{d-2}{2}}d \vec{u}\|_{L^2(\R^d)}^2
\end{equation}
and 
\[
  \|\Ds{\frac{d-2}{2}} d P\|_{L^2(\R^d)}  \aleq \|\Ds{\frac{d-2}{2}} d\vec{u}\|_{L^2(\R^d)}.
\]
as well as (for any $q \in [1,\infty])$
\[
 \|dP\|_{L^{(2d,q)}(\R^d)} \aleq \|d \vec{u}\|_{L^{(2d,q)}(\R^d)}
\]
and
\[
 \|\Omega_P\|_{L^{(2d,q)}(\R^d)} \aleq \|d\vec{u}\|_{L^{d}(\R^d)}\ \|d \vec{u}\|_{L^{(2d,q)}(\R^d)}.
\]
\end{lemma}
\begin{proof}[Proof of \Cref{la:fullOmegaPLdest}]

We first observe that
\[
\begin{split}
 \|\Ds{\frac{d-2}{2}} \partial_t P\|_{L^2(\R^d)}
 \leq& \|\Ds{\frac{d-2}{2}} \brac{\Omega^0_P P}\|_{L^2(\R^d)} + \|\Ds{\frac{d-2}{2}} \brac{\Omega^0 P}\|_{L^2(\R^d)}\\
\aleq& \brac{\|\Ds{\frac{d-2}{2}} \Omega^0_P\|_{L^2(\R^d)}+\|\Ds{\frac{d-2}{2}} \Omega^0\|_{L^2(\R^d)}}\,  \brac{1+\|\Ds{\frac{d}{2}} P\|_{L^{2}(\R^d)}}
 \end{split}
\]
Thus, with \Cref{la:spatialOmegaPLdest}, and \eqref{eq:Omegacond:growth}
\begin{equation}\label{eq:Ddm22dPl2}
\begin{split}
 \|\Ds{\frac{d-2}{2}} d P\|_{L^2(\R^d)}
\aleq& \|\Ds{\frac{d-2}{2}} \Omega^0_P\|_{L^2(\R^d)}+\|\Ds{\frac{d-2}{2}} d\vec{u}\|_{L^2(\R^d)}.
 \end{split}
\end{equation}
And we recall that we already have from \Cref{la:spatialOmegaPLdest}
\begin{equation}\label{eq:Ddm22NPl2}
\begin{split}
 \|\Ds{\frac{d-2}{2}} \nabla P\|_{L^2(\R^d)}
\aleq& \|\Ds{\frac{d-2}{2}} d\vec{u}\|_{L^2(\R^d)}.
 \end{split}
\end{equation}

Next we consider the $\Omega^0_P$-estimate. We have
\[
\begin{split}
 \lap \Omega_P^0 = &\sum_{a = 1}^d \partial_{a a } \Omega_P^0\\
 =&-\sum_{a = 1}^d \underbrace{\partial_{a t} \Omega_P^a}_{=0} + \sum_{a = 1}^d\partial_a \brac{\partial_t \Omega_P^a+\partial_a \Omega_P^0 }\\
 \end{split}
\]
From the definition of $\Omega_P$, \eqref{eq:OmegaPalpha}, 
\begin{equation}\label{eq:asdcioxckj1}
\begin{split}
 &\partial_a \Omega_P^0 + \partial_t (\Omega_P^a)  \\
 =&\partial_a (\partial_t P P^T+ P\Omega^0\, P^T) + \partial_t ( -\partial_a P\, P^T + P \Omega^a P^T)\\
 =&P \brac{\partial_a \Omega^0 +\partial_t \Omega^a} P^T \\
 &+\partial_t P \partial_a P^T - \partial_a P \partial_t P^T\\
 &+ \partial_a P\,  \Omega^0 P^T\\
 &-\partial_t P\, \Omega^a P^T\\
 &+ P  \Omega^0\, \partial_a P^T\\
 &-\partial_t P\, \Omega^a P^T\\
 \end{split}
\end{equation}
Thus, we may write
\begin{equation}\label{eq:asdcioxckj12}
\begin{split}
 \Omega^0_P =& \lapms{1} \Rz_a (P \underbrace{\brac{\partial_a \Omega^0 +\partial_t \Omega^a}}_{=G_{1;a,0}(\vec{u})} P^T) \\
 &+\lapms{1} \Rz_a \brac{\partial_t P \partial_a P^T - \partial_a P \partial_t P^T}\\
 &+ \lapms{1} \Rz_a \brac{\partial_a P\,  \Omega^0 P^T}\\
 &-\lapms{1} \Rz_a \brac{\partial_t P\, \Omega^a P^T}\\
 &+ \lapms{1} \Rz_a \brac{P  \Omega^0\, \partial_a P^T}\\
 &-\lapms{1} \Rz_a \brac{\partial_t P\, \Omega^a P^T}
\end{split}  
\end{equation}
Thus
\begin{equation}\label{eq:aklsdmlkmcd2}
\begin{split}
 \|\Ds{\frac{d-2}{2}}\Omega^0_P\|_{L^{2}(\R^d)} \aleq& (1+\|\Ds{\frac{d-4}{2}} P\|_{L^{\frac{2d}{d-4}}(\R^d)}) \brac{\|G_{1;a,0}(\vec{u})\|_{L^{\frac{d}{2}}(\R^d)}+\|\Ds{\frac{d-4}{2}} G_{1;a,0}(\vec{u})\|_{L^{2}(\R^d)}}\\
 &+\|\Ds{\frac{d-4}{2}} \nabla P\|_{L^{\frac{2d}{d-2}}(\R^d)}\, \|\Ds{\frac{d-4}{2}}\partial_t P\|_{L^{\frac{2d}{d-2}}(\R^d)}\\
 &+\|\Ds{\frac{d-2}{2}} dP\|_{L^2} \|\Ds{\frac{d-2}{2}}\Omega^0\|_{L^{2}} \brac{1+\|\Ds{\frac{d}{2]}} P\|_{L^{2}}}\\
\end{split}
 \end{equation}%
So in particular we obtain with the help of \eqref{eq:Ddm22NPl2}, \eqref{eq:G1cond:1}, \eqref{eq:411Ssmall},\eqref{eq:Omegacond:growth}
\[
\|\Ds{\frac{d-2}{2}}\Omega^0_P\|_{L^{2}(\R^d)}  \aleq \|\Ds{\frac{d-2}{2}}d\vec{u}\|_{L^2(\R^d)}^2 +\eps \|\Ds{\frac{d-2}{2}} dP\|_{L^2} \\
\]
which together with \eqref{eq:Ddm22dPl2}, for $\eps \ll 1$, implies
\begin{equation}\label{eq:smallnessneeded}
 \|\Ds{\frac{d-2}{2}} dP\|_{L^2} \aleq \|\Ds{\frac{d-2}{2}}d\vec{u}\|_{L^2(\R^d)}.
\end{equation}
Plugging this again into \eqref{eq:aklsdmlkmcd2}, again with the help of \eqref{eq:Ddm22NPl2}, \eqref{eq:G1cond:1}, \eqref{eq:411Ssmall},\eqref{eq:Omegacond:growth}, we have
\[
 \|\Ds{\frac{d-2}{2}}\Omega^0_P\|_{L^{2}(\R^d)} \aleq \|\Ds{\frac{d-2}{2}}d\vec{u}\|_{L^2(\R^d)}^2.
 \]

 Similarly,
 \[
 \begin{split}
  \|\Omega^0_P\|_{L^{(2d,q)}(\R^d)} \aleq & \|G_{1;a,0}(\vec{u})\|_{L^{(\frac{2d}{3},q)}(\R^d)} \\
 &+\|\partial_t P\|_{L^{(2d,q)}(\R^d)} \|\nabla P\|_{L^d(\R^d)}\\
 &+ \|\Omega^0\|_{L^{(2d,q)}(\R^d)} \|\nabla P\|_{L^d(\R^d)}\\
 &+\|\partial_t P\|_{L^{(2d,q)}(\R^d)}\ \| \Omega^a \|_{L^d(\R^d)}\\
\overset{\eqref{eq:G1cond:3}}{\aleq}& \|d \vec{u}\|_{L^d(\R^d)}\, \|d \vec{u}\|_{L^{(2d,q)}(\R^d)}\\
&+\|\partial_t P\|_{L^{(2d,q)}(\R^d)} \|d\vec{u}\|_{L^d(\R^d)}
 \end{split}
 \]
 On the other hand,
 \[
  \|\partial_t P\|_{L^{(2d,q)}(\R^d)} \aleq \|\Omega^0_P\|_{L^{(2d,q)}(\R^d)} + \|d \vec{u}\|_{L^{(2d,q)}(\R^d)}
 \]
we conclude that by the smallness assumption \eqref{eq:411Ssmall}
\[
 \|\partial_t P\|_{L^{(2d,q)}(\R^d)} \aleq \|d \vec{u}\|_{L^{(2d,q)}(\R^d)}
\]
and thus also
\[
 \|\Omega^0_P\|_{L^{(2d,q)}(\R^d)} \aleq \|d\vec{u}\|_{L^{d}(\R^d)}\ \|d \vec{u}\|_{L^{(2d,q)}(\R^d)}.
\]

\end{proof}

\begin{lemma}\label{la:fullOmegaPLdestdd}
Under the assumptions of \Cref{la:spatialOmegaPLdest}, we actually have
\begin{equation}\label{eq:411SSfulldd}
 \|\Ds{\frac{d-4}{2}} d\Omega_P\|_{L^2(\R^d)}  \aleq \|\Ds{\frac{d-2}{2}}d \vec{u}\|_{L^2(\R^d)}^2
\end{equation}
and
\begin{equation}\label{eq:411SSfullddP}
  \|\Ds{\frac{d-4}{2}} d dP\|_{L^2(\R^d)}  \aleq \|\Ds{\frac{d-2}{2}} d\vec{u}\|_{L^2(\R^d)}.
\end{equation}

\end{lemma}
\begin{proof}
We observe that
\[
 \|\Ds{\frac{d-4}{2}} d dP\|_{L^2(\R^d)} \aleq \|\Ds{\frac{d-2}{2}} d P\|_{L^2(\R^d)} + \|\Ds{\frac{d-4}{2}} \partial_{tt}P\|_{L^2(\R^d)}
\]
The first term is already estimated in \Cref{la:fullOmegaPLdest}. We then observe
\[
\begin{split}
 &\|\Ds{\frac{d-4}{2}} \partial_{tt}P\|_{L^2(\R^d)}\\
 =&\|\Ds{\frac{d-4}{2}} \brac{\partial_{tt}P\, P^T P}\|_{L^2(\R^d)}\\
 =&\|\Ds{\frac{d-4}{2}} \brac{\partial_t\brac{\partial_{t}P\, P^T} P}\|_{L^2(\R^d)}+\|\Ds{\frac{d-4}{2}} \brac{\partial_{t}P\, \partial_t P^T\, P}\|_{L^2(\R^d)}\\
 =&\|\Ds{\frac{d-4}{2}} \brac{\partial_t\Omega_P^0\ P}\|_{L^2(\R^d)}+\|\Ds{\frac{d-4}{2}} \brac{\partial_t(P^t \Omega^0 P)\ P}\|_{L^2(\R^d)}+\|\Ds{\frac{d-4}{2}} \brac{\partial_{t}P\, \partial_t P^T\, P}\|_{L^2(\R^d)}\\
 \aleq& (1+\|\Ds{\frac{d}{2}} P\|_{L^2(\R^d)})\, \|\Ds{\frac{d-4}{2}} \brac{\partial_t\Omega_P^0}\|_{L^2(\R^d)}\\
 &+\brac{1+\|\Ds{\frac{d-2}{2}} dP\|_{L^2(\R^d)}+ \|\Ds{\frac{d-2}{2}} dP\|_{L^2(\R^d)}^2+\|\Ds{\frac{d-2}{2}} dP\|_{L^2(\R^d)}^3}\, \|\Ds{\frac{d-4}{2}} d \Omega^0\|_{L^2(\R^d)}\\
 &+(1+\|\Ds{\frac{d}{2}} P\|_{L^2(\R^d)})\|\Ds{\frac{d-2}{2}} dP\|_{L^2(\R^d)}^2\\
 \end{split}
\]
In view of \Cref{la:fullOmegaPLdest}, \eqref{eq:Omegacond:growth}, and the smallness assumption \eqref{eq:411Ssmall}, we obtain
\[
 \|\Ds{\frac{d-4}{2}} d dP\|_{L^2(\R^d)} \aleq \|\Ds{\frac{d-2}{2}}d\vec{u}\|_{L^2(\R^d)}+\|\Ds{\frac{d-4}{2}} d\Omega_P\|_{L^2(\R^d)}.
\]
So it remains to estimate $d \Omega_P$. For $a \in \{1,\ldots,d\}$ we use \eqref{eq:asdcioxckj1}
\[
 \begin{split}
 \partial_t \Omega^a_P =& \partial_a \Omega^P\\
 &+P \brac{G_{1;a,0}(\vec{u})} P^T \\
 &+\partial_t P \partial_a P^T - \partial_a P \partial_t P^T\\
 &+ \partial_a P\,  \Omega^0 P^T\\
 &-\partial_t P\, \Omega^a P^T\\
 &+ P  \Omega^0\, \partial_a P^T\\
 &-\partial_t P\, \Omega^a P^T\\
 \end{split}
\]
and thus with \Cref{la:fullOmegaPLdest}, \eqref{eq:G1cond:1}, \eqref{eq:Omegacond:growth}, and the smallness assumption \eqref{eq:411Ssmall},
\[
 \max_{a \in \{1,\ldots,d\}}\|\Ds{\frac{d-4}{2}} d \Omega^a_P\|_{L^2(\R^d))} \aleq \|\Ds{\frac{d-2}{2}}d\vec{u}\|_{L^2(\R^d)}^2.
\]
So what we actually only need to estimate is $\partial_t \Omega^0_P$, and we observe again from \eqref{eq:asdcioxckj1}
\begin{equation}\label{eq:ptomega0pest}
\begin{split}
 \Ds{\frac{d-4}{2}} \partial_t\Omega^0_P =& \Ds{\frac{d-6}{2}}  \Rz_a \partial_t (P \underbrace{\brac{\partial_a \Omega^0 +\partial_t \Omega^a}}_{=G_{1;a,0}(\vec{u})} P^T) \\
 &+\Ds{\frac{d-6}{2}}  \Rz_a \partial_t \brac{\partial_t P \partial_a P^T - \partial_a P \partial_t P^T}\\
 &+ \Ds{\frac{d-6}{2}}  \Rz_a \partial_t\brac{\partial_a P\,  \Omega^0 P^T}\\
 &-\Ds{\frac{d-6}{2}}  \Rz_a \partial_t \brac{\partial_t P\, \Omega^a P^T}\\
 &+ \Ds{\frac{d-6}{2}}  \Rz_a \partial_t \brac{P  \Omega^0\, \partial_a P^T}\\
 &-\Ds{\frac{d-6}{2}}  \Rz_a \partial_t \brac{\partial_t P\, \Omega^a P^T}\\
 \overset{\eqref{eq:OmegaPalpha}}{=}& \Ds{\frac{d-6}{2}}  \Rz_a \partial_t (P\, G_{1;a,0}(\vec{u})\, P^T) \\
 &+\Ds{\frac{d-6}{2}}  \Rz_a \partial_t \brac{\Omega_P^0 \partial_a P^T - \partial_a P\, P^T  \Omega_P^0}\\
 &+\Ds{\frac{d-6}{2}}  \Rz_a \partial_t \brac{-P \Omega^0\, \partial_a P^T + \partial_a P\, \Omega^0 P^T}\\
 &+ \Ds{\frac{d-6}{2}}  \Rz_a \partial_t\brac{\partial_a P\,  \Omega^0 P^T}\\
 &+\Ds{\frac{d-6}{2}}  \Rz_a \partial_t \brac{P\Omega^0  \Omega^a P^T}\\
 &-\Ds{\frac{d-6}{2}}  \Rz_a \partial_t \brac{\Omega^0_P P \Omega^a P^T}\\
 &+ \Ds{\frac{d-6}{2}}  \Rz_a \partial_t \brac{P  \Omega^0\, \partial_a P^T}\\
 &+\Ds{\frac{d-6}{2}}  \Rz_a \partial_t \brac{\Omega^0\, \Omega^a P^T}\\
&-\Ds{\frac{d-6}{2}}  \Rz_a \partial_t \brac{\Omega^0_P P \Omega^a P^T}
\end{split}
\end{equation}
Then we have \Cref{la:fullOmegaPLdest}, \eqref{eq:G1cond:1}, \eqref{eq:Omegacond:growth}, and the smallness assumption \eqref{eq:411Ssmall},
\[
\begin{split}
 &\|\Ds{\frac{d-4}{2}}\partial_t\Omega^0_P \|_{L^2(\R^d)} \\
 \aleq &\brac{1+\|\Ds{\frac{d-2}{2}}dP\|_{L^2(\R^d)}+\|\Ds{\frac{d-2}{2}}dP\|_{L^2(\R^d)}^2} \|\Ds{\frac{d-2}{2}} d\vec{u}\|_{L^2(\R^d)}^2\\
 &+ \brac{1+\|\Ds{\frac{d-2}{2}}dP\|_{L^2(\R^d)}+\|\Ds{\frac{d-2}{2}}dP\|_{L^2(\R^d)}^2}\, \brac{\|\Ds{\frac{d-2}{2}}\Omega \|_{L^2(\R^d)}+\|\Ds{\frac{d-2}{2}}dP\|_{L^2(\R^d)}}\, \|\Ds{\frac{d-4}{2}}\partial_t  \Omega^0_P\|_{L^2(\R^d)}\\
 \aleq& \|\Ds{\frac{d-2}{2}} d\vec{u}\|_{L^2(\R^d)}^2\\
 &+ \eps \|\Ds{\frac{d-4}{2}}\partial_t  \Omega^0_P\|_{L^2(\R^d)}\\
 \end{split}
\]
For $\eps$ sufficiently small we can absorb the last term and obtain the desired result, \eqref{eq:411SSfulldd} and \eqref{eq:411SSfullddP}.
\end{proof}

\begin{lemma}
Let $s \in [0,\frac{d-{2}}{2}{]}$, $p \in (\frac{d}{d+s-1} ,\frac{d}{s})$ then, assuming \eqref{eq:411Ssmall} is satisfied,
\begin{equation}\label{eq:411SSfullLqdd1}
 \|\Ds{s}\Omega_P\|_{L^p(\R^d)} \aleq \|\Ds{\frac{d}{2}} \vec{u}\|_{L^2(\R^d)}\, \|\Ds{s} d\vec{u} \|_{L^p(\R^d)}
\end{equation}
and if $s \geq 1$ and $p > \frac{d}{d+s-2}$
\begin{equation}\label{eq:411SSfullLqdd}
 \|\Ds{s-1}d\Omega_P\|_{L^p(\R^d)} \aleq \|\Ds{\frac{d}{2}} \vec{u}\|_{L^2(\R^d)}\, \|\Ds{s} d\vec{u} \|_{L^p(\R^d)}.
\end{equation}
Also
\begin{equation}\label{eq:411SSfullLqddP1}
 \|\Ds{s}dP\|_{L^p(\R^d)} \aleq \|\Ds{s} d\vec{u} \|_{L^p(\R^d)}.
\end{equation}
and if $s \geq 1$
\begin{equation}\label{eq:411SSfullLqddP}
 \|\Ds{s-1}ddP\|_{L^p(\R^d)} \aleq \|\Ds{s} d\vec{u} \|_{L^p(\R^d)}.
\end{equation}
We also have  as a variation of \eqref{eq:411SSfullLqdd}, cf. \eqref{eq:G1cond:3},
\begin{equation}\label{eq:411SSfullG1cond:3:OmegaP}
 \|d \Omega_P\|_{L^{(\frac{2d}{3},2)}(\R^d)} \aleq \|d\vec{u}\|_{L^{(2d,2)}(\R^d)}\, \|\Ds{\frac{d-2}{2}} d\vec{u}\|_{L^{2}(\R^d)}
\end{equation}

\end{lemma}
\begin{proof}
We begin again with the \underline{spatial part} $\Omega_P' = (\Omega_P^a)_{a=1\ldots,d}$ and then observe that since $\div(\Omega_P') = 0$,
\[
\begin{split}
 \|\Ds{s} \Omega_P'\|_{L^p(\R^d)} \aleq &\|\Ds{s} \Rz^\perp \brac{\Omega_P'}\|_{L^p(\R^d)}\\
  \aleq& \|\Ds{s} [\Rz^\perp,P](\nabla P^T)\|_{L^p(\R^d)}+\|\Ds{s+1} \Rz^\perp \brac{P \Omega' P^T}\|_{L^p(\R^d)}\\
  \aleq& \brac{\|\Ds{\frac{d-2}{2}} dP\|_{L^2(\R^d)}+\|\Ds{\frac{d-2}{2}} dP\|_{L^2(\R^d)}^2}\, \brac{\| \Ds{s+1}P \|_{L^{p}(\R^d)}+\|\Ds{s}\Omega'\|_{L^p(\R^d)}}
\\
&+ \|\Ds{s-1}\curl(\Omega')\|_{L^p(\R^d)}
 \end{split}
\]
With \Cref{la:fullOmegaPLdest}, \eqref{eq:Omegacond:growthpq}, \eqref{eq:G1cond:1Lpq}, and \eqref{eq:411Ssmall}, we find for any $s \in (0,\frac{d}{2})$,
\begin{equation}\label{eq:Lpqest:1}
\begin{split}
 \|\Ds{s} \Omega_P'\|_{L^p(\R^d)}   \aleq& \|\Ds{\frac{d-2}{2}} d\vec{u}\|_{L^2(\R^d)}\, \brac{\| \Ds{s+1}P \|_{L^{p}(\R^d)}+\|\Ds{s}d\vec{u}\|_{L^p(\R^d)}}
 \end{split}
\end{equation}
Moreover,
\[
 \nabla P = \Omega_P'\, P - P \Omega'
\]
so that again with the same estimates as above, whenever $p < \frac{d}{s}$,
\begin{equation}\label{eq:Lpqest:2}
\begin{split}
 \|\Ds{s+1} P\|_{L^p(\R^d)} \aleq &\max_{t \in [0,s]}  \|\Ds{t} \Omega_P'\|_{L^{\frac{pd}{d-(s-t)p}}(\R^d)}\, \|\Ds{s-t} P\|_{L^{\frac{d}{s-t}}(\R^d)}\\
 &+\max_{t \in [0,s]} \in \|\Ds{t} \Omega'\|_{L^{\frac{pd}{d-(s-t)p}}(\R^d)}\, \|\Ds{s-t} P\|_{L^{\frac{d}{s-t}}(\R^d)}\\
 \aleq&(1+\|\Ds{\frac{d-2}{2}} dP\|_{L^{2}(\R^d)}) \brac{\|\Ds{s} \Omega_P'\|_{L^p(\R^d)}+\|\Ds{s} d\vec{u}\|_{L^{\frac{p}{p}}(\R^d)}}
 \end{split}
\end{equation}
Combining \eqref{eq:Lpqest:1} and \eqref{eq:Lpqest:2}, together with \Cref{la:fullOmegaPLdest} and \eqref{eq:411Ssmall} we readily obtain \eqref{eq:411SSfullLqdd} and the spatial version of \eqref{eq:411SSfullLqddP1}:
\begin{equation}\label{eq:Lpqest:spatialP}
\begin{split}
 \|\Ds{s+1} P\|_{L^p(\R^d)} \aleq& \|\Ds{s} d\vec{u}\|_{L^p(\R^d)}\\
\max_{a \in \{1,\ldots,d\}} \|\Ds{s} \Omega_P^a\|_{L^p(\R^d)} \aleq& \|\Ds{\frac{d-2}{2}} d\vec{u}\|_{L^2(\R^d)}\, \|\Ds{s} d\vec{u}\|_{L^p(\R^d)}
 \end{split}
\end{equation}

With the help of \eqref{eq:asdcioxckj12} we obtain under the assumption that $p> \frac{d}{d-(1-s)}$, for $\bar{s} := \max\{s-1,0\}$ and $\underline{s} := -\min\{s-1,0\}$ (observe that we can choose the spatial objects for the $L^p$-norm and can apply \eqref{eq:Lpqest:spatialP})
\[
\begin{split}
 &\|\Ds{s}\Omega^0_P\|_{L^p(\R^d)}\\
 \aleq& \|\Ds{\bar{s}} \brac{P G_{1;a,0}(\vec{u}) P^T}\|_{L^{\frac{dp}{d+\underline{s} p}}(\R^d)}\\
 &+\max_{a \in \{1,\ldots,d\}}\|\Ds{\bar{s}}\brac{\partial_t P \partial_a P^T}\|_{L^{\frac{dp}{d+\underline{s} p}}(\R^d)}\\
 &+\max_{a \in \{1,\ldots,d\}}\|\Ds{\bar{s}}\brac{\partial_t P\, \Omega^a P^T}\|_{L^{\frac{dp}{d+\underline{s} p}}(\R^d)}\\
 &+\max_{a \in \{1,\ldots,d\}}\|\Ds{\bar{s}}\brac{\partial_a P\,  \Omega^0 P^T}\|_{L^{\frac{dp}{d+\underline{s} p}}(\R^d)}\\
 &+\max_{a \in \{1,\ldots,d\}}\|\Ds{\bar{s}}\brac{P  \Omega^0\, \partial_a P^T}\|_{L^{\frac{dp}{d+\underline{s} p}}(\R^d)}\\
 &+\max_{a \in \{1,\ldots,d\}}\|\Ds{\bar{s}}\brac{\partial_t P\, \Omega^a P^T}\|_{L^{\frac{dp}{d+\underline{s} p}}(\R^d)}\\
 \aleq&\brac{\|\Ds{\frac{d-2}{2}} dP\|_{L^2(\R^d)} + \|\Ds{\frac{d-2}{2}} d\vec{u}\|_{L^2(\R^d)}+\|\Ds{\frac{d-2}{2}} \Omega\|_{L^{2}(\R^d)}} \brac{\|\Ds{s}d\vec{u}\|_{L^p(\R^d)}+\|\Ds{s+1}P\|_{L^p(\R^d)}+\|\Omega'\|_{L^p(\R^d)}}\\
 \aleq&\|\Ds{\frac{d-2}{2}} d\vec{u}\|_{L^2(\R^d)} \|\Ds{s}d\vec{u}\|_{L^p(\R^d)}.
 \end{split}
\]
Since
\[
 \partial_t P = \Omega_P^0\, P + P \Omega^0,
\]
we can argue as before, and we have established \eqref{eq:411SSfullLqdd1} and \eqref{eq:411SSfullLqddP1}.

We continue by using \eqref{eq:ptomega0pest}. Let $\bar{s} := \max\{s-2,0\}$ and $\underline{s} := -\min\{s-2,0\}$ (observe that $\bar{s} \leq s-1$ since $s \geq 1$ by assumption), then we have for $p > \frac{d}{d+s-2}$, again by choosing the $L^p$-norms carefully,
\begin{equation}\label{eq:alkcmlzcipvxj1}
\begin{split}
&\|\Ds{s-1}\partial_t\Omega^0_P\|_{L^p(\R^d)}\\
\aleq& \|\Ds{\bar{s}}\partial_t (P\, G_{1;a,0}(\vec{u})\, P^T)\|_{L^{\frac{dp}{d+\underline{s}p}}(\R^d)} \\
 &+\|\Ds{\bar{s}}\partial_t \brac{\Omega_P^0 \partial_a P^T - \partial_a P\, P^T  \Omega_P^0}\|_{L^{\frac{dp}{d+\underline{s}p}}(\R^d)}\\
 &+\|\Ds{\bar{s}}\partial_t \brac{-P \Omega^0\, \partial_a P^T + \partial_a P\, \Omega^0 P^T}\|_{L^{\frac{dp}{d+\underline{s}p}}(\R^d)}\\
 &+ \|\Ds{\bar{s}}\partial_t\brac{\partial_a P\,  \Omega^0 P^T}\|_{L^{\frac{dp}{d+\underline{s}p}}(\R^d)}\\
 &+\|\Ds{\bar{s}}\partial_t \brac{P\Omega^0  \Omega^a P^T}\|_{L^{\frac{dp}{d+\underline{s}p}}(\R^d)}\\
 &+\|\Ds{\bar{s}}\partial_t \brac{\Omega^0_P P \Omega^a P^T}\|_{L^{\frac{dp}{d+\underline{s}p}}(\R^d)}\\
 &+ \|\Ds{\bar{s}}\partial_t \brac{P  \Omega^0\, \partial_a P^T}\|_{L^{\frac{dp}{d+\underline{s}p}}(\R^d)}\\
 &+\|\Ds{\bar{s}}\partial_t \brac{\Omega^0\, \Omega^a P^T}\|_{L^{\frac{dp}{d+\underline{s}p}}(\R^d)}\\
&+\|\Ds{\bar{s}}\partial_t \brac{\Omega^0_P P \Omega^a P^T}\|_{L^{\frac{dp}{d+\underline{s}p}}(\R^d)}\\
\aleq&\|\Ds{\frac{d-2}{2}}dP\|_{L^2(\R^d)}\|G_{1;a,0}(\vec{u})\|_{L^{p}(\R^d)} +  \|\partial_t G_{1;a,0}(\vec{u})\|_{L^{\frac{dp}{d+\underline{s}p}}(\R^d)}\\
&+\brac{\|\Ds{\frac{d-2}{2}}dP\|_{L^2(\R^d)}+ \|\Ds{\frac{d-4}{2}}d\Omega^0_P\|_{L^2(\R^d)}+\|\Ds{\frac{d-4}{2}}d\Omega\|_{L^2(\R^d)}}\cdot \\
&\quad \, \cdot \brac{\|\Ds{s}\Omega^0_P\|_{L^p(\R^d)}+\|\Ds{s+1}P\|_{L^p(\R^d)}}\\
\overset{\eqref{eq:411SSfullLqdd1},\eqref{eq:411SSfullLqddP1}}{\aleq}& \|\Ds{\frac{d-2}{2}} d\vec{u}\|_{L^2(\R^d)} \|\Ds{s}d\vec{u}\|_{L^p(\R^d)}
\end{split}
\end{equation}
Observing \eqref{eq:asdcioxckj12} and choosing again wisely the $L^p$-norms and \eqref{eq:411SSfullddP} we conclude
\[
 \|\Ds{s-1}\partial_t\Omega_P\|_{L^p(\R^d)}\aleq \|\Ds{s-1}\partial_t\Omega_P^0\|_{L^p(\R^d)}+\|\Ds{\frac{d-2}{2}} d\vec{u}\|_{L^2(\R^d)} \|\Ds{s}d\vec{u}\|_{L^p(\R^d)}.
\]
This establishes \eqref{eq:411SSfullLqdd}.

For \eqref{eq:411SSfullLqddP} we observe again that
\[
\begin{split}
 \partial_{tt} P =& \partial_t (\partial_t P\, P^T) P - (\partial_t P\, \partial_t P^T) P\\
  =& \partial_t (\Omega^0_P) P + \partial_t (P\Omega^0 P^T)\, P- (\partial_t P\, \partial_t P^T) P\\
 \end{split}
\]
Then we have similar to the arguments above, making also full use of \eqref{eq:411SSfullLqddP1},
\[
 \|\Ds{s-1}ddP\|_{L^p(\R^d)} \overset{\eqref{eq:411SSfullLqddP1}}{\aleq }  \|\Ds{s}d\vec{u}\|_{L^p(\R^d)}.
\]

Lastly, for \eqref{eq:411SSfullG1cond:3:OmegaP}, we slightly adapt the arguments e.g. as in \eqref{eq:alkcmlzcipvxj1}. By collecting only $dP$ and $\Omega_P$-terms (without additional derivatives) in $L^{(2d,2)}$-norms, we find
\[
\begin{split}
&\|\partial_t\Omega^0_P\|_{L^{(\frac{2d}{3},2)}(\R^d)}\\
\aleq& \|\partial_t (P\, G_{1;a,0}(\vec{u})\, P^T)\|_{L^{(\frac{2d}{5},2)}(\R^d)} \\
 &+\|\partial_t \brac{\Omega_P^0 \partial_a P^T - \partial_a P\, P^T  \Omega_P^0}\|_{L^{(\frac{2d}{5},2)}(\R^d)}\\
 &+\|\partial_t \brac{-P \Omega^0\, \partial_a P^T + \partial_a P\, \Omega^0 P^T}\|_{L^{(\frac{2d}{5},2)}(\R^d)}\\
 &+ \|\partial_t\brac{\partial_a P\,  \Omega^0 P^T}\|_{L^{(\frac{2d}{5},2)}(\R^d)}\\
 &+\|\partial_t \brac{P\Omega^0  \Omega^a P^T}\|_{L^{(\frac{2d}{5},2)}(\R^d)}\\
 &+\|\partial_t \brac{\Omega^0_P P \Omega^a P^T}\|_{L^{(\frac{2d}{5},2)}(\R^d)}\\
 &+ \|\partial_t \brac{P  \Omega^0\, \partial_a P^T}\|_{L^{(\frac{2d}{5},2)}(\R^d)}\\
 &+\|\partial_t \brac{\Omega^0\, \Omega^a P^T}\|_{L^{(\frac{2d}{5},2)}(\R^d)}\\
&+\|\partial_t \brac{\Omega^0_P P \Omega^a P^T}\|_{L^{(\frac{2d}{5},2)}(\R^d)}\\
\aleq&(1+\|\partial_t P\|_{L^{d}(\R^d)})\, \|dG_{1;a,0}(\vec{u})\|_{L^{(\frac{2d}{5},2)}(\R^d)} \\
&+ \brac{\|d\Omega_P\|_{L^{\frac{d}{2}}(\R^d)} + \|d\Omega\|_{L^{\frac{d}{2}}(\R^d)}+ \|d\Dso P\|_{L^{\frac{d}{2}}(\R^d)} }\brac{\|dP\|_{L^{(2d,2)}} +\|\Omega\|_{L^{(2d,2)}(\R^d)}}\\
\aleq&\|d\vec{u}\|_{L^{2d}(\R^d)}\, \|\Ds{\frac{d-2}{2}} d\vec{u}\|_{L^{2}(\R^d)}.
\end{split}
\]
In the second to last inequality we used \eqref{eq:411Ssmall} to avoid writing squares of $L^d$-norms.
In the last inequality we used \eqref{eq:411SSfullLqddP1},\eqref{eq:411Ssmall},\eqref{eq:G1cond:3},\eqref{eq:Omegacond:growth},\eqref{eq:Omegacond:growth2d}.

The other terms of $d\Omega_P$ are similar.

We can conclude.
\end{proof}

\begin{lemma}[$L^\infty$-estimates]\label{la:omegaPLinfty}
Under the assumptions of \Cref{la:spatialOmegaPLdest} we have
\[
 \|\Omega_P\|_{L^\infty} \aleq \|d\vec{u}\|_{L^{(2d,2)}(\R^d)}^2.
\]
\end{lemma}
\begin{proof}

 We have for $b \in \{1,\ldots,n\}$
 \[
 \begin{split}
 \lap \Omega_P^b = &\sum_{a = 1}^d \partial_{a a } \Omega_P^b\\
 =&\sum_{a = 1}^d \underbrace{\partial_{a b} \Omega_P^a}_{=0} - \sum_{a = 1}^d\partial_a \brac{\partial_b \Omega_P^a-\partial_a \Omega_P^b }\\
 \end{split}
 \]

and thus
\[
  \Dso \Omega_P^b = c\sum_{a = 1}^d\Rz_a\brac{\partial_b \Omega_P^a-\partial_a \Omega_P^b }
 \]
Similar to \eqref{eq:asdcioxckj1}, using \eqref{eq:OmegaPalpha}
 \begin{equation}\label{eq:asdcioxckj2}
\begin{split}
 &\partial_a \Omega_P^b {-} \partial_b (\Omega_P^b)  \\
 =&\partial_a (-\partial_b P P^T+ P\Omega^b\, P^T) - \partial_b ( -\partial_a P\, P^T + P \Omega^a P^T)\\
 =&P \brac{\partial_a \Omega^b {-}\partial_b \Omega^a} P^T \\
 &-\partial_b P \partial_a P^T + \partial_a P \partial_b P^T\\
 &+ \partial_a P\,  \Omega^b P^T-\partial_b P\, \Omega^a P^T\\
 &+ P  \Omega^b\, \partial_a P^T-\partial_b P\, \Omega^a P^T\\
 \end{split}
\end{equation}
Consequently, using also \eqref{eq:Omegacond:growth2d},
\[
\begin{split}
 \|\Dso \Omega_P^b\|_{L^{(d,1)}(\R^d)} \aleq &\|\partial_b \Omega_P^a-\partial_a \Omega_P^b \|_{L^{(d,1)}(\R^d)}\\
&+\||DP|^2\|_{L^{(d,1)}(\R^d)} + \||DP| |\Omega|\|_{L^{(d,1)}(\R^d)}  \\
\overset{\eqref{eq:Omegacond:growth2d}}{\aleq}& \|\Dso P\|_{L^{(2d,2)}(\R^d)}^2 + \|d\vec{u}\|_{L^{(2d,2)}(\R^d)}^2\\
\overset{L.~\ref{la:L2destP},\eqref{eq:Omegacond:growth2d},\eqref{eq:411Ssmall}}{\aleq}&\|\Omega\|_{L^{(2d,2)}(\R^d)}^2 + \|d\vec{u}\|_{L^{(2d,2)}(\R^d)}^2.
  \end{split}
 \]
By Sobolev embedding $W^{1,(d,1)}(\R^d) \subset C^0(\R^d)$,
\[
 \max_{b \in \{1,\ldots,d\}} \|\Omega_P^b\|_{L^\infty(\R^d)} \aleq \|d\vec{u}\|_{L^{(2d,2)}(\R^d)}^2.
\]

We also have
\[
\begin{split}
 \|\Omega_P^a\|_{L^{2d,2}(\R^d)} \aleq & \|\Dso \Omega_P^a\|_{L^{\frac{2d}{3},2}(\R^d)}
\end{split}
 \]

\underline{Time estimates}: Again we write
 \[
 \begin{split}
  \Dso \Omega_P^0 = &-c\sum_{a = 1}^d\Rz_a\brac{\partial_0 \Omega_P^a+\partial_a \Omega_P^0 }\\
  \overset{\eqref{eq:asdcioxckj1}}{=}  &-c\sum_{a = 1}^d\Rz_a\brac{P \brac{G_{1;a,0}(\vec{u})} P^T} \\
 &-c\sum_{a = 1}^d\Rz_a\brac{\partial_t P \partial_a P^T - \partial_a P \partial_t P^T}\\
 &-c\sum_{a = 1}^d\Rz_a\brac{ \partial_a P\,  \Omega^0 P^T}\\
 &-c\sum_{a = 1}^d\Rz_a\brac{-\partial_t P\, \Omega^a P^T}\\
 & -c\sum_{a = 1}^d\Rz_a\brac{P  \Omega^0\, \partial_a P^T}\\
 &-c\sum_{a = 1}^d\Rz_a\brac{-\partial_t P\, \Omega^a P^T}\\
 \end{split}
 \]
and thus
\[
\begin{split}
 \|\Dso \Omega_P^0 \|_{L^{(d,1)}(\R^d)} \aleq& \|G_{1;a,0}(\vec{u})\|_{L^{d,1}(\R^d)}\\
 &+ \|dP\|_{L^{(2d,2)}(\R^d)}\, \|\nabla P\|_{L^{(2d,2)}(\R^d)}\\
 &+\|dP\|_{L^{(2d,2)}(\R^d)}\, \|\Omega\|_{L^{2d,2}(\R^d)}.
\end{split}
 \]
 With \eqref{eq:G1cond:2}, \Cref{la:fullOmegaPLdest}, \eqref{eq:Omegacond:growth2d}, \eqref{eq:411Ssmall},
 \[
\begin{split}
 \|\Dso \Omega_P^0 \|_{L^{(d,1)}(\R^d)} \aleq \|d\vec{u}\|_{L^{(2d,2)}(\R^d)}^2.
\end{split}
 \]
We conclude
\[
 \|\Omega_P\|_{L^{\infty}(\R^d)} \aleq \|d\vec{u}\|_{L^{(2d,2)}(\R^d)}^2.
\]
\end{proof}

\begin{lemma}
Let $d \geq 4$. For $\sigma \in [0,1]$ such that $(1-\sigma) \in {(}\frac{2(s_1-1)}{d-2},1]$ and $s_1 \in [0, \frac{d-2}{2}{]}$ we have
\begin{equation}\label{eq:OmegaPest:3542}
 \|\Ds{s_1} \Omega_P\|_{L^{\frac{2d}{1+2s_1  -\sigma}}} \aleq\|d \vec{u}\|_{L^{2d}(\R^d)}^{2-(1-\sigma)} \ \|\Ds{\frac{d-2}{2}} d\vec{u}\|_{L^2(\R^d)}^{1-\sigma}
\end{equation}
If additionally $s_1 \geq 1$ and assuming additionally that ${1-\sigma > \frac{2}{d-2}}$ (in particular: assume $d \geq 5$)
\begin{equation}\label{eq:OmegaPest:3542d}
 \|\Ds{s_1-1} d\Omega_P\|_{L^{\frac{2d}{1+2s_1  -\sigma}}} \aleq\|d \vec{u}\|_{L^{2d}(\R^d)}^{2-(1-\sigma)} \ \|\Ds{\frac{d-2}{2}} d\vec{u}\|_{L^2(\R^d)}^{1-\sigma}
\end{equation}

Note that both estimates have a replacement that holds for $s_1 = \frac{d-2}{2}$ and $\sigma = 0$ in \Cref{la:fullOmegaPLdestdd}.

A special limit estimate: for $1-\sigma \in [\frac{2}{d-2},1]$ then
\begin{equation}\label{eq:OmegaPest:3542dspecial}
 \|d\Omega_P\|_{L^{\frac{2d}{3  -\sigma}}} \aleq\|d \vec{u}\|_{L^{2d}(\R^d)}^{1+\sigma} \ \|\Ds{\frac{d-2}{2}} d\vec{u}\|_{L^2(\R^d)}^{1-\sigma}
\end{equation}
\end{lemma}
\begin{proof}

\underline{Spatial part of \eqref{eq:OmegaPest:3542}}: Denoting by $\Omega_P' := (\Omega_P^a)_{a=1,\ldots,n}$ we have by assumption $\Rz \cdot \Omega_P' = 0$ and thus
\[
\begin{split}
&\|\Ds{s_1} \Omega_P^a\|_{L^{\frac{2d}{1+2s_1  -\sigma}}}\\
\aleq& \|\Ds{s_1} \Rz^\perp(\Omega_P')\|_{L^{\frac{2d}{1+2s_1  -\sigma}}}\\
\aleq& \|\Ds{s_1} \Rz^\perp(P \nabla P^T)\|_{L^{\frac{2d}{1+2s_1  -\sigma}}}+\|\Ds{s_1} \Rz^\perp(P \Omega P^T)\|_{L^{\frac{2d}{1+2s_1  -\sigma}}}\\
\end{split}
\]
For the first term we have
\[
 \Ds{s_1}\Rz^\perp (P \nabla P^T) = \Ds{s_1-1} \curl (P \nabla P^T),
\]
so that with \Cref{la:fullOmegaPLdest} (where we use $1-\sigma$ instead of $\sigma$ for the application of \eqref{eq:GN:234})
\[
\begin{split}
 \|\Ds{s_1} \Rz^\perp(P \nabla P^T)\|_{L^{\frac{2d}{1+2s_1  -\sigma}}}
 \aleq&\max_{a,b \in \{1,\ldots,d\}}\|\Ds{s_1-1} \brac{\partial_a P \partial_b P^T}\|_{L^{\frac{2d}{1+2s_1  -\sigma}}}\\
 \overset{\eqref{eq:GN:234}}{\aleq}& \|\Dso P\|_{L^{2d}(\R^d)}^{2-(1-\sigma)} \ \|\Ds{\frac{d}{2}} P\|_{L^2(\R^d)}^{1-\sigma}\\
 \aleq&\|\Dso \vec{u}\|_{L^{2d}(\R^d)}^{2-(1-\sigma)} \ \|\Ds{\frac{d}{2}} \vec{u}\|_{L^2(\R^d)}^{1-\sigma}\\
 \end{split}
 \]
For the second term we have
\[
\begin{split}
 &\|\Ds{s_1} \Rz^\perp(P \Omega P^T)\|_{L^{\frac{2d}{1+2s_1  -\sigma}}} \\
 \aleq&\max_{i,\ldots,m}\|\Ds{s_1} \brac{[\Rz^\perp,P_{ij}] (P_{ml}\Omega_{kl})}\|_{L^{\frac{2d}{1+2s_1  -\sigma}}}\\
 &+\max_{i,\ldots,m}\|\Ds{s_1} \brac{P_{ij} [\Rz^\perp, P_{ml}](\Omega_{kl}}\|_{L^{\frac{2d}{1+2s_1  -\sigma}}}\\
 &+\max_{i,\ldots,m}\|\Ds{s_1} \brac{P_{ij} P_{ml} \Rz^\perp \Omega_{kl}}\|_{L^{\frac{2d}{1+2s_1  -\sigma}}}\\
  \aleq&\max_{i,\ldots,m}\|\Ds{s_1} \brac{[\Rz^\perp,P_{ij}] (P_{ml}\Omega_{kl})}\|_{L^{\frac{2d}{1+2s_1  -\sigma}}}\\
 &+\max_{t \in [0,s_1]} \max_{i,\ldots,m}
 \|\Ds{s_1-t} P_{ij}\|_{L^{\frac{d}{s_1-t}}(\R^d)}\,
 \|\Ds{t} [\Rz^\perp, P_{ml}](\Omega_{kl})\|_{L^{\frac{2d}{1+2t  -\sigma}}}\\
 &+\max_{i,\ldots,m} \max_{t_1+t_2+t_3 \in [0,s_1]} \|\Ds{t_1} P_{ij}\|_{L^{\frac{d}{t_1}}(\R^d)}\, \|\Ds{t_2} P_{ij}\|_{L^{\frac{d}{t_2}}(\R^d)}\, \|\Ds{t_3} \Rz^\perp \Omega_{kl}\|_{L^{\frac{2d}{1+2t_3 -\sigma}}}\\
\end{split}
 \]
So estimating the extra terms by $\|\Ds{d} P_{ij}\|_{L^{2}(\R^d)} \aleq \|\Ds{d} \vec{u}\|_{L^{2}(\R^d)} \leq 1$ and $\|P\|_{L^\infty} \leq 1$ we arrive at
\[
\begin{split}
 &\|\Ds{s_1} \Rz^\perp(P \Omega P^T)\|_{L^{\frac{2d}{1+2s_1  -\sigma}}} \\
  \aleq &\max_{i,\ldots,m} \max_{t_1+t_2+t_3 \in [0,s_1]} \|\Ds{t_3} \Rz^\perp \Omega_{kl}\|_{L^{\frac{2d}{1+2t_3 -\sigma}}}\\
&+\max_{i,\ldots,m}\|\Ds{s_1} \brac{[\Rz^\perp,P_{ij}] (P_{ml}\Omega_{kl})}\|_{L^{\frac{2d}{1+2s_1  -\sigma}}}\\
 &+\max_{t \in [0,s_1]} \max_{i,\ldots,m}
 \,
 \|\Ds{t} [\Rz^\perp, P_{ml}](\Omega_{kl})\|_{L^{\frac{2d}{1+2t  -\sigma}}}\\
\end{split}
 \]
Observe that since $(1-\sigma) \in {(}\frac{2(s_1-1)}{d-2},1]$ then $(1-\sigma) \in {(}\frac{2(t-1)}{d-2},1]$ for any $t \leq s_1$. Thus, by the curl-condition for the admissibility of $\Omega$,
\[
\begin{split}
& \|\Ds{t_3} \Rz^\perp \Omega_{kl}\|_{L^{\frac{2d}{1+2t_3 -\sigma}}} \\
\aeq&\|\Ds{t_3-1} \curl \Omega_{kl}\|_{L^{\frac{2d}{1+2t_3 -\sigma}}}\\
\overset{\eqref{eq:G1cond:GN1spatial}}{\aleq} &\|\Dso \vec{u}\|_{L^{2d}(\R^d)}^{2-(1-\sigma)} \ \|\Ds{\frac{d}{2}} \vec{u}\|_{L^2(\R^d)}^{1-\sigma}.
\end{split}
\]
From \eqref{eq:GN:gagliardo6} we have
\[
\begin{split}
 &\|\Ds{s_1} \brac{[\Rz^\perp,P_{ij}] (P_{ml}\Omega_{kl})}\|_{L^{\frac{2d}{1+2s_1  -\sigma}}}\\
 \aleq& \brac{\|\Dso P\|_{L^{2d}(\R^d)}^{1+\sigma}+\|P_{ml}\Omega_{kl}\|_{L^{2d}(\R^d)}^{1+\sigma}} \brac{\|\Ds{\frac{d}{2}} P\|_{L^2(\R^d)}^{1-\sigma}+\|\Ds{\frac{d-2}{2}} \brac{P_{ml}\Omega_{kl}}\|_{L^2(\R^d)}^{1-\sigma}}\\
 \aleq&\brac{\|\Dso P\|_{L^{2d}(\R^d)}^{1+\sigma}+\|\Omega\|_{L^{2d}(\R^d)}^{1+\sigma}} \brac{\|\Ds{\frac{d}{2}} P\|_{L^2(\R^d)}^{1-\sigma}+
 \max_{t \in [0,\frac{d-2}{2}]} \|\Ds{t} P\|_{L^{\frac{d}{t}}(\R^d)}^{1-\sigma} \|\Ds{\frac{d-2}{2}-t} \Omega\|_{L^{\frac{2d}{d-t2}}(\R^d)}^{1-\sigma}}\\
 \aleq&\|d\vec{u}\|_{L^{2d}(\R^d)}^{1+\sigma}\, \|\Ds{\frac{d-2}{2}} d\vec{u}\|_{L^2(\R^d)}^{1-\sigma}\\
 \end{split}
 \]
In the last step we used \Cref{la:fullOmegaPLdest}, \eqref{eq:Omegacond:growth2d}, the smallness of $\|\Ds{\frac{d-2}{2}}d u\|_{L^2(\R^d)}$ to estimate additional terms by $1$.
For exactly the same reasons we have
\[
\begin{split}
 &\|\Ds{s_1} \brac{[\Rz^\perp,P_{ml}] (\Omega_{kl})}\|_{L^{\frac{2d}{1+2s_1  -\sigma}}}\\
 \aleq&\|d\vec{u}\|_{L^{2d}(\R^d)}^{1+\sigma}\, \|\Ds{\frac{d-2}{2}} d\vec{u}\|_{L^2(\R^d)}^{1-\sigma}\\
 \end{split}
 \]

\underline{time part of \eqref{eq:OmegaPest:3542}}
Again we write
 \[
 \begin{split}
  \Ds{s_1}\Omega_P^0 \overset{\eqref{eq:asdcioxckj1}}{=}  &-c\Ds{s_1-1} \sum_{a = 1}^d\Rz_a\brac{P \brac{G_{1;a,0}(\vec{u})} P^T} \\
 &-c\sum_{a = 1}^d\Ds{s_1-1}\Rz_a\brac{\partial_t P \partial_a P^T - \partial_a P \partial_t P^T}\\
 &-c\sum_{a = 1}^d\Ds{s_1-1}\Rz_a\brac{ \partial_a P\,  \Omega^0 P^T}\\
 &-c\sum_{a = 1}^d\Ds{s_1-1}\Rz_a\brac{-\partial_t P\, \Omega^a P^T}\\
 & -c\sum_{a = 1}^d\Ds{s_1-1}\Rz_a\brac{P  \Omega^0\, \partial_a P^T}\\
 &-c\sum_{a = 1}^d\Ds{s_1-1}\Rz_a\brac{-\partial_t P\, \Omega^a P^T}\\
 \end{split}
 \]
From \eqref{eq:GN:234} (again, for $1-\sigma$ in place of $\sigma$) combined yet again with the observation that
\[
 \|\Ds{\frac{d-2}{2}} (PA) \|_{L^{2}(\R^d)} \aleq \brac{\|P\|_{L^\infty} +\|\Ds{\frac{d}{2}} P\|_{L^2(\R^d)}}\ \|\Ds{\frac{d-2}{2}} a\|_{L^{2}(\R^d)} \aleq \|\Ds{\frac{d-2}{2}} A\|_{L^{2}(\R^d)}.
\]
we readily find
 \[
 \begin{split}
    \|\Omega_P^0\|_{L^{\frac{2d}{1+2s_1  -\sigma}}} \aleq & \|\Ds{s_1-1}\brac{P \brac{G_{1;}(\vec{u})} P^T}\|_{L^{\frac{2d}{1+2s_1  -\sigma}}} \\
 &+\|d\vec{u}\|_{L^{2d}(\R^d)}^{1+\sigma}\, \|\Ds{\frac{d-2}{2}} d\vec{u}\|_{L^2(\R^d)}^{1-\sigma}.
 \end{split}
 \]
For the first term, if $s_1 \leq 1$ we use Sobolev embedding
\[
\begin{split}
 &\|\Ds{s_1-1}\brac{P \brac{G_{1;}(\vec{u})} P^T}\|_{L^{\frac{2d}{1+2s_1  -\sigma}}} \\
 \aleq &\|P \brac{G_{1;}(\vec{u})} P^T\|_{L^{\frac{2d}{3-\sigma}}}\\
 \aleq &\|G_{1;}(\vec{u})\|_{L^{\frac{2d}{3-\sigma}}}\\
 \overset{\eqref{eq:G1cond:GN1spatial}}{\aleq}& \|d\vec{u}\|_{L^{2d}(\R^d)}^{2-(1-\sigma)} \ \|\Ds{\frac{d-2}{2}} d\vec{u}\|_{L^2(\R^d)}^{1-\sigma}.
\end{split}
 \]
If $s_1 \geq 1$, using again \Cref{la:fullOmegaPLdest} and the smallness $\|\Ds{\frac{d-2}{2}}d u\|_{L^2(\R^d)}$ to estimate unnecessary terms by one,
\begin{equation}\label{eq:3489sdfisdf}
\begin{split}
 &\|\Ds{s_1-1}\brac{P \brac{G_{1;}(\vec{u})} P^T}\|_{L^{\frac{2d}{1+2s_1  -\sigma}}} \\
 \aleq &\brac{1+\|\Ds{\frac{d}{2}} P\|_{L^{2}(\R^d)}+\|\Ds{\frac{d}{2}} P\|_{L^{2}(\R^d)}^2}\, \max_{t \in [0,s_1]}\|\Ds{t-1}\brac{G_{1;}(\vec{u})} \|_{L^{\frac{2d}{1+2t  -\sigma}}}\\
 \overset{\eqref{eq:G1cond:GN1spatial}}{\aleq}& \|d\vec{u}\|_{L^{2d}(\R^d)}^{2-(1-\sigma)} \ \|\Ds{\frac{d-2}{2}} d\vec{u}\|_{L^2(\R^d)}^{1-\sigma}.
\end{split}
 \end{equation}
This concludes the proof of \eqref{eq:OmegaPest:3542}.

As for \eqref{eq:OmegaPest:3542d} we first observe that with the arguments above, using \eqref{eq:asdcioxckj1}, \eqref{eq:asdcioxckj2}, \Cref{la:fullOmegaPLdest}, \eqref{eq:G1cond:GN1spatial}, \eqref{eq:411Ssmall}
\begin{equation}\label{eq:AAAAARGH223}
\begin{split}
 \|\Ds{s_1-1} d\Omega_P\|_{L^{\frac{2d}{1+2s_1  -\sigma}}} \aleq &\|\Ds{s_1-1} \partial_t \Omega_P^0\|_{L^{\frac{2d}{1+2s_1  -\sigma}}}\\
 &+\|\Ds{s_1-1} PG_{1;}(\vec{u})P^T\|_{L^{\frac{2d}{1+2s_1  -\sigma}}}\\
 &+\max_{\alpha,\beta \in \{0,\ldots,d\}}\|\Ds{s_1-1}  (\partial_\alpha P\, \partial_\beta P^T)\|_{L^{\frac{2d}{1+2s_1  -\sigma}}}\\
 &+\max_{\alpha,\beta \in \{0,\ldots,d\}}\|\Ds{s_1-1}  (\partial_\alpha P\, \Omega^\beta P^T)\|_{L^{\frac{2d}{1+2s_1  -\sigma}}}\\
 &+\max_{\alpha,\beta \in \{0,\ldots,d\}}\|\Ds{s_1-1}  (P\, \Omega^\beta \partial_\alpha P^T)\|_{L^{\frac{2d}{1+2s_1  -\sigma}}}\\
 \aleq& \|\Ds{s_1-1} \partial_t \Omega_P^0\|_{L^{\frac{2d}{1+2s_1  -\sigma}}}+\|d\vec{u}\|_{L^{2d}(\R^d)}^{2-(1-\sigma)} \ \|\Ds{\frac{d-2}{2}} d\vec{u}\|_{L^2(\R^d)}^{1-\sigma}.
 \end{split}
\end{equation}
That is, it remains to estimate
\[
 \begin{split}
&\|\Ds{s_1-1} \partial_t \Omega_P^0\|_{L^{\frac{2d}{1+2s_1  -\sigma}}}  \\
\overset{\eqref{eq:asdcioxckj12}}{\aleq}& \max_{a \in \{1,\ldots,d\}}\|\Ds{s_1-2} \partial_t (P G_{1;a,0}(\vec{u}) P^T) \|_{L^{\frac{2d}{1+2s_1  -\sigma}}}\\
&+\max_{\alpha,\beta \in \{0,\ldots,d\}}\|\Ds{s_1-2} \brac{\partial_\alpha \partial_t P\, \partial_\beta P^T }   \|_{L^{\frac{2d}{1+2s_1  -\sigma}}}\\
&+\max_{\alpha,\beta \in \{0,\ldots,d\}}\|\Ds{s_1-2}  \partial_t\brac{\partial_\alpha P\,  \Omega^\beta P^T}  \|_{L^{\frac{2d}{1+2s_1  -\sigma}}}\\
&+\max_{\alpha,\beta \in \{0,\ldots,d\}}\|\Ds{s_1-2}  \partial_t\brac{P\,  \Omega^\beta\, \partial_\alpha P^T}  \|_{L^{\frac{2d}{1+2s_1  -\sigma}}}\\
 \end{split}
\]
Arguing similar to \eqref{eq:3489sdfisdf}
\[
\begin{split}
 &\|\Ds{s_1-2} \partial_t (P G_{1;a,0}(\vec{u}) P^T) \|_{L^{\frac{2d}{1+2s_1  -\sigma}}}\\
 \aleq& \|\partial_{t} \brac{P G_{1;}(\vec{u}) P^T}\|_{L^{\frac{2d}{5-\sigma}}(\R^d)} + (1+\|\Ds{\frac{d-2}{2}} dP\|_{L^2(\R^d)}^2)\max_{t \in [2,s_1]} \|\Ds{t-2} dG_{1;}(\vec{u})\|_{L^{\frac{2d}{1+2t-\sigma}}(\R^d)}\\
 \overset{\eqref{eq:411Ssmall}}{\aleq}&\max_{t \in [2,s_1]} \|\Ds{t-2} dG_{1;}(\vec{u})\|_{L^{\frac{2d}{1+2t-\sigma}}(\R^d)}\\
 \overset{\eqref{eq:G1cond:GN1}}{\aleq}& \|d \vec{u}\|_{L^{2d}}^{1+\sigma}\, \|\Ds{\frac{d-2}{2}} d\vec{u}\|_{L^2(\R^d)}^{1-\sigma}.
 \end{split}
\]
The above makes sense if $s_1 \geq 2$. If $s_1 \leq 2$ we obtain
\[
\begin{split}
 &\|\Ds{s_1-2} \partial_t (P G_{1;a,0}(\vec{u}) P^T) \|_{L^{\frac{2d}{1+2s_1  -\sigma}}}\\
\aleq&\|dG_{1;}(\vec{u})\|_{L^{\frac{2d}{1-\sigma}}(\R^d)}\\
 \overset{\eqref{eq:G1cond:GN1}}{\aleq}& \|d\vec{u}\|_{L^{2d}(\R^d)}^{2-(1-\sigma)} \ \|\Ds{\frac{d-4}{2}} dd\vec{u}\|_{L^2(\R^d)}^{1-\sigma}
 \end{split}
\]

Next, recall that $s_1 \in [0,\frac{d-2}{2}]$

If $s_1 -2 \leq 0$ then for $\sigma_1,\sigma_2 \in [0,1]$ with $1-\sigma_1 + 1-\sigma_2 = 1-\sigma$,
\[
\begin{split}
 &\|\Ds{s_1-2} \brac{\partial_\alpha \partial_t P\, \partial_\beta P^T }   \|_{L^{\frac{2d}{2+2(s_1-1) + 1-\sigma}}}\\
 \aleq&\|\partial_\alpha \partial_t P\, \partial_\beta P^T   \|_{L^{\frac{2d}{4 + 1-\sigma}}}\\
 \aleq& \|\partial_\alpha \partial_t P\|_{L^{\frac{2d}{4-\sigma_1}}(\R^d)} \| \partial_\beta P^T   \|_{L^{\frac{2d}{2-\sigma_2}}(\R^d)}\\
 \overset{\eqref{eq:411SSfullLqddP}}{\aleq}& \|\partial_\alpha d\vec{u}\|_{L^{\frac{2d}{3+(1-\sigma_1)}}(\R^d)} \| d P^T   \|_{L^{\frac{2d}{1+(1-\sigma_2)}}(\R^d)}\\
 \overset{\eqref{eq:GN:1}}{\aleq}& \|d\vec{u}\|_{L^{2d}(\R^d)}^{\sigma_1}\, \|\Ds{\frac{d-2}{2}} d\vec{u}\|_{L^2(\R^d)}^{1-\sigma_1}\, \| dP   \|_{L^{2d}(\R^d)}^{\sigma_2}\, \| \Ds{\frac{d-2}{2}}dP   \|_{L^{2d}(\R^d)}^{1-\sigma_2}\\
 \overset{\Cref{la:fullOmegaPLdest}}{\aleq}& \|d\vec{u}\|_{L^{2d}(\R^d)}^{1+\sigma}\, \|\Ds{\frac{d-2}{2}} d\vec{u}\|_{L^2(\R^d)}^{1-\sigma}.
 \end{split}
\]
This works for $1-\sigma_1 \in [\frac{2}{d-2},1]$, and by assumption we have $(1-\sigma) \in (\frac{2}{d-2},1]$, so we choose $1-\sigma_1$ just a little bit smaller than $1-\sigma$ (or equal, for that matter).

If $s_1 -2 \geq 0$, for $\sigma_1,\sigma_2\in [0,1]$ such that $\sigma_1 + \sigma_2 = 1-\sigma$,
\[
\begin{split}
 &\|\Ds{s_1-2} \brac{\partial_\alpha \partial_t P\, \partial_\beta P^T }   \|_{L^{\frac{2d}{2+2(s_1-1) +1 -\sigma}}}\\
\aleq&\max_{t \in [0,s_1-2]}\|\Ds{t} \partial_\alpha \partial_t P\|_{L^{\frac{2d}{1+2(t+1)+\sigma_1}}(\R^d)} \|\Ds{s_1-2-t} \partial_\beta P^T    \|_{L^{\frac{2d}{1+2(s_1-2-t)  +\sigma_2}}(\R^d)}\\
\overset{\eqref{eq:411SSfullLqddP}}{\aleq}&\max_{t \in [0,s_1-2]}\|\Ds{t+1} d\vec{u}\|_{L^{\frac{2d}{1+2(t+1)+\sigma_1}}(\R^d)} \|\Ds{s_1-2-t} d\vec{u}    \|_{L^{\frac{2d}{1+2(s_1-2-t)  +\sigma_2}}(\R^d)}\\
\overset{\eqref{eq:GN:1}}{\aleq}& \|d\vec{u}\|_{L^{2d}(\R^d)}^{1-\sigma_1}\|\Ds{\frac{d-2}{2}}d\vec{u}\|_{L^{2}(\R^d)}^{\sigma_1}\ \|d\vec{u}\|_{L^{2d}(\R^d)}^{1-\sigma_2}\|\Ds{\frac{d-2}{2}}d\vec{u}\|_{L^{2}(\R^d)}^{\sigma_2}\\
=& \|d\vec{u}\|_{L^{2d}(\R^d)}^{1+\sigma} \|\Ds{\frac{d-2}{2}}d\vec{u}\|_{L^{2}(\R^d)}^{1-\sigma}.
 \end{split}
\]
For this to work we need that $\sigma_1 \in [\frac{2(t+1)}{d-2},1]$ and $\sigma_2 \in [\frac{2(s_1-2-t)}{d-2},1]$. Since $\sigma_1+\sigma_2 = (1-\sigma) \in {(}\frac{2(s_1-1)}{d-2},1]$ this choice of $\sigma_1$, $\sigma_2$ is possible.

The last term we consider for \eqref{eq:AAAAARGH223} is the following (the remaining term goes exactly the same) is
\[
\begin{split}
\max_{\alpha,\beta \in \{0,\ldots,d\}}\|\Ds{s_1-2}  \partial_t\brac{\partial_\alpha P\,  \Omega^\beta P^T}  \|_{L^{\frac{2d}{1+2s_1  -\sigma}}}\\
\end{split}
\]

Again, if $s_1 -2 \leq 0$ then for $\sigma_1,\sigma_2 \in [0,1]$ with $1-\sigma_1 + 1-\sigma_2 = 1-\sigma$,
\[
\begin{split}
 &\|\Ds{s_1-2} \partial_t\brac{\partial_\alpha P\,  \Omega^\beta P^T}     \|_{L^{\frac{2d}{2+2(s_1-1) + 1-\sigma}}}\\
 \aleq&\|\partial_t\brac{\partial_\alpha P\,  \Omega^\beta P^T}   \|_{L^{\frac{2d}{4 + 1-\sigma}}}\\
 \aleq& \|P^T\|_{L^\infty} \|\partial_\alpha \partial_t P\|_{L^{\frac{2d}{4-\sigma_1}}(\R^d)} \| \Omega^\beta   \|_{L^{\frac{2d}{2-\sigma_2}}(\R^d)}\\
 &+\|P^T\|_{L^\infty} \|\partial_t \Omega^\beta   \|_{L^{\frac{2d}{4-\sigma_1}}(\R^d)} \|\partial_\alpha P\|_{L^{\frac{2d}{2-\sigma_2}}(\R^d)}\\
 &+\|\partial_t P^T\|_{L^{d}(\R^d)} \|\Omega^\beta   \|_{L^{\frac{2d}{2-\sigma_1}}(\R^d)} \|\partial_\alpha P\|_{L^{\frac{2d}{2-\sigma_2}}(\R^d)}\\
\overset{\eqref{eq:411SSfullLqddP}, \eqref{eq:Omegacond:growthpq}}{\aleq}& \|\partial_\alpha d\vec{u}\|_{L^{\frac{2d}{3+(1-\sigma_1)}}(\R^d)} \| d \vec{u}  \|_{L^{\frac{2d}{1+(1-\sigma_2)}}(\R^d)}\\
&+\|d\vec{u}\|_{L^{d}(\R^d)}\, \|d \vec{u}\|_{L^{\frac{2d}{2-\sigma_1}}(\R^d)} \|d \vec{u}\|_{L^{\frac{2d}{2-\sigma_2}}(\R^d)}\\
\overset{\eqref{eq:GN:1}}{\aleq}& \|d\vec{u}\|_{L^{2d}(\R^d)}^{1+\sigma}\, \|\Ds{\frac{d-2}{2}} d\vec{u}\|_{L^2(\R^d)}^{1-\sigma}\\
&+\|d\vec{u}\|_{L^{d}(\R^d)} \|d\vec{u}\|_{L^{2d}(\R^d)}^{1+\sigma} \ \|\Ds{\frac{d-2}{2}} d\vec{u}\|_{L^2(\R^d)}^{1-\sigma}\\
\overset{\eqref{eq:411Ssmall}}{\aleq}&\|d\vec{u}\|_{L^{2d}(\R^d)}^{1+\sigma}\, \|\Ds{\frac{d-2}{2}} d\vec{u}\|_{L^2(\R^d)}^{1-\sigma}.
 \end{split}
\]
Again by the choice of $\sigma$ we can choose $\sigma_1$, $\sigma_2$ accordingly to fit the assumptions of \eqref{eq:GN:1}.

As for \eqref{eq:OmegaPest:3542dspecial}, we argue similar to \eqref{eq:411SSfullG1cond:3:OmegaP}.
\[
\begin{split}
&\|\partial_t\Omega^0_P\|_{L^{\frac{2d}{3-\sigma}}(\R^d)}\\
\aleq& \|\partial_t (P\, G_{1;a,0}(\vec{u})\, P^T)\|_{L^{\frac{2d}{5-\sigma}}(\R^d)} \\
 &+\|\partial_t \brac{\Omega_P^0 \partial_a P^T - \partial_a P\, P^T  \Omega_P^0}\|_{L^{\frac{2d}{5-\sigma}}(\R^d)}\\
 &+\|\partial_t \brac{-P \Omega^0\, \partial_a P^T + \partial_a P\, \Omega^0 P^T}\|_{L^{\frac{2d}{5-\sigma}}(\R^d)}\\
 &+ \|\partial_t\brac{\partial_a P\,  \Omega^0 P^T}\|_{L^{\frac{2d}{5-\sigma}}(\R^d)}\\
 &+\|\partial_t \brac{P\Omega^0  \Omega^a P^T}\|_{L^{\frac{2d}{5-\sigma}}(\R^d)}\\
 &+\|\partial_t \brac{\Omega^0_P P \Omega^a P^T}\|_{L^{\frac{2d}{5-\sigma}}(\R^d)}\\
 &+ \|\partial_t \brac{P  \Omega^0\, \partial_a P^T}\|_{L^{\frac{2d}{5-\sigma}}(\R^d)}\\
 &+\|\partial_t \brac{\Omega^0\, \Omega^a P^T}\|_{L^{\frac{2d}{5-\sigma}}(\R^d)}\\
&+\|\partial_t \brac{\Omega^0_P P \Omega^a P^T}\|_{L^{\frac{2d}{5-\sigma}}(\R^d)}\\
\aleq&(1+\|\partial_t P\|_{L^{d}(\R^d)})\, \|dG_{1;a,0}(\vec{u})\|_{L^{\frac{2d}{5-\sigma}}(\R^d)} \\
&+ \brac{\|d\Omega_P\|_{L^{\frac{2d}{4-\sigma}}(\R^d)} + \|d\Omega\|_{L^{\frac{2d}{4-\sigma}}(\R^d)}+ \|d\Dso P\|_{L^{\frac{2d}{4-\sigma}}(\R^d)} }\brac{\|dP\|_{L^{2d}} +\|\Omega\|_{L^{2d}(\R^d)}}\\
\aleq&\|dG_{1;a,0}(\vec{u})\|_{L^{\frac{2d}{5-\sigma}}(\R^d)} + \|\Dso d\vec{u}\|_{L^{\frac{2d}{4-\sigma}}(\R^d)}\, \|d\vec{u}\|_{L^{2d}},
\end{split}
\]
where in the last inequality we used \eqref{eq:411Ssmall},\eqref{eq:411SSfullLqdd}, \eqref{eq:Omegacond:growthpq}, \eqref{eq:Omegacond:growthpq2}, \eqref{eq:411SSfullLqddP1}
If $d=3$ we'd need $\sigma > 0$ to apply \eqref{eq:411SSfullLqdd}, but since we assume $d \geq 4$ this is true for any $\sigma \in [0,1]$.

From \eqref{eq:G1cond:3sigma}, \eqref{eq:dduest1}, \eqref{eq:dduest2}, for $(1-\sigma) \in [\frac{2}{d-2},1]$,
\[
 \|dG_{1;a,0}(\vec{u})\|_{L^{\frac{2d}{5-\sigma}}(\R^d)} \aleq \| d \vec{u} \|_{L^{2d}(\R^d)}^{1+\sigma}\, \| \Ds{\frac{d-2}{2}} d \vec{u} \|_{L^{2}(\R^d)}^{1-\sigma}
\]

From \eqref{eq:GN:1}, whenever $(1-\sigma) \in [\frac{2}{d-2},1]$,
\[
 \|d\nabla \vec{u}\|_{L^{\frac{2d}{4-\sigma}}(\R^d)} \aleq \|d\vec{u}\|_{L^{2d}(\R^d)}^\sigma\ \|\Ds{\frac{d-2}{2}} d\vec{u}\|_{L^2(\R^d)}^{1-\sigma}.
\]
So we have shown that whenever $(1-\sigma) \in [\frac{2}{d-2},1]$,
\[
 \|\partial_t\Omega^0_P\|_{L^{\frac{2d}{3-\sigma}}(\R^d)} \aleq \| d \vec{u} \|_{L^{2d}(\R^d)}^{1+\sigma}\, \| \Ds{\frac{d-2}{2}} d \vec{u} \|_{L^{2}(\R^d)}^{1-\sigma}
\]
The other terms of $d\Omega_P$ are similar, we can conclude.
\end{proof}

\section{A priori estimates: Proof of Theorem~\ref{th:aprioriest}}\label{s:apriori}

Let $\vec{u}$ be a solution to \eqref{eq:generalwave} as in the assumptions of \Cref{th:aprioriest}.

In the following we take $P: \R^d\times (-T,T) \times \R \to SO(N)$ from \Cref{s:gauge} -- which is possible for $\eps$ small enough, by the growth assumptions on $\Omega$, \eqref{eq:Omegacond:growth}.

We compute
\[
\begin{split}
 &\partial_{t} (P\partial_t \vec{u}) - \div(P \nabla \vec{u})\\
 =& \partial_t P \partial_t \vec{u} - \nabla P\, \nabla \vec{u} + P (\partial_{tt} \vec{u}- \lap \vec{u})\\
 =& \partial_t P \partial_t \vec{u} - \nabla P\, \nabla \vec{u} + P \Omega^0 \partial_t \vec{u} + P \Omega \nabla \vec{u}\\
 =&(\partial_t P P^T+ P\Omega^0\, P^T)\, P\partial_t \vec{u} + (-\nabla P\, P^T + P \Omega P^T )\cdot P\nabla \vec{u} + P F(\vec{u})\\
 \end{split}
\]
With the definition \eqref{eq:OmegaPalpha} this becomes
\begin{equation}\label{eq:ap:Pvueq}
 \partial_{t} (P\partial_t \vec{u}) - \div(P \nabla \vec{u}) = \Omega_P \cdot P d \vec{u} + P F(\vec{u}).
\end{equation}

Now we perform essentially a space-time Hodge decomposition. Solve
\begin{equation}\label{eq:ap:veq}
 \begin{cases}
  \brac{\partial_{tt} - \lap} \vec{v} := \Omega_P \cdot P d \vec{u} + P F(\vec{u})\\
  \vec{v}(0) = 0\\
  \partial_t\vec{v}(0) = 0
 \end{cases}
\end{equation}
and set 
\begin{equation}\label{eq:WPuv}
 W_{i} := P_{ij} d u^j-dv^i \quad \text{$i = 1,\ldots N$}
\end{equation}

We will compute the equation of $W_i$ below, but first we observe the estimates for $v$ we obtain from \Cref{s:gauge} and the inversion of the wave operator in \eqref{eq:ap:veq}.

\begin{lemma}[Estimate for $\vec{v}$]\label{la:estimatesv}
Let $s \in (1,\frac{d^2-4d+1}{2(d-1)})$ and $d\geq 4$. There exists $\eps > 0$ such that if \eqref{eq:ap:smallnessdd} holds
and $v$ solves \eqref{eq:ap:veq} then
\begin{equation}\label{eq:estimatev:goal}
\begin{split}
 &\|\Ds{\frac{d-2}{2}} d\vec{v}\|_{C^0_t L^2_x((\R^d\times (-T,T))} + \|\Ds{s-1}{d} \vec{v}\|_{L^2_t L^{(\frac{2d}{2s-1},2)}_x(\R^d \times (-T,T))}\\
\aleq_s & \|\Ds{s-1} d\vec{u}\|_{L^2_t L^{(\frac{2d}{2s-1},2)}_x(\R^d \times (-T,T))}^2.
 \end{split}
 \end{equation}
\end{lemma}
\begin{proof}
At a fixed point in time $t \in (-T,T)$ we estimate
\[
 \|\Ds{\frac{d-2}{2}}(\partial_{tt}-\lap)\vec{v}\|_{L^2_x (\R^d)} \aleq \|\Ds{\frac{d-2}{2}}\brac{\Omega_P \cdot P d\vec{u}}\|_{L^2_x (\R^d)}+\|\Ds{\frac{d-2}{2}}\brac{P F(\vec{u})}\|_{L^2_x (\R^d)}
\]

We observe that 
\[
\begin{split}
 \|\Ds{\frac{d-2}{2}}\brac{P F(\vec{u})}\|_{L^2 (\R^d)} \aleq&
 \brac{\|P\|_{L^\infty(\R^d)} + \|\Ds{\frac{d}{2}} P\|_{L^2(\R^d)}}\, \|\Ds{\frac{d-2}{2}} F(\vec{u})\|_{L^2(\R^d)}.
 \end{split}
\]
By admissibility of $F$, \Cref{def:admissibleF} \eqref{eq:Fcond:3}, we have
\[
 \|\Ds{\frac{d-2}{2}}\brac{P F(\vec{u})}\|_{L^2 (\R^d)} \aleq \brac{1 + \|\Ds{\frac{d}{2}} P\|_{L^2(\R^d)}}\,
 \brac{1 + \|\Ds{\frac{d-2}{2}} d\vec{u}\|_{L^2(\R^d)}^{\gamma}}
  \|\Ds{s-1} d\vec{u}\|_{L^{(\frac{2d}{2s-1},2)}(\R^d)}^2
\]
In view of \Cref{la:fullOmegaPLdest}, the smallness assumption \eqref{eq:ap:smallnessdd} we have shown for suitably small $\eps$
\begin{equation}\label{eq:Festasdkjasd}
 \|\Ds{\frac{d-2}{2}}\brac{P F(\vec{u})}\|_{L^2 (\R^d)} \aleq
  \|\Ds{s-1} d\vec{u}\|_{L^{(\frac{2d}{2s-1},2)}(\R^d)}^2
  \end{equation}

On the other hand we compute, we will now show
\begin{equation}\label{eq:adasd2}
\max_{\alpha,\beta \in \{0,\ldots,d\}}\|\Ds{\frac{d-4}{2}} d (\Omega_P^\alpha \, P^T \partial_\beta \vec{u})\|_{L^2(\R^d)}
 \aleq\gamma  \|\Ds{s-2} dd\vec{u}\|_{L^{(\frac{2d}{2s-1},2)}(\R^d)}^2.
\end{equation}
and
\begin{equation}\label{eq:adasd2spatial}
\max_{\alpha,\beta \in \{0,\ldots,d\}}\|\Ds{\frac{d-2}{2}}  (\Omega_P^\alpha \, P^T \partial_\beta \vec{u})\|_{L^2(\R^d)}
 \aleq  \|\Ds{s-1} d\vec{u}\|_{L^{(\frac{2d}{2s-1},2)}(\R^d)}^2.
\end{equation}

Observe that this is slightly more general that what is needed here, but later for the equation in \Cref{la:ap:Wwave} we also need the time-derivative.

In order to obtain \eqref{eq:adasd2}, with the help of the Leibniz rule,
\begin{equation}\label{eq:vest:asdiuxcvziv2}
\begin{split}
 &\max_{\alpha,\beta \in \{0,\ldots,d\}}\|\Ds{\frac{d-4}{2}}d (\Omega_P^\alpha \, P \partial_\beta \vec{u})\|_{L^2(\R^d)}\\
 \aleq& \|\Omega_P\|_{L^\infty(\R^d)}\, \|\Ds{\frac{d-4}{2}}d  \brac{Pd\vec{u}}\|_{L^2(\R^d)}\\
 &+\|P\|_{L^\infty}\, \|\Ds{\frac{d-4}{2}}d\Omega_P\|_{L^{\frac{2d}{d-1}}(\R^d)}\, \| d\vec{u}\|_{L^{2d}(\R^d)}\\
%
  &+\max_{s_1+s_2=\frac{d-2}{2}, s_1,s_2 \geq 1} \|\Ds{s_1-1} d\Omega_P\|_{L^{\frac{2d}{1+2s_1-\sigma}}(\R^d)}\, \|\Ds{s_2} (P d \vec{u})\|_{L^{\frac{2d}{1+\sigma+2s_2}}(\R^d)}\\
    &+\max_{s_1+s_2=\frac{d-2}{2}, s_1,s_2 \geq 1} \|\Ds{s_1} \Omega_P\|_{L^{\frac{2d}{1+2s_1-\sigma}}(\R^d)}\, \|\Ds{s_2-1}d (P d \vec{u})\|_{L^{\frac{2d}{1+\sigma+2s_2}}(\R^d)}.
  \end{split}
\end{equation}
For the first line of \eqref{eq:vest:asdiuxcvziv2} by \Cref{la:omegaPLinfty}
\[
 \|\Omega_P\|_{L^\infty} \aleq \|d\vec{u}\|_{L^{(2d,2)}(\R^d)}^2,
\]
so that
\[
\begin{split}
& \|\Omega_P\|_{L^\infty(\R^d)}\, \|\Ds{\frac{d-4}{2}}d  \brac{Pd\vec{u}}\|_{L^2(\R^d)}\\
\aleq& \|d\vec{u}\|_{L^{(2d,2)}(\R^d)}^2\, (1+\|\Ds{\frac{d-2}{2}}d P\|_{L^2(\R^d)}) \brac{\|\Ds{\frac{d-2}{2}} d\vec{u}\|_{L^2(\R^d)}+\|\Ds{\frac{d-4}{2}} dd\vec{u}\|_{L^2(\R^d)}}\\
 \aleq&  \|d\vec{u}\|_{L^{(2d,2)}(\R^d)}^2,
 \end{split}
\]
For the second line of \eqref{eq:vest:asdiuxcvziv2}, if $d \geq 5$ we apply \eqref{eq:OmegaPest:3542d} (for $s_1 = \frac{d-2}{2}$, $\sigma=0$) to obtain
\[
 \|\Ds{\frac{d-4}{2}} d\Omega_P\|_{L^{\frac{2d}{d-1}}} \aleq\|d \vec{u}\|_{L^{2d}(\R^d)} \, \|\Ds{\frac{d-2}{2}} d\vec{u}\|_{L^2(\R^d)}.
\]
If $d=4$, we use and obtain \eqref{eq:411SSfullG1cond:3:OmegaP}
\[
 \|\Ds{\frac{d-4}{2}} d\Omega_P\|_{L^{\frac{2d}{d-1}}} \aleq \|d\vec{u}\|_{L^{(2d,2)}(\R^d)}\, \|\Ds{\frac{d-2}{2}} d\vec{u}\|_{L^{2}(\R^d)}
\]
Thus,
\[
\begin{split}
&\|P\|_{L^\infty}\, \|\Ds{\frac{d-4}{2}}d\Omega_P\|_{L^{\frac{2d}{d-1}}(\R^d)}\, \| d\vec{u}\|_{L^{2d}(\R^d)}\\
\aleq &\|d \vec{u}\|_{L^{(2d,2)}(\R^d)}^2 \, \|\Ds{\frac{d-2}{2}} d\vec{u}\|_{L^2(\R^d)}\\
\aleq & \|d \vec{u}\|_{L^{(2d,2)}(\R^d)}^2\\
\end{split}
\]
For $d=4$, this is all of \eqref{eq:vest:asdiuxcvziv2}, since the $\max$-lines are trivially zero.

So \underline{assume from now on $d \geq 5$}.
For the third line of \eqref{eq:vest:asdiuxcvziv2} we apply \eqref{eq:OmegaPest:3542d}
\[
 \|\Ds{s_1-1} d\Omega_P\|_{L^{\frac{2d}{1+2s_1  -\sigma}}} \aleq\|d \vec{u}\|_{L^{2d}(\R^d)}^{2-(1-\sigma)} \ \|\Ds{\frac{d-2}{2}} d\vec{u}\|_{L^2(\R^d)}^{1-\sigma}.
\]
which holds whenever
\begin{equation}\label{eq:xcvklxcklxmvxsigma1}
1-\sigma > \max\{\frac{2(s_1-1)}{d-2} ,\frac{2}{d-2}\}.
\end{equation}
and
\[
\begin{split}
 &\|\Ds{s_2} \brac{P d \vec{u}}\|_{L^{\frac{2d}{1+\sigma+2s_2}}(\R^d)}\\
 &\aleq \|P\|_{L^\infty} \|\Ds{s_2} d \vec{u}\|_{L^{\frac{2d}{1+\sigma+2s_2}}(\R^d)} + \max_{t_1+t_2 = s_2} \|\Ds{t_1} P\|_{L^{\frac{d}{t_1}}(\R^d)} \|\Ds{t_2} d \vec{u}\|_{L^{\frac{2d}{1+\sigma+2t_2}}(\R^d)}\\
  \overset{\eqref{eq:GN:1}}{\aleq}&\|d \vec{u}\|_{L^{2d}(\R^d)}^{1-\sigma}\, \|\Ds{\frac{d-2}{2}} d \vec{u}\|_{L^2(\R^d)}^\sigma +\|\Ds{\frac{d}{2}} P\|_{L^{2}(\R^d)} \|d \vec{u}\|_{L^{2d}(\R^d)}^{1-\sigma}\, \|\Ds{\frac{d-2}{2}} d \vec{u}\|_{L^2(\R^d)}^{\sigma}\\
  \aleq&\brac{1+\|\Ds{\frac{d}{2}} P\|_{L^{2}(\R^d)}}\, \|d \vec{u}\|_{L^{2d}(\R^d)}^{1-\sigma}\, \|\Ds{\frac{d-2}{2}} d \vec{u}\|_{L^2(\R^d)}^{\sigma}.
 \end{split}
\]
which holds if we have $ \sigma \in {(}\frac{2t_2}{d-2},1]$ so in particular if
\begin{equation}\label{eq:xcvklxcklxmvxsigma2}
 \sigma \in {(}\frac{2s_2}{d-2},1]
\end{equation}
So we need to choose $\sigma \in [0,1]$ such that \eqref{eq:xcvklxcklxmvxsigma1} and \eqref{eq:xcvklxcklxmvxsigma2} both hold (for a fixed $s_1,s_2 \in (0,\frac{d-2}{2})$ with $s_1+s_2 = \frac{d-2}{2}$).
That is we need to find $\sigma \in [0,1]$ such that
\[
(1-\sigma) >\max\{\frac{2(s_1-1)}{d-2} ,\frac{2}{d-2}\}, \quad  \sigma > \frac{2s_2}{d-2}
\]
which is possible if
\[
\begin{split}
&\frac{2(s_1-1)}{d-2} < 1-\frac{2s_2}{d-2}\\
\Leftrightarrow&s_1+s_2-1 < \frac{d-2}{2}\\
\end{split}
\]
(the latter one is true since $s_1+s_2 = \frac{d-2}{2}$), and if
\[
\begin{split}
&\frac{2}{d-2} < 1-\frac{2s_2}{d-2}\\
\Leftrightarrow&s_2 < \frac{d-4}{2}\\
\end{split}
\]
This is true unless $s_2 = \frac{d-4}{2}$, $s_1=1$, that is the remaining case is
\[
\begin{split}
& \|d\Omega_P\|_{L^{\frac{2d}{3-\sigma}}(\R^d)}\, \|\Ds{\frac{d-4}{2}} (P d \vec{u})\|_{L^{\frac{2d}{1+\sigma+2\frac{d-4}{2}}}(\R^d)}\\
 \aleq& (1+\|\Ds{d} P\|_{L^2(\R^d)})\, \|d\Omega_P\|_{L^{\frac{2d}{3-\sigma}}(\R^d)}\, \|\Ds{\frac{d-4}{2}} d \vec{u}\|_{L^{\frac{2d}{1+\sigma+2\frac{d-4}{2}}}(\R^d)}\\
 \end{split}
\]
In this case we apply \eqref{eq:OmegaPest:3542dspecial} and \eqref{eq:GN:1} for $\sigma =\frac{d-4}{d-2}$ which implies
\[
\|\Ds{\frac{d-4}{2}} d \vec{u}\|_{L^{\frac{2d}{1+\sigma+2\frac{d-4}{2}}}(\R^d)} \aleq
\|d \vec{u}\|_{L^{2d}(\R^d)}^{1-\sigma}\, \|\Ds{\frac{d-2}{2}} d \vec{u}\|_{L^{2}(\R^d)}^\sigma
\]
In summary, we have an estimate for the third line of \eqref{eq:vest:asdiuxcvziv2} (for the right choice of $\sigma$ depending on $s_1$ and $s_2$)
\[
\begin{split}
&\max_{s_1+s_2=\frac{d-2}{2}, s_1,s_2 \geq 1} \|\Ds{s_1-1} d\Omega_P\|_{L^{\frac{2d}{1+2s_1-\sigma}}(\R^d)}\, \|\Ds{s_2} (P d \vec{u})\|_{L^{\frac{2d}{1+\sigma+2s_2}}(\R^d)}\\
\aleq & \|d \vec{u}\|_{L^{2d}(\R^d)}^2\, \|\Ds{\frac{d-2}{2}} d \vec{u}\|_{L^{2}(\R^d)}.
\end{split}
\]

For the fourth line of \eqref{eq:vest:asdiuxcvziv2} we use \eqref{eq:OmegaPest:3542},

\[
 \|\Ds{s_1} \Omega_P\|_{L^{\frac{2d}{1+2s_1  -\sigma}}} \aleq\|d \vec{u}\|_{L^{2d}(\R^d)}^{2-(1-\sigma)} \ \|\Ds{\frac{d-2}{2}} d\vec{u}\|_{L^2(\R^d)}^{1-\sigma}
\]
which holds if \[1-\sigma > \frac{2(s_1-1)}{d-2}. \]
On the other hand, by Leibniz rule
\[
\begin{split}
\|\Ds{s_2-1}d (P d \vec{u})\|_{L^{\frac{2d}{1+\sigma+2s_2}}(\R^d)}
\aleq& (1+\|\Ds{\frac{d-2}{2}} dP\|_{L^2(\R^d)})\, \|\Ds{s_2-1}dd \vec{u}\|_{L^{\frac{2d}{1+\sigma+2s_2}}(\R^d)}
\end{split}
\]
Recall that $s_1+s_2 = \frac{d-2}{2}$ and $s_1,s_2 \geq 1$

Apply \eqref{eq:GN:1} to $f = \Ds{-1}dd \vec{u}$ then we have
\[
 \|\Ds{s_2-1} dd\vec{u}\|_{L^{\frac{2d}{1+\sigma+2s_2}}(\R^d)} \aleq \|\Ds{-1} dd\vec{u}\|_{L^{2d}(\R^d)}^{1-\sigma} \|\Ds{\frac{d-4}{2}} dd\vec{u}\|_{L^2(\R^d)}^{\sigma}
\]
which holds if $\sigma \geq \frac{2s_2}{d-2}$.

So we need to find $\sigma \in [0,1]$ such that
\[
\frac{2(s_1-1)}{d-2} < 1-\sigma \leq 1-\frac{2s_2}{d-2}
\]
which again is possible since $s_1+s_2-1 < \frac{d-2}{2}$.

That is, we have shown
\[
\begin{split}
&\max_{s_1+s_2=\frac{d-2}{2}, s_1,s_2 \geq 1} \|\Ds{s_1} \Omega_P\|_{L^{\frac{2d}{1+2s_1-\sigma}}(\R^d)}\, \|\Ds{s_2-1} d(P d \vec{u})\|_{L^{\frac{2d}{1+\sigma+2s_2}}(\R^d)}\\
\aleq & \|\Ds{-1}dd \vec{u}\|_{L^{2d}(\R^d)}^2\, \|\Ds{\frac{d-2}{2}} d \vec{u}\|_{L^{2}(\R^d)}.
\end{split}
\]

This proves \eqref{eq:adasd2} and \eqref{eq:adasd2spatial} (and since the last line of \eqref{eq:vest:asdiuxcvziv2} does not exist for \eqref{eq:adasd2spatial})

In particular, we have shown for any $t \in (-T,T)$
\[
 \|\Ds{\frac{d-2}{2}}(\partial_{tt}-\lap)\vec{v}\|_{L^2_x (\R^d)}  \aleq \gamma  \|\Ds{s-1} d\vec{u}\|_{L^{(\frac{2d}{2s-1},2)}(\R^d)}^2
\]
Integrating this inequality $L^1_t(-T,T)$ we have
\[
 \|\Ds{\frac{d-4}{2}}d(\partial_{tt}-\lap)\vec{v}\|_{L^1_tL^2_x (\R^d \times (-T,T))}  \aleq \gamma  \|\Ds{s-1} d\vec{u}\|_{L^2_t L^{(\frac{2d}{2s-1},2)}(\R^d \times (-T,T) )}^2
\]
From Strichartz estimates \Cref{la:strichartz}, we have
\[
\begin{split}
 &\|\Ds{\frac{d-2}{2}} d\vec{v}\|_{C^0_t L^2(\R^d \times (-T,T))} + \|\Ds{s-1} {d}\vec{v}\|_{L^2_t L^{(\frac{2d}{2s-1},2)}_x(\R^d \times (-T,T)}\\
 \aleq &\|\Ds{\frac{d-2}{2}}(\partial_{tt}-\lap)\vec{v}\|_{L^1_t L^2_x (\R^d \times (-T,T))}\\
 \aleq&  \|\Ds{s-1} d\vec{u}\|_{L^2_t L^{(\frac{2d}{2s-1},2)}_x(\R^d \times (-T,T) )}^2.
 \end{split}
\]
We can conclude.
\end{proof}

We need to obtain similar estimates for $W=P d \vec{u}-d\vec{v}$, which for $\alpha \in [0,\ldots,d]$ and $i \in \{1,\ldots,N\}$ we write as
\[
 W_{i}^\alpha := P_{ij} \partial_\alpha u^j-\partial_\alpha v^i.
\]
As usual we identify $\partial_0$ with $\partial_t$.

\begin{lemma}\label{la:ap:Wwave}
We have
\begin{equation}\label{eq:Wwavet}
\begin{split}
 (\partial_{tt} - \lap) \vec{W}^0 =
  &- \sum_{a=1}^d \partial_a (-\Omega_P^a\, P \partial_t \vec{u} - \Omega_P^0\, P\partial_a \vec{u}+P\, G_{2;a,0}(\vec{u})). \\
 \end{split}
\end{equation}
and for $b \in \{1,\ldots,d\}$,
\begin{equation}\label{eq:Wwaveb}
\begin{split}
 (\partial_{tt} - \lap) \vec{W}^b
  =&\partial_t (\Omega_P^0 \, P\partial_b \vec{u}+ \Omega_P^b\, P\partial_t \vec{u}-P\, G_{2;b,0}(\vec{u}))\\
  &- \sum_{a=1}^d \partial_a (-\Omega_P^a\, P \partial_b \vec{u} + \Omega_P^b\, P\partial_a \vec{u}+P\, G_{2;a,b}(\vec{u})). \\
 \end{split}
\end{equation}
where we chose  ``$+$'' for $\pm$ if $\beta = 0$, and $-$ for $\pm$ if $\beta \in \{1,\ldots,d\}$.
\end{lemma}
\begin{proof}
Essentially by construction, we have for any $i \in \{1,\ldots,N\}$
\[
\begin{split}
 \partial_t W_i^0 - \sum_{a=1}^d\partial_a W_i^a =&\partial_t (P_{ij} \partial_t u^j - \partial_t v^i) - \partial_a (P_{ij} \partial_a u^j - \partial_a v^i) \\
 =&\partial_t (P_{ij} \partial_t u^j ) - \div (P_{ij} \nabla u^j ) 
 -\brac{\partial_{tt} v^i - \lap v^i} = 0.
 \end{split}
\]

We then have 
\[
\begin{split}
 (\partial_{tt} - \lap) W_i^\beta =& \partial_\beta (\underbrace{\partial_t W_i^t - \partial_a W_i^a}_{=0}) + \partial_t (\partial_t W_i^\beta - \partial_\beta W_i^t) - \partial_a (\partial_a W_i^\beta - \partial_\beta W_i^a) \\
  =&\partial_t (\partial_t(P_{ij} \partial_\beta u^j)- \partial_\beta (P_{ij} \partial_t u^j)) - \partial_a (\partial_a (P_{ij} \partial_\beta u^j) - \partial_\beta (P_{ij} \partial_a u^j)) \\
  =&\partial_t (\partial_t P_{ij}\, \partial_\beta u^j- \partial_\beta P_{ij}\, \partial_t u^j) - \partial_a (\partial_a P_{ij} \partial_\beta u^j - \partial_\beta P_{ij}\, \partial_a u^j). \\
 \end{split}
\]

We now need to distinguish between $\beta = 0$ and $\beta \neq 0$.

For $\beta = 0$ we have
\[
\begin{split}
 (\partial_{tt} - \lap) W_i^0 =
  &- \sum_{a=1}^d \partial_a (\partial_a P_{ij} \partial_t u^j - \partial_t P_{ij}\, \partial_a u^j). \\
 \end{split}
\]
We now use the cancellation condition \eqref{eq:omeacommdu}.
\[
\begin{split}
&\partial_a P \partial_t \vec{u} - \partial_t P\, \partial_a \vec{u}\\
=&\brac{\partial_a P\, P^T  -P\Omega^a P^T}P \partial_t \vec{u} - \brac{\partial_t P\, P^T + P \Omega^0 P^T}P\partial_a \vec{u}\\
&+P\brac{\Omega^a \partial_t \vec{u} + \Omega^0 \partial_a \vec{u}}\\
\overset{\eqref{eq:OmegaPalpha}, \eqref{eq:omeacommdu}}{=}&
-\Omega_P^a\, P \partial_t \vec{u} - \Omega_P^0\, P\partial_a \vec{u}+P\, G_{2;a,0}(\vec{u})\\
\end{split}
\]
That is, we have established \eqref{eq:Wwavet}.

For $\beta = b\in \{1,\ldots,N\}$ we observe
\[
\begin{split}
 &\partial_t P\, \partial_b \vec{u}- \partial_b P\, \partial_t \vec{u}\\
 =&\brac{\partial_t P\, P^T + P \Omega^0 P^T}\, P\partial_b \vec{u}+ \brac{- \partial_b P\, P^T + P \Omega^b P^T}P\partial_t \vec{u}\\
 &-P\brac{\Omega^0 \partial_b \vec{u} +  \Omega^b \partial_t \vec{u}}\\
 \overset{\eqref{eq:OmegaPalpha}, \eqref{eq:omeacommdu}}{=}& 
 \Omega_P^0 \, P\partial_b \vec{u}+ \Omega_P^b\, P\partial_t \vec{u}-P\, G_{2;b,0}(\vec{u}).
 \end{split}
\]
Also
\[
\begin{split}
&\partial_a P\, \partial_b \vec{u} - \partial_b P\, \partial_a \vec{u}\\
=&\brac{\partial_a P\, P^T  -P\Omega^a P^T}P \partial_b \vec{u} - \brac{\partial_b P\, P^T - P \Omega^b P^T}P\partial_a \vec{u}\\
&+P\brac{\Omega^a \partial_b \vec{u} - \Omega^b \partial_a \vec{u}}\\
\overset{\eqref{eq:OmegaPalpha}, \eqref{eq:omeacommdu}}{=}&
-\Omega_P^a\, P \partial_b \vec{u} + \Omega_P^b\, P\partial_a \vec{u}+P\, G_{2;a,b}(\vec{u})\\
\end{split}
\]
Thus, we have \eqref{eq:Wwaveb}.

We can conclude
\end{proof}

We also have control of the initial data of $W$:
\begin{lemma}\label{la:Winitial}
For $\eps$ in \eqref{eq:ap:smallnessdd} small enough we have
\[
 \|\Ds{\frac{d-2}{2}}\vec{W}(0)\|_{L^2(\R^d)} + \|\Ds{\frac{d-4}{2}}\partial_t \vec{W}(0)\|_{L^2(\R^d)} \aleq \|\Ds{\frac{d}{2}}\vec{u_0}\|_{L^2(\R^d)} + \|\Ds{\frac{d-2}{2}}\vec{\dot{u}_0}\|_{L^2(\R^d)}.
\]
\end{lemma}
\begin{proof}
We observe that by the choice of $\vec{v}(0) = \partial_t \vec{v} (0) = 0$ we have,
\[
 W_i^a \Big|_{t=0} = P_{ij} \partial_a \vec{u_0}, \quad W_i^0 \Big|_{t=0} = P_{ij} \vec{\dot{u}_0}
\]
Then in view of \Cref{la:fullOmegaPLdest} and \eqref{eq:ap:smallnessdd}
\[
\begin{split}
 \|\Ds{\frac{d-2}{2}} W\|_{L^2(\R^d)} \aleq& \max_{s \in [0,\frac{d-2}{2}]} \|\Ds{s} P\|_{L^{\frac{d}{s}}(\R^d)}\, \brac{\|\Ds{\frac{d-2}{2}-s} \nabla \vec{u_0}\|_{L^{\frac{2d}{d-2s}}(\R^d)} + \|\Ds{\frac{d-2}{2}-s} \vec{\dot{u}_0}\|_{L^{\frac{2d}{d-2s}}(\R^d)}}\\
 \aleq&\brac{\|P\|_{L^\infty} + \|\Ds{\frac{d-2}{2}} P\|_{L^{2}(\R^d)}} \brac{\|\Ds{\frac{d}{2}}\vec{u_0}\|_{L^2(\R^d)} + \|\Ds{\frac{d-2}{2}}\vec{\dot{u}_0}\|_{L^2(\R^d)}}\\
 \aleq&\|\Ds{\frac{d}{2}}\vec{u_0}\|_{L^2(\R^d)} + \|\Ds{\frac{d-2}{2}}\vec{\dot{u}_0}\|_{L^2(\R^d)}.
 \end{split}
\]
Next, we have
\[
 \partial_t W^a \Big |_{t=0}= \partial_t (P \partial_a \vec{u}) \Big |_{t=0}= \partial_{t} P \partial_a \vec{u}\Big |_{t=0} + P \partial_{a} \partial_t \vec{u}\Big |_{t=0}.
\]
Consequently,
\[
\begin{split}
 &\|\Ds{\frac{d-4}{2}} \partial_t W^a(0)\|_{L^2(\R^d)}\\
 \aleq& \|\Ds{\frac{d-4}{2}} \brac{\partial_t P\, \partial_a \vec{u_0}}\|_{L^2(\R^d)} + \|\Ds{\frac{d-4}{2}} \brac{P\, \partial_a \vec{\dot{u}_0}}\|_{L^2(\R^d)}\\
 \aleq&\max_{s \in [0,\frac{d-4}{2}]} \brac{\|\Ds{s} dP\|_{L^{\frac{d}{s+1}}(\R^d)} \|\Ds{\frac{d-2}{2}-s} \vec{u_0}\|_{L^{\frac{2d}{d-2(s+1)}(\R^d)}}
  + \|\Ds{s} P\|_{L^{\frac{d}{s}}(\R^d)}\, \|\Ds{\frac{d-2}{2}-s}\vec{\dot{u}_0}\|_{L^{\frac{2d}{d-2s}}(\R^d)} }\\
  \aleq&\brac{\|P\|_{L^\infty} + \|\Ds{\frac{d-2}{2}} dP\|_{L^{2}(\R^d)}} \brac{\|\Ds{\frac{d}{2}} \vec{u_0}\|_{L^{2}(\R^d)}
  + \|\Ds{\frac{d-2}{2}}\vec{\dot{u}_0}\|_{L^{2}(\R^d)}}.\\
 \end{split}
\]
So again in view of \Cref{la:fullOmegaPLdest} and \eqref{eq:ap:smallnessdd},
\[
\max_{a \in \{1,\ldots,d\}} \|\Ds{\frac{d-4}{2}} \partial_t W^a(0)\|_{L^2(\R^d)}
  \aleq \|\Ds{\frac{d}{2}} \vec{u_0}\|_{L^{2}(\R^d)}
  + \|\Ds{\frac{d-2}{2}}\vec{\dot{u}_0}\|_{L^{2}(\R^d)}.
\]
Lastly, we observe that from the definition of $\vec{v}$, i.e. \eqref{eq:ap:veq} and \eqref{eq:ap:Pvueq},
\[
 \brac{\partial_{tt} - \lap} \vec{v} = \partial_{t} (P \partial_t \vec{u}) - \div(P \nabla \vec{u}),
\]
that is
\[
\partial_{t} (P \partial_t \vec{u}) - \partial_{tt}  \vec{v} = \div(P \nabla \vec{u})- \lap \vec{v}.
\]
Consequently,
\[
\begin{split}
 \partial_t W^0 =& \partial_t (P \partial_t \vec{u})-\partial_{tt} \vec{v}\\
 =&\div(P \nabla \vec{u})- \lap \vec{v}.
 \end{split}
\]
Since $\vec{v}(0) = 0$ we conclude
\[
 \partial_t W^0 \Big |_{t=0}= \div(P \nabla \vec{u})\Big |_{t=0} = \nabla P\Big |_{t=0} \cdot \nabla \vec{u_0} + P\Big |_{t=0} \lap \vec{u_0}.
\]
But then
\[
\begin{split}
 &\|\Ds{\frac{d-4}{2}} W^0(0)\|_{L^2(\R^d)}\\
 \aleq& \max_{s \in [0,\ldots,\frac{d-4}{2}]} \brac{\|\Ds{1+s}P\|_{L^{\frac{d}{1+s}}(\R^d)}\, \|\Ds{\frac{d-2}{2}-s} \vec{u}_0\|_{L^{\frac{2d}{d-2(1+s)}}(\R^d)} + \|\Ds{s} P\|_{L^{\frac{d}{s}}(\R^d)}\, \|\Ds{\frac{d}{2}-s}\vec{u_0}\|_{L^{\frac{2d}{d-2s}}(\R^d)}}\\
 \aleq& \brac{\|P\|_{L^\infty(\R^d)} + \|\Ds{\frac{d}{2}} P\|_{L^2(\R^d)}}\, \|\Ds{\frac{d}{2}} \vec{u_0}\|_{L^{2}(\R^d)}\\
 \aleq&\|\Ds{\frac{d}{2}} \vec{u_0}\|_{L^{2}(\R^d)}.
 \end{split}
\]
In the last step we used again \Cref{la:fullOmegaPLdest} and \eqref{eq:ap:smallnessdd}.

We can conclude.
\end{proof}

As a consequence of \Cref{la:ap:Wwave}  and \Cref{la:Winitial} we obtain
\begin{lemma}[Estimates of $W$]\label{la:estimatesw}
Let $s \in (1,\frac{d^2-4d+1}{2(d-1)})$. There exists $\eps > 0$ such that if \eqref{eq:ap:smallnessdd} is satisfied with that $\eps$,
\begin{equation}\label{eq:estimatew:goal}
\begin{split}
 \|\Ds{\frac{d-2}{2}} W\|_{C^0_t L^2(\R^d)} &+ \|\Ds{s-1} W\|_{L^2_t L^{(\frac{2d}{2s-1},2)}_x(\R^d) \times (0,T)}\\
 \aleq_s &\|\Ds{s-2} dd\vec{u}\|_{L^2_t L^{(\frac{2d}{2s-1},2)}_x(\R^d \times (-T,T))}^2\\
 &+\|\Ds{\frac{d}{2}} \vec{u_0}\|_{L^2(\R^d)} + \|\Ds{\frac{d-2}{2}} \vec{\dot{u}_0}\|_{L^2(\R^d)}.
 \end{split}
 \end{equation}
\end{lemma}
\begin{proof}
From Strichartz estimates, \Cref{la:strichartz}, taking into account also \Cref{la:Winitial},
\[
\begin{split}
 &\|\Ds{\frac{d-4}{2}} dW\|_{C^0_t L^2(\R^d)} + \|\Ds{s-1} W\|_{L^2_t L^{(\frac{2d}{2s-1},2)}_x(\R^d) \times (0,T)}\\
 \aleq&\|\Ds{\frac{d-2}{2}} W(0)\|_{L^2(\R^d)} + \|\Ds{\frac{d-4}{2}} \partial_t W(0)\|_{L^2(\R^d)}+\|\Ds{\frac{d-4}{2}} \brac{(\partial_{tt}-\lap) W)} \|_{L^1_t L^2_x(\R^d \times (-T,T))}\\
 \aleq&\|\Ds{\frac{d}{2}}\vec{u_0}\|_{L^2(\R^d)} + \|\Ds{\frac{d-2}{2}}\vec{\dot{u}_0}\|_{L^2(\R^d)}\\
 &+\|\Ds{\frac{d-4}{2}} \brac{(\partial_{tt}-\lap) W)} \|_{L^1_t L^2_x(\R^d \times (-T,T))}.
 \end{split}
\]
In view of \Cref{la:ap:Wwave} we have for fixed time $t \in (-T,T)$
\[
\begin{split}
 &\|\Ds{\frac{d-4}{2}} (\partial_{tt} - \lap) \vec{W}\|_{L^2(\R^d)}\\
 \aleq& \max_{\alpha,\beta \in \{0,\ldots,d\}}\|\Ds{\frac{d-2}{2}}\brac{\Omega_P^\alpha\, P^T \partial_\beta \vec{u} }\|_{L^2 (\R^d)}\\
 &+ \max_{\alpha,\beta \in \{0,\ldots,d\}}\|\Ds{\frac{d-4}{2}}d\brac{P G_{2;\alpha,\beta}}\|_{L^2 (\R^d)}\\
 &+\|\Ds{\frac{d-4}{2}} \partial_t (\Omega_P^0 \, P\partial_b \vec{u})\|_{L^2(\R^d)}+ \|\Ds{\frac{d-4}{2}} \partial_t (\Omega_P^b\, P\partial_t \vec{u})\|_{L^2(\R^d)}\\
 \overset{\eqref{eq:adasd2}}{\aleq}&  \|\Ds{s-2} dd\vec{u}\|_{L^{(\frac{2d}{2s-1},2)}(\R^d)}^2\\
 &+ \max_{\alpha,\beta \in \{0,\ldots,d\}}\|\Ds{\frac{d-4}{2}}d\brac{P G_{2;\alpha,\beta}}\|_{L^2 (\R^d)}\\
 \end{split}
\]
As before when obtaining the estimate \eqref{eq:Festasdkjasd} we observe that by \Cref{la:fullOmegaPLdest},
\[
\begin{split}
 &\|\Ds{\frac{d-4}{2}}d\brac{P G_{2;\alpha,\beta}}\|_{L^2 (\R^d)}\\
 \aleq&\brac{\|P\|_{L^\infty} + \|\Ds{\frac{d-2}{2}}dP\|_{L^d(\R^d)}} \|\Ds{\frac{d-4}{2}}dG_{2;\alpha,\beta}\|_{L^2 (\R^d)}\\
 \overset{\eqref{eq:G2cond:1}, \eqref{eq:411Ssmall}}{\aleq}&\gamma \|d \vec{u}\|_{L^{(2d,2)}(\R^d)}^2\\
\end{split}
 \]
Consequently, we have shown that for each fixed $t \in (-T,T)$ we have
\[
 \|\Ds{\frac{d-4}{2}} (\partial_{tt} - \lap) \vec{W}\|_{L^2(\R^d)}  \aleq \gamma \|\Ds{s-2} dd\vec{u}\|_{L^{(\frac{2d}{2s-1},2)}(\R^d)}^2
\]
Integrating this in time we find \eqref{eq:estimatew:goal}. We can conclude.
\end{proof}

Combining \Cref{la:estimatesw}, \Cref{la:estimatesv} we have 
\begin{proof}[Proof of \Cref{th:aprioriest}]
From \eqref{eq:WPuv}
\[
d\vec{u} = P^T \vec{W}+ P^t d\vec{v},
\]
Thus for each fixed time $t \in (-T,T)$
\begin{equation}\label{eq:ddvecuL2est}
\begin{split}
 \|\Ds{\frac{d-2}{2}} d\vec{u}\|_{L^2(\R^d)} \aleq& \brac{\|P\|_{L^\infty} + \|\Ds{\frac{d}{2}} P\|_{L^{2}(\R^d)}} \brac{\|\Ds{\frac{d-2}{2}} \vec{W}\|_{L^2(\R^d)} + \|\Ds{\frac{d-2}{2}} d \vec{v}\|_{L^2(\R^d)}}\\
 \aleq&\|\Ds{\frac{d-2}{2}} \vec{W}\|_{L^2(\R^d)} + \|\Ds{\frac{d-2}{2}} d \vec{v}\|_{L^2(\R^d)},
 \end{split}
\end{equation}
where in the last step we used \Cref{la:fullOmegaPLdest} and \eqref{eq:ap:smallnessdd}.

Similarly,
\begin{equation}\label{eq:ddvecuL2dest}
\begin{split}
\|\Ds{s-1}d \vec{u}\|_{L^{(\frac{2d}{2s-1},2)}_x(\R^d)} \aleq& \brac{\|P\|_{L^\infty(\R^d)} + \|\Ds{\frac{d}{2}} P\|_{L^2(\R^d)}} \brac{\|\Ds{s-1} \vec{W}\|_{L^{(\frac{2d}{2s-1},2)}_x(\R^d)} + \|\Ds{s-1} d\vec{v}\|_{L^{(\frac{2d}{2s-1},2)}_x(\R^d)}}\\
  \aleq&\|\Ds{s-1} \vec{W}\|_{L^{(\frac{2d}{2s-1},2)}_x(\R^d)} + \|\Ds{s-1} d\vec{v}\|_{L^{(\frac{2d}{2s-1},2)}_x(\R^d)}\\
\end{split}
\end{equation}

Combining \eqref{eq:ddvecuL2est}, \eqref{eq:ddvecuL2dest} with \eqref{eq:estimatev:goal} and \eqref{eq:estimatew:goal} we obtain
\begin{equation}\label{eq:contmethaskldjsad}
\begin{split}
 &\|\Ds{\frac{d-2}{2}} d\vec{u}\|_{L^\infty_t L^2(\R^d\times (-T,T))} + \|\Ds{s-1}d \vec{u}\|_{L^2_t L^{(\frac{2d}{2s-1},2)}_x(\R^d \times (-T,T))}\\
 \aleq& \|\Ds{s-1} d\vec{u}\|_{L^2_t L^{(\frac{2d}{2s-1},2)}_x(\R^d \times (-T,T))}^2\\
 &+\|\Ds{\frac{d}{2}} \vec{u_0}\|_{L^2(\R^d)} + \|\Ds{\frac{d-2}{2}} \vec{\dot{u}_0}\|_{L^2(\R^d)}
 \end{split}
\end{equation}
In particular, if we set
\[
 A(\tau) := \|\Ds{s-1}d \vec{u}\|_{L^2_t L^{(\frac{2d}{2s-1},2)}_x(\R^d \times (-\tau,\tau))}
\]
then $\tau \mapsto A(\tau)$ is continuous (by absolute continuity of the integral), and we have for any $\tau \leq T$
\[
 A(\tau) \leq C\, A(\tau)^2 + C\, B,
\]
where
\[
B:= \brac{\|\Ds{\frac{d}{2}} \vec{u_0}\|_{L^2(\R^d)} + \|\Ds{\frac{d-2}{2}} \vec{\dot{u}_0}\|_{L^2(\R^d)}}.
\]
If $B < \frac{1}{4C^2}$ then we conclude from the polynomial inequality above
\[
 \text{either }A(\tau) \leq \frac{1}{2C} - \sqrt{\frac{1}{4C^2} - B} \quad \text{or} \quad A(\tau) \geq \frac{1}{2C} - \sqrt{\frac{1}{4C^2} - B}
\]
Observe that the first option is equivalent to
\[
\begin{split}
 &\sqrt{\frac{1}{4C^2} - B} \leq \frac{1}{2C} - A(\tau) \\
 \Leftrightarrow&\frac{1}{4C^2} - B \leq \frac{1}{4C^2} + \brac{A(\tau)}^2  -\frac{1}{C} A(\tau)\\
 \Leftrightarrow& \frac{1}{C} A(\tau) \leq  B+ \brac{A(\tau)}^2 \\
 \end{split}
\]
Clearly for $\tau = 0$ this is satisfied (since $C \geq 1$). Since $\tau \mapsto A(\tau)$ is continuous, we find that
\[
 A(\tau) \leq \frac{1}{2C} - \sqrt{\frac{1}{4C^2} - B} \quad \forall \tau \in (0,T]
\]
and in particular we have
\[
 A(T) \aleq B.
\]
Thus, \eqref{eq:contmethaskldjsad} becomes
\[
\begin{split}
 &\|\Ds{\frac{d-2}{2}} d\vec{u}\|_{L^\infty_t L^2(\R^d\times (-T,T))} + \|\Ds{s-1}d \vec{u}\|_{L^2_t L^{(\frac{2d}{2s-1},2)}_x(\R^d \times (-T,T))}\\
 \aleq& (\|\Ds{\frac{d}{2}} \vec{u_0}\|_{L^2(\R^d)} + \|\Ds{\frac{d-2}{2}} \vec{\dot{u}_0}\|_{L^2(\R^d)})^2\\
 &+\|\Ds{\frac{d}{2}} \vec{u_0}\|_{L^2(\R^d)} + \|\Ds{\frac{d-2}{2}} \vec{\dot{u}_0}\|_{L^2(\R^d)}\\
 \aleq&\|\Ds{\frac{d}{2}} \vec{u_0}\|_{L^2(\R^d)} + \|\Ds{\frac{d-2}{2}} \vec{\dot{u}_0}\|_{L^2(\R^d)}.
 \end{split}
\]
That is, we have established \eqref{eq:apriori:est} and can conclude.
\end{proof}

\section{Halfwave map and wave map equation fit into Theorem~\ref{th:aprioriest}}\label{s:halfwavewave}
\subsection{Wave map equation}
Firstly, we observe that our assumptions are applicable to the classical wave map equation
\[
 \partial_{tt} \vec{u} - \lap \vec{u} \perp T_{\vec{u}} \S^{2}.
\]
That is
\[
 \partial_{tt} \vec{u} - \lap \vec{u} = \vec{u} \langle \vec{u}, \partial_{tt} \vec{u} - \lap \vec{u}\rangle
\]
Set
\[
 \Omega^{\alpha}_{ij} := \begin{cases}
                          u^i \partial_a u^j - u^j \partial_a u^i \quad& a=\alpha \in \{1,\ldots,n\}\\
                         -\brac{u^i \partial_t u^j - u^j \partial_t u^i} \quad& \alpha =0\\
                         \end{cases}
\]
Then we observe that
\[
\begin{split}
 &u^i \langle \vec{u}, \partial_{tt} \vec{u} - \lap \vec{u}\rangle \\
 =&-u^i \brac{\langle \partial_t \vec{u}, \partial_{t} \vec{u}\rangle- \sum_{a=1}^d\langle \partial_a \vec{u},\partial_a \vec{u}\rangle} \\
 =&-u^i \brac{\partial_t u^j\, \partial_{t} u^j- \sum_{a=1}^d \partial_a u^j\,\partial_a u^j} \\
 =&-\brac{u^i \partial_t u^j\, -u^j \partial_t u^i}\partial_{t} u^j\\
 &+ \sum_{a=1}^d  \brac{u^i\partial_a u^j -u^j \partial_a u^i} \partial_a u^j \\
 =&\brac{\Omega \cdot d\vec{u}}_i
 \end{split}
\]
So we can choose $F \equiv 0$ in \eqref{eq:generalwave}, which is clearly admissible in the sense of \Cref{def:admissibleF}.

As for $\Omega$ we now show
\begin{proposition}\label{pr:wavemap:omegaadmissible}
$\Omega$ as above for the wave maps into $\S^2$ satisfies the conditions of \Cref{def:admissibleomega}.
\end{proposition}

Indeed, it is clearly antisymmetric, $\Omega_{ij}^\alpha = -\Omega_{ji}^\alpha$.

\begin{lemma}
$\Omega$ satisfies the growth conditions \eqref{eq:Omegacond:growthpq}, \eqref{eq:Omegacond:growthpq2}
\end{lemma}
\begin{proof}
We have for $p < \frac{d}{s+1}$
\[
\begin{split}
 &\|\Ds{s} d\Omega\|_{L^{(p,q)}(\R^d)}\\
 \aleq& \|\Ds{s} (\vec{u}\, dd\vec{u})\|_{L^{(p,q)}(\R^d)}+\|\Ds{s} (d\vec{u}\, d\vec{u})\|_{L^{(p,q)}(\R^d)}\\
  \aleq& \max_{t \in [0,s]} \|\Ds{t} \vec{u}\|_{L^{\frac{d}{t}}(\R^d)} \|\Ds{s-t}dd\vec{u}\|_{L^{(\frac{dp}{d-tp},q)}(\R^d)}+\|\Ds{t} d\vec{u}\|_{L^{\frac{d}{t+1}}(\R^d)}\, \|\Ds{s-t} d\vec{u}\|_{L^{\frac{dp}{d-(t+1)p}}(\R^d)}\\
  \aleq& \brac{1+\|\Ds{\frac{d-2}{2}} d\vec{u}\|_{L^2(\R^d)}} \|\Ds{s}dd\vec{u}\|_{L^{(p,q)}(\R^d)}\\
    \aleq& \brac{1+\|\Ds{\frac{d-2}{2}} d\vec{u}\|_{L^2(\R^d)}} \brac{\|\Ds{s+1}d\vec{u}\|_{L^{(p,q)}(\R^d)}+\|\Ds{s}\partial_{tt}\vec{u}\|_{L^{(p,q)}(\R^d)}}\\
 \end{split}
\]
Similar to the proof of \Cref{la:dealingwithddu}, we observe
\[
 \partial_{tt} \vec{u} = \lap \vec{u} + \vec{u} |d\vec{u}|^2
\]
So,
\[
\begin{split}
&\|\Ds{s}\partial_{tt}\vec{u}\|_{L^{(p,q)}(\R^d)}\\
\aleq& \|\Ds{s+1}d\vec{u}\|_{L^{(p,q)}(\R^d)} + \|\Ds{s} (\vec{u} |d\vec{u}|^2)\|_{L^{(p,q)}(\R^d)}\\
\aleq& \|\Ds{s+1}d\vec{u}\|_{L^{(p,q)}(\R^d)} + \max_{t_1+t_2+t_3 =s,\ t_i \geq 0} \|\Ds{t_1}\vec{u}\|_{L^{\frac{d}{t_1}}(\R^d)} \|\Ds{t_2} d\vec{u} \|_{L^{\frac{d}{t_2+1}}(\R^d)}\, \| \Ds{t_3} d\vec{u}\|_{L^{(\frac{dp}{d-(t_1+t_2+1)p},q)}(\R^d)}\\
\aleq&(1+\|\Ds{\frac{d-2}{2}} d\vec{u}\|_{L^2(\R^d)}^2) \|\Ds{s+1}d\vec{u}\|_{L^{(p,q)}(\R^d)}.
\end{split}
\]
We conclude that
\[
 \|\Ds{s} d\Omega\|_{L^{(p,q)}(\R^d)} \aleq (1+\|\Ds{\frac{d-2}{2}} d\vec{u}\|_{L^2(\R^d)}^3)\, \|\Ds{s+1}d\vec{u}\|_{L^{(p,q)}(\R^d)},
\]
which establishes \eqref{eq:Omegacond:growthpq}.

The proof of \eqref{eq:Omegacond:growthpq2} is almost verbatim only without the $\partial_{tt}$ term (and essentially for $s-1$ instead of $s$, leading to the assumption $p < \frac{d}{s}$ in this case).
\end{proof}

\begin{lemma}
$\Omega$ satisfies the curl-conditions \eqref{eq:G1cond:1Lpq}, \eqref{eq:G1cond:1}, \eqref{eq:G1cond:3}, \eqref{eq:G1cond:3sigma}, \eqref{eq:G1cond:2},\eqref{eq:G1cond:GN1spatial} \eqref{eq:G1cond:GN1}.
\end{lemma}
\begin{proof}
We observe
\[
\begin{split}
 \partial_a \Omega^b - \partial_b \Omega^a =& \partial_a u^i\, \partial_b u^j -\partial_b u^i\, \partial_a u^j - \brac{\partial_a u^j\, \partial_b u^i -\partial_b u^j\, \partial_a u^i}\\
 =&2 \partial_a u^i\, \partial_b u^j -2\partial_b u^i\, \partial_a u^j,
 \end{split}
\]
and similarly
\[
\begin{split}
 \partial_a \Omega^0 + \partial_t \Omega^a =& -\partial_a (u^i \partial_t u^j - u^j \partial_t u^i) + \partial_t (u^i \partial_a u^j - u^j \partial_a u^i)\\
  =&-\partial_a u^i \partial_t u^j + \partial_a u^j \partial_t u^i + \partial_t u^i \partial_a u^j - \partial_t u^j \partial_a u^i\\
  =&-2\partial_a u^i \partial_t u^j + 2\partial_a u^j \partial_t u^i\\
 \end{split}
\]
So elements of $G_{1;\alpha,\beta}(\vec{u})$ are of the form $\partial_\alpha u^i\, \partial_\beta u^j$.

\underline{As for \eqref{eq:G1cond:1Lpq}} we observe for $1 \leq \tilde{s}\leq d$, $1<p < \frac{2d}{\tilde{s}+2}$
\[
 \begin{split}
&\|\Ds{\tilde{s}} \brac{\partial_\alpha u\, \partial_\beta v}\|_{L^{(p,q)}(\R^d)}\\
\aleq& \max_{t \in [0,\frac{\tilde{s}}{2}]} \|\Ds{t} \partial_\alpha u\|_{L^{\frac{d}{t+1}}(\R^d)}\, \|\Ds{\tilde{s}-t} \partial_\beta v\|_{L^{(\frac{dp}{d-(t+1)p},q)}(\R^d)}\\
&+\max_{t \in [\frac{\tilde{s}}{2},\tilde{s}]} \|\Ds{t} \partial_\alpha  u\|_{L^{(\frac{dp}{d-(1+\tilde{s}-t)p,q)}}(\R^d)}\, \|\Ds{\tilde{s}-t} \partial_\beta v\|_{L^{\frac{d}{1+\tilde{s}-t}}(\R^d)}\\
\aleq& \|\Ds{\frac{d-2}{2}} d u\|_{L^{2}(\R^d)}\, \|\Ds{\tilde{s}+1} dv\|_{L^{(p,q)}(\R^d)}\\
&+\|\Ds{\tilde{s}+1} du\|_{L^{(p,q)}(\R^d)}\, \|\Ds{\frac{d-2}{2}} dv\|_{L^{2}(\R^d)}\\
 \end{split}
\]
This readily establishes \eqref{eq:G1cond:1Lpq} for $s := \tilde{s}+1$,

\underline{As for \eqref{eq:G1cond:1},} since $d \geq 4$ we have from Sobolev embedding,
\[
\begin{split}
 &\|\Ds{\frac{d-6}{2}}d (\partial_\alpha u\ \partial_{\beta} v)\|_{L^{2}(\R^d)} \\
 \aleq&\|\Ds{\frac{d-4}{2}}d (\partial_\alpha u\ \partial_{\beta} v)\|_{L^{\frac{2d}{d+2}}(\R^d)}\\
 \aleq&\max_{t \in [0,\frac{d-4}{2}]}\|\Ds{t}ddu\|_{L^{\frac{2d}{4+2t}}(\R^d)} \|\Ds{\frac{d-4}{2}-t} dv\|_{L^{\frac{2d}{d-2-2t}}(\R^d)}\\
 &+\max_{t \in [0,\frac{d-4}{2}]}\|\Ds{t}du\|_{L^{\frac{2d}{2+2t}}(\R^d)} \|\Ds{\frac{d-4}{2}-t} ddv\|_{L^{\frac{2d}{d-2t}}(\R^d)}\\
 \aleq&\brac{\|\Ds{\frac{d-4}{2}} ddu\|_{L^2(\R^d)}+\|\Ds{\frac{d-4}{2}} ddv\|_{L^2(\R^d)}}\, \brac{\|\Ds{\frac{d-2}{2}} du\|_{L^2(\R^d)}+\|\Ds{\frac{d-2}{2}} dv\|_{L^2(\R^d)}}
\end{split}
\]

\underline{As for \eqref{eq:G1cond:3} and \eqref{eq:G1cond:3sigma},} by H\"older's inequality,
\[
\begin{split}
  &\|d\brac{\partial_\alpha u\ \partial_\beta v}\|_{L^{(\frac{2d}{5-\sigma},q)}(\R^d)}\\
  \aleq\|ddu\|_{L^{\frac{2d}{4-\sigma}}(\R^d)}\, \|dv\|_{L^{(2d,q)}(\R^d)}+\|ddv\|_{L^{\frac{2d}{4-\sigma}}(\R^d)}\, \|du\|_{L^{(2d,q)}(\R^d)}.
  \end{split}
\]
This readily proves \eqref{eq:G1cond:3} by taking $\sigma = 0$. For \eqref{eq:G1cond:3sigma}, apply \eqref{eq:GN:1}, whenever $(1-\sigma) \in [\frac{2}{d-2},1]$,
\begin{equation}\label{eq:wm:askldalskjd}
 \|dd\vec{u}\|_{L^{\frac{2d}{4-\sigma}}(\R^d)} = \|\Dso \Ds{-1}dd\vec{u}\|_{L^{\frac{2d}{4-\sigma}}(\R^d)}\aleq \|\Ds{-1}dd\vec{u}\|_{L^{2d}(\R^d)}^\sigma\ \|\Ds{\frac{d-4}{2}} dd\vec{u}\|_{L^2(\R^d)}^{1-\sigma}.
\end{equation}
Thus
\[
\begin{split}
  &\|dG_{1;\alpha,\beta}(\vec{u})\|_{L^{(\frac{2d}{5-\sigma},q)}(\R^d)}\\
  \aleq&\|\Ds{-1}dd\vec{u}\|_{L^{2d}(\R^d)}^\sigma\ \|\Ds{\frac{d-4}{2}} dd\vec{u}\|_{L^2(\R^d)}^{1-\sigma} \, \|d\vec{u}\|_{L^{(2d,q)}(\R^d)}
  \end{split}
\]

\underline{As for \eqref{eq:G1cond:2},} similar to the previous argument
\[
  \|\partial_\alpha u\ \partial_\beta v\|_{L^{(d,1)}(\R^d)}
  \aleq\|du\|_{L^{(2d,2)}(\R^d)}\, \|dv\|_{L^{(2d,2)}(\R^d)}
\]

\underline{As for \eqref{eq:G1cond:GN1spatial},} it is a direct consequence of \eqref{eq:GN:234}.

\underline{As for \eqref{eq:G1cond:GN1},} assume $(1-\sigma) \in [\frac{2(s-1)}{d-2},1]$ and $s \in [2, \frac{d-2}{2}]$.

If $s=2$ we observe
\[
\begin{split}
&\|\partial_\alpha du\ \partial_\beta v\|_{L^{\frac{2d}{5-\sigma}}(\R^d)}\\
\aleq&\|\partial_\alpha du\|_{L^{\frac{2d}{4-\sigma}}(\R^d)} \|\partial_\beta v\|_{L^{2d}(\R^d)}\\
\overset{\eqref{eq:wm:askldalskjd}}{\aleq}& \|\Ds{-1}dd\vec{u}\|_{L^{2d}(\R^d)}^\sigma\ \|\Ds{\frac{d-4}{2}} dd\vec{u}\|_{L^2(\R^d)}^{1-\sigma}\|dv\|_{L^{2d}(\R^d)}\\
\end{split}
\]
This proves the claim for $s = 2$ and $1-\sigma \in [\frac{2}{d-2},1]$.

If $s > 2$
\[
\begin{split}
&\|\Ds{s-2} \brac{\partial_\alpha du\ \partial_\beta v}\|_{L^{\frac{2d}{1+2s-\sigma}}(\R^d)}\\
\aleq&\max_{t \in [0,s-2]}\| \Ds{t+1}\Ds{-1}ddu\|_{L^{\frac{2d}{1+2(1+t)+\sigma_1}}(\R^d)} \|\Ds{s-2-t}d v\|_{L^{\frac{2d}{1+2(s-2-t)+\sigma_2 }}(\R^d)}\\
\overset{\eqref{eq:GN:1}}{\aleq}& \max_{t \in [0,s-2]}\|\Ds{-1}ddu\|_{L^{2d}(\R^d)}^{1-\sigma_1} \|\Ds{\frac{d-4}{2}} ddu\|_{L^2(\R^d)}^{\sigma_1}
\|dv\|_{L^{2d}(\R^d)}^{1-\sigma_2} \|\Ds{\frac{d-2}{2}} dv\|_{L^2(\R^d)}^{\sigma_2}
\end{split}
\]
for $\sigma_1 + \sigma_2 = 1-\sigma$, where we must ensure $\sigma_1 \in [\frac{2(t+1)}{d-2},1]$, $\sigma_2 \in [\frac{2(s-2-t)}{d-2},1]$ which is possible since $(1-\sigma) \in [\frac{2(s-1)}{d-2},1]$. We can conclude.
\end{proof}

Lastly, we observe
\[
\begin{split}
 &\brac{\Omega^a \partial_b \vec{u} - \Omega^b \partial_a \vec{u}}_i\\
 =&u^i \partial_a u^j \partial_b u^j
 - \underbrace{u^j \partial_a u^i\partial_b u^j}_{=0}
 - u^i \partial_b u^j \partial_a u^j \underbrace{+ u^j \partial_b u^i \partial_a u^j}_{=0}\\
 =&u^i \partial_a u^j \partial_b u^j - u^i \partial_b u^j \partial_a u^j\\
 =&0
 \end{split}
\]
and similarly
\[
\begin{split}
 &\brac{\Omega^a \partial_t \vec{u} + \Omega^0 \partial_a \vec{u}}_i\\
 =&u^i \partial_a u^j \partial_t u^j- \underbrace{u^j \partial_a u^i}_{=0}\, \partial_t u^j-u^i \partial_t u^j\, \partial_a u^j + \underbrace{u^j \partial_t u^i\, \partial_a u^j}_{=0}\\
 =&0
 \end{split}
\]
That is the cancellation condition \eqref{eq:omeacommdu}, \eqref{eq:G2cond:1} is trivially satisfied.

That is, \Cref{pr:wavemap:omegaadmissible} is established.

\subsection{Halfwave map equation}
The other application is the halfwave map equation for $\vec{u}: \R^d \times [-T,T] \to \S^2$
\[
 \partial_{t} \vec{u} - \vec{u} \wedge \laps{1} \vec{u} = 0.
\]

Krieger and Sire showed in \cite{KS18} that we can solve instead

\begin{equation}\label{eq:solks}
\begin{split}
 \partial_{tt} \vec{u} - \lap \vec{u} =& \vec{u} |\nabla \vec{u}|^2  - \vec{u} |\Dso{}\vec{u}|^2\\
 &+ \Dso{} \vec{u}\, \brac{\scpr{\vec{u}}{\Dso{} \vec{u}}}\\
 &+ \vec{u} \wedge \Dso{} \brac{ \vec{u} \wedge\Dso{} \vec{u}}- \vec{u} \wedge \brac{\vec{u} \wedge (-\lap) \vec{u}}.
 \end{split}
\end{equation}

We can write
\[
 \partial_{tt} \vec{u} - \lap \vec{u} = \vec{u} \langle\vec{u}, \partial_{tt} \vec{u} - \lap \vec{u}\rangle  + \Pi_{T_u \S^{2}} \brac{\partial_{tt} \vec{u} - \lap \vec{u}},
\]
As discussed above for the wave-map equation we can choose $\Omega$ admissible so that
\[
 \vec{u} \langle\vec{u}, \partial_{tt} \vec{u} - \lap \vec{u}\rangle = \Omega \cdot d\vec{u}.
\]
So we set $\Pi_{T_u \S^{2}} \vec{v} = \vec{v}-\vec{u} \langle \vec{u},\vec{v}\rangle$ and thus we set
\[
\begin{split}
 F(\vec{u}) :=& \Pi_{T_u \S^{2}} \brac{\partial_{tt} \vec{u} - \lap \vec{u}}\\
 =&\Dso{} \vec{u}\, \scpr{\vec{u}}{\Dso{} \vec{u}}-\vec{u}\, |\langle \vec{u}, \Dso \vec{u}\rangle|^2\\
 &+ \vec{u} \wedge \Dso{} \brac{ \vec{u} \wedge\Dso{} \vec{u}}- \vec{u} \wedge \brac{\vec{u} \wedge (-\lap) \vec{u}}
\end{split}
\]
where we got rid of time-derivatives $\partial_t \vec{u}$ by the half-wave map equation itself.

We need to prove $F$ is admissible in the sense of \Cref{def:admissibleF}, i.e. we need to prove the growth estimate \eqref{eq:Fcond:2} and the perturbative estimate \eqref{eq:Fcond:3}.  

We look at the terms of $F(\vec{u})$ one by one.
\begin{lemma}\label{la:term1Fest}
Let $\vec{u}: \R^d \to \S^2$,
then we have the corresponding growth estimate for \eqref{eq:Fcond:2}, for any $p < \frac{d}{t+1}$
\[
 \|\Ds{t}  \brac{\Dso \vec{u}\, \langle \vec{u}, \Dso \vec{u}\rangle }\|_{L^{p}(\R^d)} \aleq   \|\Ds{\frac{d}{2}} \vec{u}\|_{L^2(\R^d)}^2 \|\Ds{t+2}\vec{u}\|_{L^{p}(\R^d)},
\]
the corresponding estimate to \eqref{eq:Fcond:2v2}
\[
 \|\Dso \vec{u}\, \langle \vec{u}, \Dso \vec{u}\rangle \|_{L^{(\frac{2d}{3},2)}(\R^d)} \aleq   \|\Dso \vec{u}\|_{L^{2d}(\R^d)}\ \|\Ds{\frac{d}{2}} \vec{u}\|_{L^2(\R^d)}^2\\
\]

and the corresponding perturbative estimate for \eqref{eq:Fcond:3}
\[
 \|\Ds{\frac{d-2}{2}} \brac{\Dso \vec{u}\, \langle \vec{u}, \Dso \vec{u}\rangle }\|_{L^2} \aleq \|\Ds{\frac{d}{2}} \vec{u}\|_{L^2(\R^d)} \,  \|\Dso \vec{u}\|_{L^{(2d,2)}(\R^d)}^2,
\]
\end{lemma}
\begin{proof}
For the growth estimate corresponding to \eqref{eq:Fcond:2} we observe
\[
\begin{split}
 &\|\Ds{t}  \brac{\Dso \vec{u}\, \langle \vec{u}, \Dso \vec{u}\rangle }\|_{L^{p}(\R^d)}\\
 \overset{p < \frac{d}{t+1}}\aleq&\max_{t_1+t_2+t_3=t} \|\Ds{1+t_1} \vec{u}\|_{L^{\frac{dp}{d-(t+1-t_1)p}}(\R^d)}\ \|\Ds{1/2+t_2}\vec{u}\|_{L^{\frac{2d}{1+2t_2}}(\R^d)} \|\Ds{1/2+t_3} \vec{u}\|_{L^{\frac{2d}{1+2t_3}}(\R^d)}\\
 \aleq&\|\Ds{\frac{d}{2}} \vec{u}\|_{L^2(\R^d)}^2\, \|\Ds{t+2} \vec{u}\|_{L^{p}(\R^d)}.
 \end{split}
\]
If $t=0$ the same argument implies
\[
\begin{split}
 &\|\Ds{t}  \brac{\Dso \vec{u}\, \langle \vec{u}, \Dso \vec{u}\rangle }\|_{L^{p}(\R^d)}\\
 \aleq&\|\Dso \vec{u}\|_{L^{\frac{dp}{d-p}}(\R^d)}\ \|\Ds{\frac{d}{2}} \vec{u}\|_{L^2(\R^d)}^2\\
 \end{split}
\]
which for $p=\frac{2}{3}d$ implies the second claim corresponding to \eqref{eq:Fcond:2v2}.

The second perturbative estimate follows from \eqref{eq:gagliardo5}, observing that
\[
 \langle \vec{u}, \Dso \vec{u}\rangle  = H_{\Dso} \brac{\Ds{-1}\Dso\vec{u}\cdot,\Ds{-1}\Dso\vec{u})}.
\]
\end{proof}

\begin{lemma}
Let $\vec{u}: \R^d \to \S^2$, then we have the corresponding growth estimate for \eqref{eq:Fcond:2}, for any $p < \frac{d}{t+1}$
\[
 \|\Ds{t}  \brac{\vec{u}\, |\langle \vec{u}, \Dso \vec{u}\rangle|^2}\|_{L^{p}(\R^d)} \aleq   \brac{\|\Ds{\frac{d}{2}} \vec{u}\|_{L^2(\R^d)}^{2}+\|\Ds{\frac{d}{2}} \vec{u}\|_{L^2(\R^d)}^{4}} \|\Ds{t+2}\vec{u}\|_{L^{p}(\R^d)},
\]
the corresponding estimate to \eqref{eq:Fcond:2v2}
\[
 \|\brac{\vec{u}\, |\langle \vec{u}, \Dso \vec{u}\rangle|^2}\|_{L^{(\frac{2d}{3},2)}(\R^d)} \aleq   \|\Dso \vec{u}\|_{L^{2d}(\R^d)}\ \|\Ds{\frac{d}{2}} \vec{u}\|_{L^2(\R^d)}^2\\
\]
and the corresponding perturbative estimate for \eqref{eq:Fcond:3}
 \[
   \|\Ds{\frac{d-2}{2}} \brac{\vec{u}\, |\langle \vec{u}, \Dso \vec{u}\rangle|^2}\|_{L^2} \aleq \|\Ds{\frac{d}{2}} \vec{u}\|_{L^2(\R^d)}^{4} \,  \|\Dso \vec{u}\|_{L^{(2d,2)}(\R^d)}^2
\]
\end{lemma}
\begin{proof}
For the growth estimate, by writing
\[
\vec{u}\, |\langle \vec{u}, \Dso \vec{u}\rangle|^2 = \vec{u}\, \langle \vec{u}, \Dso \vec{u}\rangle  H_{\Dso} (\vec{u} \cdot,\vec{u}),
\]
we observe
\[
\begin{split}
  &\|\Ds{t}  \brac{\vec{u}\, |\langle \vec{u}, \Dso \vec{u}\rangle|^2}\|_{L^{p}(\R^d)}\\
  \aleq &\max_{\sum_{i=1}^5 t_i = t} \|\Ds{t_1} \vec{u}\|_{L^{\frac{d}{t_1}}} \|\Ds{t_2} \vec{u}\|_{L^{\frac{d}{t_2}}(\R^d)}  \|\Ds{1+t_3} \vec{u}\|_{L^{\frac{dp}{d-(t+1-t_3)p}}(\R^d)} \|\Ds{\frac{1}{2}+t_4} \vec{u}\|_{L^{\frac{2d}{1+t_4}}(\R^d)}\, \|\Ds{\frac{1}{2}+t_5} \vec{u}\|_{L^{\frac{2d}{1+t_5}}(\R^d)}\\
  \aleq& \brac{\|\vec{u}\|_{L^\infty}^2 \|\Ds{\frac{d}{2}} \vec{u}\|_{L^2(\R^d)}^2 + \|\vec{u}\|_{L^\infty} \|\Ds{\frac{d}{2}} \vec{u}\|_{L^2(\R^d)}^3 + \|\Ds{\frac{d}{2}} \vec{u}\|_{L^2(\R^d)}^4}\, \|\Ds{t+2} \vec{u}\|_{L^{p}(\R^d)}.
\end{split}
\]
The corresponding estimate to \eqref{eq:Fcond:2v2} follows as in the proof of \Cref{la:term1Fest} since 
\[\abs{\vec{u}\, |\langle \vec{u}, \Dso \vec{u}\rangle|^2} \aleq \abs{\Dso \vec{u}}\, \abs{\langle \vec{u}, \Dso \vec{u}\rangle}.\]
For the perturbative estimate we write
\[
 \vec{u}\, |\langle \vec{u}, \Dso \vec{u}\rangle|^2 = \vec{u}\, u^i \Dso u^i\, H_{\Dso} \brac{\Ds{-1}\Dso\vec{u}\cdot,\Ds{-1}\Dso\vec{u})}
\]
Set $Q=\vec{u}\vec{u}^i$. We observe that
\[
\begin{split}
 \|\Dso Q\|_{L^{(2d,2)}(\R^d)} \aleq& \|\vec{u}\|_{L^\infty}\, \|\Dso \vec{u}\|_{L^{(2d,2)}(\R^d)} + \|\Ds{\frac{1}{4}} \vec{u}\|_{L^{4d}(\R^d)}\, \|\Ds{\frac{3}{4}} \vec{u}\|_{L^{(4d,2)}(\R^d)}\\
 \aleq & \|\Dso \vec{u}\|_{L^{(2d,2)}(\R^d)} + \|\Ds{\frac{d}{2}} \vec{u}\|_{L^2(\R^d)}\, \|\Dso\vec{u}\|_{L^{(2d,2)}(\R^d)}.
 \end{split}
\]
and
\[
 \|\Ds{\frac{d-2}{2}} \Dso Q\|_{L^2(\R^d)} \aleq \|\vec{u}\|_{L^\infty} \|\Ds{\frac{d}{2}} \vec{u}\|_{L^2(\R^d)} + \|\|\vec{u}\|_{L^\infty} \|\Ds{\frac{d}{2}} \vec{u}\|_{L^2(\R^d)}^2.
\]
Consequently, from \eqref{eq:gagliardo5}
\[
\begin{split}
&\|\Ds{\frac{d-2}{2}} \brac{\vec{u}\, |\langle \vec{u}, \Dso \vec{u}\rangle|^2}\|_{L^2} \\
 \aleq &
 \brac{1+\|\Ds{\frac{d}{2}} \vec{u}\|_{L^2(\R^d)}^4}\, \|\Ds{\frac{d}{2}} \vec{u}\|_{L^2(\R^d)}\, \|\Dso \vec{u}\|_{L^{(2d,2)}(\R^d)}^{2}
 \end{split}
\]

\end{proof}

\begin{lemma}
Let $\vec{u}: \R^d \to \S^2$, then we have the corresponding growth estimate for \eqref{eq:Fcond:2}: for $t \in [1,\frac{2d-1}{4}]$, $p \in (1,\frac{2d}{2t+\frac{1}{2}})$,
\[
 \|\Ds{t}  \brac{\vec{u} \wedge \Dso{} \brac{ \vec{u} \wedge\Dso{} \vec{u}}- \vec{u} \wedge \brac{\vec{u} \wedge (-\lap) \vec{u}}}\|_{L^{p}(\R^d)} \aleq   \brac{\|\Ds{\frac{d}{2}} \vec{u}\|_{L^2(\R^d)}^{2}+\|\Ds{\frac{d}{2}} \vec{u}\|_{L^2(\R^d)}^{4}} \|\Ds{t+2}\vec{u}\|_{L^{p}(\R^d)},
\]
the corresponding estimate to \eqref{eq:Fcond:2v2} 
\[
 \| \brac{\vec{u} \wedge \Dso{} \brac{ \vec{u} \wedge\Dso{} \vec{u}}- \vec{u} \wedge \brac{\vec{u} \wedge (-\lap) \vec{u}}}\|_{L^{(\frac{2d}{3},2)}(\R^d)} \aleq \|\Dso \vec{u}\|_{L^{2d}(\R^d)}\ \|\Ds{\frac{d}{2}} \vec{u}\|_{L^{2}(\R^d)}\\
\]
and the corresponding perturbative estimate for \eqref{eq:Fcond:3}
 \[
   \|\Ds{\frac{d-2}{2}} \brac{\vec{u} \wedge \Dso{} \brac{ \vec{u} \wedge\Dso{} \vec{u}}- \vec{u} \wedge \brac{\vec{u} \wedge (-\lap) \vec{u}}} \|_{L^2} \aleq \brac{1+\|\Ds{\frac{d}{2}} \vec{u}\|_{L^2(\R^d)}^{2}} \,  \|\Dso \vec{u}\|_{L^{(2d,2)}(\R^d)}^2
\]
\end{lemma}
\begin{proof}
We have
\begin{equation}\label{eq:wurst:1243}
\begin{split}
 &\vec{u} \wedge \Dso{} \brac{ \vec{u} \wedge\Dso{} \vec{u}}- \vec{u} \wedge \brac{\vec{u} \wedge (-\lap) \vec{u}}\\
 =&\vec{u} \wedge H_{\Dso}(\vec{u}\wedge,\Dso \vec{u})\\
 \end{split}
\end{equation}
This is enough to obtain the growth estimate corresponding to \eqref{eq:Fcond:2} via the Leibniz rule:
\[
\begin{split}
 &\|\Ds{t}  \brac{\vec{u} \wedge \Dso{} \brac{ \vec{u} \wedge\Dso{} \vec{u}}- \vec{u} \wedge \brac{\vec{u} \wedge (-\lap) \vec{u}}}\|_{L^{p}(\R^d)} \\
 \aleq&\max_{t_1+t_2+t_3= t} \|\Ds{t_1} \vec{u}\|_{L^{\frac{d}{t_1}}(\R^d)}\, \|\Ds{\frac{1}{4}+t_2} \vec{u}\|_{L^{\frac{4d}{1+4t_2}}(\R^d)}\, \|\Ds{\frac{7}{4}+t_3} \vec{u}\|_{L^{\frac{dp}{d-(t+\frac{1}{4}-t_3)p}}(\R^d)}\\
\aleq&\brac{\|\vec{u}\|_{L^\infty(\R^d)} + \|\Ds{\frac{d}{2}} \vec{u}\|_{L^2(\R^d)}}\, \|\Ds{\frac{d}{2}} \vec{u}\|_{L^2(\R^d)}\,  \|\Ds{t+2} \vec{u}\|_{L^p(\R^d)}.
 \end{split}
\]
This makes sense whenever $d-(t+\frac{1}{4})p > 0$, i.e. $p<\frac{2d}{2t+\frac{1}{2}}$ and $\frac{1}{4} + t \leq \frac{d}{2}$.
Thus the first growth estimate, corresponding to \eqref{eq:Fcond:2}, is established.

For the second growth estimate, corresponding to \eqref{eq:Fcond:2v2} we observe that for very small $\alpha>0$
\[
\begin{split}
&\|\vec{u} \wedge H_{\Dso}(\vec{u}\wedge,\Dso \vec{u})\|_{L^{\frac{2d}{3}}(\R^d)}\\
\aleq& \|\Ds{1-\alpha} \vec{u}\|_{L^{\frac{2d}{1-2\alpha}}(\R^d)}\ \|\Ds{1+\alpha} \vec{u}\|_{L^{\frac{2d}{2+2\alpha}}(\R^d)}\\
\aleq& \|\Dso \vec{u}\|_{L^{2d}(\R^d)}\ \|\Ds{\frac{d}{2}} \vec{u}\|_{L^{2}(\R^d)}\\
 \end{split}
\]

We still have to establish the perturbative estimate corresponding to \eqref{eq:Fcond:3}. Since
\[
 \vec{a} \wedge (\vec{b} \wedge \vec{c}) = \vec{b} (\vec{c}\cdot \vec{a}) - \vec{c} (\vec{a} \cdot \vec{b}),
\]
we have
\[
\begin{split}
 &\vec{u} \wedge \Dso{} \brac{ \vec{u} \wedge\Dso{} \vec{u}}- \vec{u} \wedge \brac{\vec{u} \wedge (-\lap) \vec{u}}\\
 =&\brac{\vec{u} \wedge H_{\Dso}(\vec{u}\wedge,\Dso \vec{u})}^i\\
 =&\sum_{j=1}^3 u^j  H_{\Dso}(u^i,\Dso u^j) - \sum_{j=1}^3 u^j H_{\Dso}(u^j,\Dso u^i)\\
 \end{split}
\]
Now observe that since $d$ is an integer and $d\geq 4$ we have
\[
 \|\Ds{\frac{d-2}{2}} F\|_{L^2(\R^d)} \aeq \|\Ds{t} \nabla^k F\|_{L^2(\R^d)}
\]
where $k = \lfloor \frac{d-2}{2}\rfloor$ and $t= \frac{d-2}{2}-\lfloor \frac{d-2}{2}\rfloor \in \{0,\frac{1}{2}\}$.

We then have
\[
\begin{split}
 &\|\Ds{\frac{d-2}{2}} \brac{u^i  H_{\Dso}(u^j,\Dso u^k)}\|_{L^2(\R^d)}\\
 \aleq&\max_{s_1+s_2+s_3 = k}\|\Ds{t} \brac{D^{s_1}u^i  H_{\Dso}(D^{s_2}u^j,D^{s_3}\Dso u^k)}\|_{L^2(\R^d)}
 \end{split}
\]
here $s_1,s_2,s_3$ are integers.

We first observe that if $s_1 \geq 1$ then we can apply \eqref{eq:GN:1}, and for for $\sigma_1+\sigma_2+\sigma_3 = 1$,
\[
\begin{split}
 &\max_{s_1+s_2+s_3 = s, s_1 \geq 1}\|\Ds{t} \brac{D^{s_1}u^i  H_{\Dso}(D^{s_2}u^j,D^{s_3}\Dso u^k)}\|_{L^2(\R^d)}\\
 \aleq&\max_{t_1+t_2+t_3 = \frac{d-2}{2},\, t_1 \geq 3/4} \|\Ds{t_1}u^i\|_{L^{\frac{2d}{1+\sigma_1 + 2(t_1-1)}}(\R^d)}  \|\Ds{\frac{1}{2}+t_2} u^j\|_{L^{\frac{2d}{1+\sigma_2+2(t_2-1/2)}}(\R^d)} \|\Ds{\frac{3}{2}+t_3} u^k\|_{L^{\frac{2d}{1+\sigma_3+2(1/2+t_3)}}(\R^d)}\\
 \aleq&\|\Dso u^i\|_{L^{2d}(\R^d)}^{1-\sigma_1}\, \|\Ds{\frac{d}{2}} u^i\|_{L^2(\R^d)}^{\sigma_1}\, \|\Dso u^j\|_{L^{2d}(\R^d)}^{1-\sigma_2}\, \|\Ds{\frac{d}{2}} u^j\|_{L^2(\R^d)}^{\sigma_2}\, \|\Dso u^k\|_{L^{2d}(\R^d)}^{1-\sigma_3}\, \|\Ds{\frac{d}{2}} u^k\|_{L^2(\R^d)}^{\sigma_3}\\
 \aleq&\|\Ds{\frac{d}{2}} \vec{u}\|_{L^2(\R^d)}\, \|\Dso \vec{u}\|_{L^{2d}}^2.
 \end{split}
\]
For this to work we need $\sigma_1 \geq \frac{\min\{0,2(t_1-1)\}}{d-2}$, $\sigma_2 \geq \frac{\min\{0,2(1/2+t_2-1)\}}{d-2}$, $\sigma_3 \geq \frac{2(1/2+t_3)}{d-2}$ -- observe that $t_3 \leq \frac{d-2}{2}-\frac{3}{4}$, and $1 = \sigma_1+\sigma_2+\sigma_3 \overset{!}{\geq} \frac{d-4}{d-2}$, so this is possible.

Of course this also holds if we replace $D^{s_1}$ with $\Ds{s_1}$ etc. -- we are going to write $\tilde{D}^{s}$ if it does not matter if it is $\Ds{s}$ or $D^{\lfloor s\rfloor} \Ds{s-\lfloor s\rfloor}$.

We record the obtained above estimate
\begin{equation}\label{eq:aargh:v2}
\begin{split}
 &\max_{s_1+s_2+s_3 = s, s_1 \geq 1}\|\tilde{D}^{t} \brac{\tilde{D}^{s_1}u^i  H_{\Dso}(\tilde{D}^{s_2}u^j,\tilde{D}^{1+s_3}u^k)}\|_{L^2(\R^d)}\\
 \aleq&\|\Ds{\frac{d}{2}} \vec{u}\|_{L^2(\R^d)}\, \|\Dso \vec{u}\|_{L^{2d}}^2.
 \end{split}
\end{equation}

Thus, for $t = \frac{1}{2}$ if $d\geq 4$ and odd, and $t=0$ if $d \geq 4$ and even, we arrive at
\[
\begin{split}
 &\|\Ds{\frac{d-2}{2}} \sum_{j}\brac{u^j  H_{\Dso}(u^i,\Dso u^j)}\|_{L^2(\R^d)} +\|\Ds{\frac{d-2}{2}} \sum_{j}\brac{u^j H_{\Dso}(u^j,\Dso u^i)}\|_{L^2(\R^d)}\\
 \aleq&\|\Ds{\frac{d}{2}} \vec{u}\|_{L^2(\R^d)}\, \|\Dso \vec{u}\|_{L^{2d}}^2\\
 &+\max_{s_2+s_3=\lfloor \frac{d-2}{2}\rfloor, s_i \in \Z_+}\|\sum_{j}\Ds{t} \brac{u^j  H_{\Dso}(D^{s_2} u^i,\Dso D^{s_3}u^j)}\|_{L^2(\R^d)} \\
 &+\max_{s_2+s_3=\lfloor \frac{d-2}{2}\rfloor, s_i \in \Z_+}\|\sum_{j}\Ds{t} \brac{u^j H_{\Dso}(D^{s_2}u^j,\Dso D^{s_3}u^i)}\|_{L^2(\R^d)}\\
 \end{split}
\]
We first treat some special cases.
We observe
\begin{equation}\label{eq:hHdsofgdecomp}
\begin{split}
&h(x)\, H_{\Dso}(f,g)(x)\\
=&c\int_{\R^d} \frac{h(x)\, \brac{f(x)-f(y)}\, \brac{g(x)-g(y)}}{|x-y|^{d+1}}\, dy\\
=&c\int_{\R^d} \frac{\, \brac{h(x)f(x)-h(y)f(y)}\, \brac{g(x)-g(y)}}{|x-y|^{d+1}}\, dy\\
&+c\int_{\R^d} \frac{\brac{h(x)-h(y)}\, \brac{f(x)-f(y)}\, \brac{g(x)-g(y)}}{|x-y|^{d+1}}\, dy\\
&-cf(x)\int_{\R^d} \frac{\brac{h(x)-h(y) }\, \brac{g(x)-g(y)}}{|x-y|^{d+1}}\, dy\\
=&H_{\Dso}(fh,g)(x) - f(x) H_{\Dso}(h,g)(x) + H_{\Dso,3}(f,g,h)(x)\\
\end{split}
\end{equation}
where we denote the triple commutator
\[
H_{\Dso,3}(f,g,h)(x) := c\int_{\R^d} \frac{ \brac{f(x)-f(y)}\, (g(x)-g(y))\, \brac{h(x)-h(y)}}{|x-y|^{d+1}}\, dy
\]
We have with \eqref{eq:GN:1}, observing that $\langle \vec{u},\tilde{D}^{\frac{d}{2}-s_2}\vec{u}\rangle =H_{\tilde{D}^{\frac{d}{2}-s_2}}(\vec{u}\cdot,\vec{u})$
\begin{equation}\label{eq:asklvxclkklxcjvv1v2}
 \max_{s_2 \in [0,\frac{d-1}{2}]}\|\Ds{t}H_{\Dso}(\tilde{D}^{s_2} u^i,\langle \vec{u},\tilde{D}^{\frac{d}{2}-s_2}\vec{u}\rangle )\|_{L^2(\R^d)} \aleq \|\Dso \vec{u}\|_{L^{2d}(\R^d)}^2\, \|\Ds{\frac{d}{2}} u^i\|_{L^2(\R^d)}.
\end{equation}

Thus, with the help of \eqref{eq:hHdsofgdecomp}, \Cref{la:reallyworsttermv2} and \eqref{eq:aargh:v2} we can treat all expressions where more than $1$ derivative hits one of the $u^j$-terms, at the expense of scalar product terms $\langle \vec{u},\Ds{s} \vec{u}\rangle$  which are treated in \eqref{eq:asklvxclkklxcjvv1v2} -- as long as $s_3 \neq \frac{d-2}{2}$, namely we have:
\begin{equation}\label{eq:cvioxvcposfdpjio3}
\begin{split}
 &\|\Ds{\frac{d-2}{2}} \sum_{j}\brac{u^j  H_{\Dso}(u^i,\Dso u^j)}\|_{L^2(\R^d)} +\|\Ds{\frac{d-2}{2}} \sum_{j}\brac{u^j H_{\Dso}(u^j,\Dso u^i)}\|_{L^2(\R^d)}\\
 \aleq&\|\Ds{\frac{d}{2}} \vec{u}\|_{L^2(\R^d)}\, \|\Dso \vec{u}\|_{L^{2d}}^2\\
%
 &+\max_{s_2+s_3=\lfloor \frac{d-2}{2}\rfloor, s_i \in \Z_+, s_3=\frac{d-2}{2}}\|\sum_{j}\Ds{t} \brac{u^j  H_{\Dso}(D^{s_2} u^i,\Dso D^{s_3}u^j)}\|_{L^2(\R^d)} \\
 &+\max_{s_2+s_3=\lfloor \frac{d-2}{2}\rfloor, s_i \in \Z_+; s_2=0}\|\sum_{j}\Ds{t} \brac{u^j H_{\Dso}(D^{s_2}u^j,\Dso D^{s_3}u^i)}\|_{L^2(\R^d)}\\
 \end{split}
\end{equation}
That is we still need to estimate (the second term on the right-hand side \eqref{eq:cvioxvcposfdpjio3} is only nonzero if $d$ is even, thus $t=0$)

\begin{equation}\label{eq:cvioxvcposfdpjio3:guy1}
 \|\sum_{j} \brac{u^j  H_{\Dso}(u^i,\tilde{D}^{\frac{d}{2}} u^j)}\|_{L^2(\R^d)}\aleq \|\Dso \vec{u}\|_{L^{2d}(\R^d)}^2\, \|\Ds{\frac{d}{2}} u^i\|_{L^2(\R^d)},
\end{equation}
and
\begin{equation}\label{eq:cvioxvcposfdpjio3:guy2}
 \|\sum_{j}\brac{u^j H_{\Dso}(u^j,\tilde{D}^{\frac{d}{2}}u^i)}\|_{L^2(\R^d)} \aleq \|\Dso \vec{u}\|_{L^{2d}(\R^d)}^2\, \|\Ds{\frac{d}{2}} u^i\|_{L^2(\R^d)},
\end{equation}
and
\begin{equation}\label{eq:cvioxvcposfdpjio3:guy3}
 \|\sum_{j}\Ds{\frac{1}{2}} \brac{u^j H_{\Dso}(u^j,\tilde{D}^{\frac{d-1}{2}}u^i)}\|_{L^2(\R^d)} \aleq \|\Dso \vec{u}\|_{L^{2d}(\R^d)}^2\, \|\Ds{\frac{d}{2}} u^i\|_{L^2(\R^d)}.
\end{equation}

\underline{Estimate of \eqref{eq:cvioxvcposfdpjio3:guy1}}:
We observe similar to \eqref{eq:hHdsofgdecomp}
\[
\begin{split}
&\abs{\sum_{j} u^j  H_{\Dso}( u^i,\tilde{D}^{\frac{d}{2}}u^j)(x)}\\
\aleq&\abs{H_{\Dso}( u^i,\langle \vec{u},\tilde{D}^{\frac{d}{2}}\vec{u}\rangle )(x)}\\
&+   \abs{\int_{\R^d} \frac{\abs{u^j(x)-u^j(y)}\, \abs{ u^i(x)- u^i(y)}\, \abs{\tilde{D}^{\frac{d}{2}}u^j(y)}}{|x-y|^{d+1}}\, dy}
\end{split}
\]
We have with \eqref{eq:GN:1}
\begin{equation}\label{eq:asklvxclkklxcjvv1}
 \max_{s_2 \in [0,\frac{d-1}{2}]}\|H_{\Dso}( u^i,\langle \vec{u},\tilde{D}^{\frac{d}{2}}\vec{u}\rangle )\|_{L^2(\R^d)} \aleq \|\Dso \vec{u}\|_{L^{2d}(\R^d)}^2\, \|\Ds{\frac{d}{2}} u^i\|_{L^2(\R^d)}.
\end{equation}

For the other term, pick in \cite[Lemma 2.11]{ERFS22}, $\alpha_1, \alpha_2 \in (0,1)$ with $\alpha_1 + \alpha_2 > 1$, $p_1 = \frac{2d}{2\alpha_1-1}$, $p_2 = \frac{2d}{1+2(\alpha_2-1)}$, and $
p_3 := \frac{2d}{d-2(\alpha_1+\alpha_2-1)}$ such that
\[
\frac{dp_3}{d+(\alpha_1+\alpha_2-1)p_3)} =2
\]
Observe that the last two assumptions require that $\alpha_1 + \alpha_2 \aeq 1$.

Then from \cite[Lemma A.7]{ERFS22} we find
\[
\begin{split}
 &\brac{\int_{\R^d} \abs{\int_{\R^d} \frac{\abs{u^j(x)-u^j(y)}\, \abs{ u^i(x)- u^i(y)}\, \abs{\tilde{D}^{\frac{d}{2}}u^j(y)}}{|x-y|^{d+1}}\, dy}^2 dx}^{\frac{1}{2}} \\
 \aleq&\|\Ds{\alpha_1} u^j\|_{L^{\frac{2d}{2\alpha_1-1}}(\R^d)}\, \|\Ds{\alpha_2} u^i \|_{L^{\frac{2d}{1+2(\alpha_2-1)}}(\R^d)}\, \|\Ds{\frac{d}{2}} u^j\|_{L^{2}(\R^d)}\\
 \aleq&\|\Dso \vec{u}\|_{L^{2d}(\R^d)}^2\, \|\Ds{\frac{d}{2}} \vec{u}\|_{L^2(\R^d)}.
 \end{split}
\]
Where the last inequality is Sobolev's embedding (observe that $\alpha_1,\alpha_2 < 1$).

This implies \eqref{eq:cvioxvcposfdpjio3:guy1}.

\underline{Estimate \eqref{eq:cvioxvcposfdpjio3:guy2} is proven} in \Cref{la:reallyworstterm}.

Lastly we need the \underline{estimate of \eqref{eq:cvioxvcposfdpjio3:guy3}}
\[
 \|\sum_{j}\Ds{\frac{1}{2}} \brac{u^j H_{\Dso}(u^j,\tilde{D}^{\frac{d-1}{2}}u^i)}\|_{L^2(\R^d)} \aleq \|\Dso \vec{u}\|_{L^{2d}(\R^d)}^2\, \|\Ds{\frac{d}{2}} u^i\|_{L^2(\R^d)}.
\]
This is essentially the same estimate as above, once we observe (using that $|\vec{u}|=1$)
\[
\sum_j u^j H_{\Dso}(u^j,\tilde{D}^{\frac{d-1}{2}}u^i)(x) = c\int_{\R^d} \frac{|\vec{u}(x)-\vec{u}(y)|^2 \brac{\tilde{D}^{\frac{d-1}{2}}u^i(x)-\tilde{D}^{\frac{d-1}{2}}u^i(y)}}{|x-y|^{d+1}}\, dy
\]
This finishes the proof of the corresponding perturbative estimate for \eqref{eq:Fcond:3}, and we can conclude.
\end{proof}

\section{Wellposedness: Proof of Corollary~\ref{co:wellposedness}}\label{s:wellposed}

\subsection{Short-time existence of solutions for smooth initial data}

We begin with the following existence result, which can be proven by relatively standard methods
\begin{theorem}[Existence half-wave equation]\label{th:existsmoothhalfwave}
Let $d \geq 3$. For any $K$ suitably large, and any $\Gamma>0$ there exists $T$ such that the following holds:

Assume that $\vec{u_0} \in C^2(\R^d,\S^{2})$ such that
\[
 \|\nabla \vec{u_0}\|_{H^{K-1}(\R^d)} \leq \Gamma.
\]
Then there exists a solution $\vec{u} \in C^2(\R^d \times (-T,T),\S^2)$ such that
\[
 \begin{cases}
 \partial_t \vec{u} = \vec{u} \wedge \laps{1} \vec{u} \quad &\text{in $\R^d \times (-T,T)$}\\
 \vec{u}(0) = \vec{u_0} \quad \text{in $\R^d$}\\
 \end{cases}
\]
Moreover we have
\begin{equation}\label{eq:KSPDEbdsHk}
 \|d\vec{u}\|_{L^\infty((-T,T),H^{K-1}(\R^d))}  < \infty.
\end{equation}
and
\begin{equation}\label{eq:KSPDEbds}
 \|\partial_t \vec{u}\|_{L^\infty(\R^d \times (-T,T))} +\|\Dso \vec{u}\|_{L^\infty(\R^d \times (-T,T))}+\|\nabla \vec{u}\|_{L^\infty(\R^d \times (-T,T))} < \infty.
\end{equation}
and for any $s > 1$ suitably close to $1$ we have
\begin{equation}\label{eq:existhalfwave:finitenorm}
\begin{split}
\|\Ds{\frac{d-2}{2}} d\vec{u}\|_{C_t^0 L^2(\R^d\times (-T,T))} + \|\Ds{s-1}d \vec{u}\|_{L^2_t L^{(\frac{2d}{2s-1},2)}_x(\R^d \times (-T,T))}<\infty
\end{split}
\end{equation}

\end{theorem}

First we solve a related equation, cf. \cite[(2-1)]{KS18}, without the assumption of $\vec{u}$ mapping into $\S^2$.
\begin{proposition}\label{pr:existencenotarget}
For any $K \in \N$ and any $\Gamma > 0$ there exists an $T = T(\Gamma,K) > 0$ such that the following holds:

Assume that $\vec{u_0} \in C^2(\R^d,\S^{2})$ such that
\[
 \|d\vec{u_0}\|_{H^{K-1}(\R^d)} \leq \Gamma.
\]
Then there exists a map $\vec{u} \in C^K(\R^d\times (-T,T),\col{\R^3})$, solving
\begin{equation}\label{eq:KSPDE}
\begin{split}
(\partial_{tt} - \lap) \vec{u} =& \vec{u} \brac{|\nabla \vec{u}|^2-|\partial_t \vec{u}|^2} + \Pi_{\vec{u}_\perp} \Dso \vec{u}\, \langle \vec{u} \cdot \Dso \vec{u} \rangle \\
&+\vec{u} \wedge \Dso{} \brac{ \vec{u} \wedge\Dso{} \vec{u}}- \vec{u} \wedge \brac{\vec{u} \wedge (-\lap) \vec{u}}.
\end{split}
\end{equation}
with initial boundary data
\begin{equation}\label{KSPDEinit}
\begin{cases}
 \vec{u}(0) = \vec{u_0} \quad &\text{in $\R^d$}\\
 \partial_t \vec{u}(0) = \vec{u_0} \wedge \laps{1} \vec{u_0} \quad &\text{in $\R^d$}\\
 \end{cases}
\end{equation}
and we have $|\vec{u}| \in [\frac{1}{2},\frac{3}{2}]$ in $\R^d \times (-T,T)$,

Here we denote
\[
 \Pi_{u_\perp}(\vec{v}) = \vec{v} - \frac{\vec{u}}{|\vec{u}|}\langle \vec{v}, \frac{\vec{u}}{|\vec{u}|}\rangle
\]

Moreover $\vec{u}$ satisfies \eqref{eq:KSPDEbds}, \eqref{eq:KSPDEbdsHk}, and \eqref{eq:existhalfwave:finitenorm}.
\end{proposition}

\begin{proof}
This is a consequence of a typical fixed point argument, see, e.g., \cite[p.57]{SS1998}:

Set $\vec{u_1} := \vec{u_0} \wedge \laps{1} \vec{u_0}$. Clearly, assuming $K$ is large enough, we have
\[
 \|\vec{u_1}\|_{H^{K-1}(\R^d)} \aleq \|\vec{u_1}\|_{H^{K-1}(\R^d)},
\]
and $\vec{u_1} \cdot \vec{u_0} \equiv 0$.

We want find some $T >0$ so that there exists a solution $\vec{u} \in \vec{q} + L^2(\R^d \times (-T,T),\R^3)$ to
\[
 \begin{cases}
  (\partial_{tt}-\lap) \vec{u} = H(\vec{u}) \quad &\text{in $\R^d \times (-T,T)$}\\
  \vec{u}(0) = \vec{u}_0\\
  \partial_t \vec{u}(0) = \vec{u}_1.
 \end{cases}
\]
where
\begin{equation}\label{eq:Hvecuest}
\begin{split}
H(\vec{u}) :=& \vec{u} \brac{|\nabla \vec{u}|^2-|\partial_t \vec{u}|^2} + \Pi_{\vec{u}_\perp} \Dso \vec{u}\, \langle \vec{u} \cdot \Dso \vec{u} \rangle \\
&+\vec{u} \wedge \Dso{} \brac{ \vec{u} \wedge\Dso{} \vec{u}}- \vec{u} \wedge \brac{\vec{u} \wedge (-\lap) \vec{u}}.
\end{split}
\end{equation}

For $k$ suitably large and $T$ to be fixed later, define for $\vec{u} \in L^2((-T,T)\times \R^d,\R^3)$ the norm
\[
 \|\vec{u}\|_{X} := \sup_{t \in (-T,T)} E_k(\vec{u},t)
\]
where
\[
 E_k(\vec{u},t) := \|d \vec{u}(t)\|_{H^{k-1}(\R^d)}.
\]

Observe that since $d\geq 3$, $\|\vec{u}\|_{X} < \infty$ implies $\vec{u}-\vec{q}(t) \in L^{\frac{2d}{d-2}}(\R^d)$ for some constant-in-space $\vec{q}(t)$, $t \in (-T,T)$.
Moreover
\[
\begin{split}
 & \|\vec{u}(t_2)-\vec{u}(t_1)\|_{H^{k-1}(\R^d)}\\
 =& \|\int_{t_1}^{t_2}\partial_t \vec{u}(t)\|_{H^{k-1}(\R^d)}\\
 \aleq& |t_2-t_1|^{\frac{1}{2}} \|\partial_t \vec{u}(t)\|_{L^2_t H^{k-1}(\R^d \times (-T,T))}\\
 \aleq&|t_2-t_1| \|\vec{u}\|_{X}.
 \end{split}
\]
For $k \gg 1$ we conclude
\begin{equation}\label{eq:uclosetosphere}
 \|\vec{u}(t)-\vec{u_0}\|_{L^\infty(\R^d)} \aleq \|\vec{u}(t)-\vec{u_0}\|_{H^{k-1}(\R^d)} \leq T \|\vec{u}\|_{X}.
\end{equation}
Also, by Sobolev embedding, if $u \in X$ then $u \in C^2(\R^d \times (-T,T),\R^3)$.

We then consider the Banach space $(X,\|\cdot\|_{X})$
\[
 X := \{\vec{u} \in L^2_t L^{\frac{2d}{d-2}}(\R^d,\R^3): \quad \|\vec{u}\|_{X} < \infty \}.
\]
And for some $\delta > 0$ the subset
\[
 X_{\delta} := \{\vec{u} \in X: \quad \vec{u}(0) = \vec{u_0},\ \partial_t \vec{u}(0) = \vec{u_1}, \quad \|\vec{u}\|_{X} \leq \delta\}.
\]
which is a complete metric space.

By \eqref{eq:uclosetosphere}: for any $\delta$ there exists $T_0$ such that for $T < T_0$ we have for any $\vec{u} \in X_{\delta}$ that $|\vec{u}| \in [\frac{1}{2},\frac{3}{2}]$ in $\R^d \times (-T,T)$.

We now define the linear map. Given $\vec{v} \in X_\delta$, set $Tv \in X$ the solution
\[
\begin{cases}
 (\partial_{tt} - \lap) S\vec{v} = H(\vec{v}) \quad &\text{in $\R^d \times (-T,T)$}\\
 S\vec{v}(0) = \vec{u_0} \quad &\text{in $\R^d$}\\
 \partial_t S\vec{v}(0) = \vec{u_1}\quad &\text{in $\R^d$}.
 \end{cases}
\]
Observe that there is an explicit formula for $Sv$, namely we have

\[
 S\vec{v}(x,t) = \vec{w}(x,t) + \int_0^t \vec{H}(x,s,t)\, ds
\]
where
\[
 \vec{w}(x,t) = \frac{e^{\i t \sqrt{-\lap}} + e^{-\i t \sqrt{-\lap}}}{2} \vec{u_0}+\frac{e^{\i t \sqrt{-\lap}} - e^{-\i t \sqrt{-\lap}}}{2\i} \sqrt{-\lap}^{-1}\vec{u_1}
\]
and thus
\[
  \sqrt{-\lap}^{-1} \partial_t \vec{w}(x,t) = -\frac{e^{\i t \sqrt{-\lap}} \col{-}  e^{-\i t \sqrt{-\lap}}}{2\col{\i}}  \vec{u_0}(x)+\frac{e^{\i t \sqrt{-\lap}} \col{+}  e^{-\i t \sqrt{-\lap}}}{\col{2}} \sqrt{-\lap}^{-1} \vec{u_1}(x)
\]
and thus
\[
 \|\vec{w}\|_{X} \leq C\, E(0),
\]
with a constant independent of $T$.

By Duhamel principle
by Minkowski (observe that the initial data for $\vec{H}$ appears as the $\partial_t$-term, hence the improvement in derivatives)
\[
\begin{split}
 \|\int_0^t \vec{H}(x,s,t)\, ds\|_{H^k(\R^d)} \aleq& \int_0^t \|\vec{H}(x,s,t)\|_{H^k(\R^d)}\, ds\\
 \aleq& \int_0^t \|\vec{H}(\vec{v})\|_{H^{\col{k-1}}(\R^d)}\, ds\\
 \end{split}
\]

Now for $k$ suitably large we need to check that for $H$ as in \eqref{eq:Hvecuest}, for $\vec{v} \in X_\delta$ (assuming that $T<T_0(\delta)$ is small enough so that $|\vec{v}| \aeq 1$ so that $\Pi_{\vec{v}^\perp}$ is well-defined and nonsingular)
\begin{equation}\label{eq:Hestasd}
 \|H\vec{v}(t)\|_{H^{\col{k-1}}(\R^d)} \aleq (1+E_k(\vec{v},t)^\sigma)\, E_k(\vec{v},t)
\end{equation}
The only term where this is not immediate from the algebra property of $H^k$ for $k \gg 1$ is
\[
\begin{split}
 &\Dso{} \brac{ \vec{u} \wedge\Dso{} \vec{u}}- \brac{\vec{u} \wedge (-\lap) \vec{u}}\\
 =&[\Dso,\vec{u}\wedge](\Dso \vec{u}).
 \end{split}
\]
The worrisome guy would be
\[
 [\Dso,\vec{u}\wedge](\Dso \nabla^{k-1}\vec{u}).
\]
but as a consequence of Coifman-Meyer/Kato-Poince estimates (cf. \cite[arxiv version (6.1)]{LScomm})
\[
 \|[\Dso,\vec{u}\wedge](\Dso \nabla^{k-1}\vec{u})\|_{L^2(\R^d)} \aleq [\vec{u}]_{\lip} \|d\vec{u}\|_{H^{k-1}(\R^d)} \aleq \|d\vec{u}\|_{H^{k-1}(\R^d)}^2.
\]
This establishes \eqref{eq:Hestasd}. In particular, we find
\[
 \|S\vec{v}\| \aleq C_0 E_k(\vec{v},0) + T \|\vec{v}\|_{X}
\]
So if we choose $\delta := 2C_0 E_k(\vec{v},0)$, and $T$ suitably small so that $|\vec{v}| \aeq 1$ we have by taking $T_0(\delta)$ possibly smaller
\[
 \|S\vec{v}\| \leq \delta \quad \forall \vec{v} \in X_\delta.
\]
That is $S: X_\delta \to X_\delta$.

We now need to show that $S$ is a contraction on $X_\delta$. For $\vec{w} := S\vec{v_1}-S\vec{v_2}$ we have
\[
\begin{cases}
 (\partial_{tt} - \lap) \brac{S\vec{v_1}-S\vec{v_2}} = H(\vec{v_1})-H(\vec{v_2}) \quad &\text{in $\R^d \times (-T,T)$}\\
 \brac{S\vec{v_1}-S\vec{v_2}}(0) = 0 \quad &\text{in $\R^d$}\\
 \partial_t \brac{S\vec{v_1}-S\vec{v_2}}(0) = 0\quad &\text{in $\R^d$}.
 \end{cases}
\]
Now similar to the argument above we have
\[
 \|\brac{S\vec{v_1}-S\vec{v_2}}\|_{X} \aleq T (1+\delta^\sigma)\, \|\vec{v_1}-\vec{v_2}\|_{X},
\]
which implies that for $T$ sufficiently small $S: X_\delta \to X_\delta$ is a contraction, and thus has a unique fixed point in $X_\delta$.

As for the regularity of $\vec{u}$, we first observe (taking $k \gg K$, that $\partial_{tt} u \in L^\infty$ and from this we bootstrap to the full regularity.

It remains to obtain \eqref{eq:KSPDEbds}, but observe that by definition of $X_\delta$ we have (for $k \gg 1$)
\[
\begin{split}
&\sup_{t \in (-T,T)}\|\nabla \vec{u}\|_{L^\infty(\R^d)}+\sup_{t \in (-T,T)}\|\partial_t \vec{u}\|_{L^\infty(\R^d)}\\
 \aleq&
 \sup_{t \in (-T,T)}\|d \vec{u}\|_{H^{k-1}(\R^d)}\\
 =&\|\vec{u}\|_{X} < \delta < \infty
 \end{split}
\]
which readily implies \eqref{eq:KSPDEbds}.

In order to see that $\vec{u}$ satisfies \eqref{eq:existhalfwave:finitenorm} we apply \Cref{la:strichartz}, and we have
\[
\begin{split}
 &\|\Ds{\frac{d-4}{2}}d\vec{u}(t)\|_{C^0_t L^{2}_x(\R^d) \times (-T,T)}  + \|\Ds{s -1}\vec{u}(t)\|_{L^2_t L^{(\frac{2d}{2s-1},2)}_x(\R^d \times (-T,T)} \\
 \aleq& \|\Ds{\frac{d-2}{2}} \vec{u_0}\|_{L^2(\R^d)} + \|\Ds{\frac{d-4}{2}} \vec{u_1}\|_{L^2(\R^d)}\\
 &+\|\Ds{\frac{d-4}{2}}H(\vec{u})\|_{L^1_t L^2_x (\R^d \times (-T,T))}\\
 \end{split}
\]
Observe that the norms on the right-hand side are finite by assumption, in particular
\[
 \|\Ds{\frac{d-4}{2}}H(\vec{u})\|_{L^1_t L^2_x (\R^d \times (-T,T))} \aleq C(T) \|\Ds{\frac{d-4}{2}}H(\vec{u})\|_{L^\infty_t L^2_x (\R^d \times (-T,T))} \overset{\eqref{eq:KSPDEbdsHk}}{<} \infty
\]
\end{proof}

\begin{proposition}\label{pr:vecuintosphere}
Take $\vec{u}$ the solution of \Cref{pr:existencenotarget}, and set
\[
 w := \brac{|\vec{u}|^2-1}.
\]
Then, assuming that $K$ was chosen sufficiently large, we have
\begin{equation}\label{KSPDEbds2Hk}
 \|d \vec{w}\|_{L^\infty((-T,T),H^{K-1}(\R^d))} < \infty.
\end{equation}
and
\begin{equation}\label{KSPDEbds2}
 \|\vec{w}\|_{L^\infty((-T,T),\R^d)} + \|\partial_t \vec{w}\|_{L^\infty(\R^d \times (-T,T))} +\|\nabla \vec{w}\|_{L^\infty(\R^d \times (-T,T))} < \infty
\end{equation}
and
\[\begin{cases}
 (\partial_{tt}-\lap) w =2\, w \brac{|\nabla \vec{u}|^2-|\partial_t \vec{u}|^2} \quad &\text{in $\R^d \times [-T,T]$}\\
 w(0) = 0 \quad &\text{in $\R^d$}\\
 \partial_t w(0) = 0 \quad &\text{in $\R^d$}.
 \end{cases}
\]
\end{proposition}
\begin{proof}

Observe that $w \in C^2(\R^d \times (-T,T)$ and
\[
 \partial_t w = 2\partial_t \vec{u}\, \cdot \vec{u}=2\partial_t \vec{u}\, \cdot \brac{\vec{u}-\vec{q}}+ 2\partial_t \vec{u} \cdot \vec{q}.
\]
That is
\[
 |\partial_t w| + |\nabla w| \aleq |\partial_t \vec{u}|
\]
Similar we can argue for space derivatives, and obtain for all $k = 1,\ldots$
\[
 |\nabla^{k-1} \partial_t w| + |\nabla^{k} w| \aleq \brac{|\nabla^{k-1} \partial_t \vec{u}| + |\nabla^{k} (\vec{u}-\vec{q}|}+ \text{lower order products}
\]
Then \eqref{KSPDEbds2Hk} and \eqref{KSPDEbds2} follow from \eqref{eq:KSPDEbdsHk} and \eqref{eq:KSPDEbds} and Sobolev embedding for $K$ suitably large.

As for the PDE, we have as in \cite[p.678]{KS18},
\[
\begin{split}
 &(\partial_{tt}-\lap) \brac{\frac{1}{2}\brac{|\vec{u}|^2-1}}\\
 =&\langle \vec{u}, \partial_{tt}-\lap \vec{u}\rangle + |\partial_{t} \vec{u}|^2 - |\nabla \vec{u}|^2 \\
 \overset{\eqref{eq:KSPDE}}{=}&\brac{|\vec{u}|^2 -1}\brac{|\nabla \vec{u}|^2-|\partial_{t} \vec{u}|^2}
 \end{split}
\]
This establishes the desired equation for $w := \brac{|\vec{u}|^2-1}$.

Since $|\vec{u}(0)|=1$ we conclude also $w(0) = 0$. Moreover,
\[
 \partial_t w(0) = 2\langle \vec{u}(0), \partial_t \vec{u}(0)\rangle = 0,
\]
by the initial values for \eqref{KSPDEinit}.
\end{proof}

\begin{lemma}[Uniqueness]\label{la:uniqueness}
Let $d \geq 3$. Assume that $w \in C^2(\R^d \times (-T,T))$ satisfies
\begin{equation}\label{eq:KSPDEbds3}
 \|\partial_t \vec{w}\|_{L^\infty((-T,T),L^2(\R^d))} +\|\nabla \vec{w}\|_{L^\infty((-T,T),L^2(\R^d)))} < \infty.
\end{equation}
and solves
\[
 \begin{cases}
 (\partial_{tt} -\lap)w = w B \quad &\text{in $\R^d \times (-T,T)$}\\
 w(0) = 0\quad &\text{in $\R^d$}\\
 \partial_t w(0) = 0 \quad &\text{in $\R^d$}.
 \end{cases}
\]
where
\[
 \sup_{t \in (-T,T)} \|B\|_{L^d(\R^d)} < \infty.
\]
Then $w \equiv 0$.
\end{lemma}
\begin{proof}
Set
\[
E(t) := \frac{1}{2}\brac{\int_{\R^d} |\nabla w|^2 + |\partial_t w|^2}
\]
$E(t)$ is finite by \eqref{eq:KSPDEbds3}.

Then we have
\[
\begin{split}
 \frac{d}{dt} E(t) =& \int_{\R^d} (\partial_{tt}-\lap)w\, \partial_t w \\
 =&\int_{\R^d} w \partial_t w\ B\\
 \leq& \|w\|_{L^{\frac{2d}{d-2}}(\R^d)}\, \|\partial_t w\|_{L^2(\R^d)}\, \|B\|_{L^d(\R^d)}\\
 \overset{d \geq 3}{\aleq}&\|dw\|_{L^2(\R^d)}^2\, \|B\|_{L^d(\R^d)}\\
 \leq&E(t)\, \|B\|_{L^d(\R^d)}.
 \end{split}
\]
Since $E(0) = 0$ and $E(t) \geq 0$ for all $t$ we get from Gronwall's lemma that actually $E(t) \equiv 0$. This readily implies $w \equiv 0$.
\end{proof}

As a corollary of \Cref{la:uniqueness} and \Cref{pr:vecuintosphere} we obtain
\begin{corollary}
The solution of \Cref{pr:existencenotarget} actually maps into the sphere.
\end{corollary}
\begin{proof}
We observe that
\[
  B := 2\, \brac{|\nabla \vec{u}|^2-|\partial_t \vec{u}|^2}
\]
satisfies by \eqref{eq:KSPDEbdsHk} and \eqref{eq:KSPDEbds}
\[
 \sup_{t} \|B\|_{L^2(\R^d)}+\|B\|_{L^\infty(\R^d)} < \infty
\]
By standard interpolation we conclude
\[
 \sup_{t} \|B\|_{L^d(\R^d)} < \infty.
\]
So in view of \Cref{pr:vecuintosphere} $w := \abs{\vec{u}}^2-1$ satisfies the assumptions of \Cref{la:uniqueness}, and thus $w \equiv 0$. This implies $|\vec{u}|^2 \equiv 1$.
\end{proof}

By now we have a solution $\vec{u}$ that maps into the sphere of \eqref{eq:KSPDE}. All we need in order to conclude \Cref{th:existsmoothhalfwave} is the following
\begin{proposition}
The solution of \Cref{pr:existencenotarget} actually solves the half-wave map equation
\[
 \partial_t \vec{u} = \vec{u} \wedge \laps{1} \vec{u} \quad \text{in $\R^d \times (-T,T)$}.
\]
\end{proposition}
\begin{proof}
We argue as in \cite[Section 6]{KS18}, although our situation is simpler since we have nice initial data.

Set
\[
 X:= \partial_t \vec{u} - \vec{u} \wedge \laps{1} \vec{u} \in L^\infty((-T,T),H^{K}(\R^d))
\]
and consider ($K$ being large enough)
\[
 \tilde{E}(t) := \frac{1}{2} \int_{\R^d} |\vec{X}(t)|^2 < \infty.
\]
Observe that by the initial data for $\vec{u}$ we have $\tilde{E}(0) = 0$.

Set
\[
\Gamma := \sup_{t \in (-T,T)} \|\vec{u}(t)\|_{L^\infty} + \|\Dso \vec{u}(t)\|_{L^\infty} + \|\nabla \vec{u}(t)\|_{L^\infty} + \|\partial_t \vec{u}(t)\|_{L^\infty(\R^d)} \overset{\eqref{eq:KSPDEbds}}{<} \infty
\]

We have \cite[p.679]{KS18}
\[
 \partial_t \vec{X} = -\vec{X} \wedge \Dso \vec{u} - \vec{u} \wedge \Dso \vec{X} - \vec{u}\ \langle \vec{X},\vec{u}\wedge \Dso \vec{u}+\partial_t \vec{u}\rangle
\]
Thus
\[
\begin{split}
 \frac{1}{2} \frac{d}{dt} \tilde{E}(t) =& \int_{\R^d} \vec{X} \cdot \partial_t \vec{X}\\
 =&-\int_{\R^d} \vec{X} \cdot \brac{\vec{u} \wedge \Dso \vec{X}} - \int_{\R^d} \langle \vec{X}, \vec{u}\rangle \langle \vec{X}, \vec{u}\wedge \Dso \vec{u}+\partial_t \vec{u}\rangle
 \end{split}
\]
Clearly,
\[
 \abs{\int_{\R^d} \langle \vec{X}, \vec{u}\rangle \langle \vec{X}, \vec{u}\wedge \Dso \vec{u}+\partial_t \vec{u}\rangle } \aleq \Gamma\, \|\vec{X}\|_{L^2(\R^d)}^2 = \Gamma\, \tilde{E}(t)
\]
Also, we use the determinant structure (where $\epsilon_{ijk} \in  \{-1, 0, 1\}$ is the Levi-Civita tensor)
\[
\begin{split}
=&\epsilon_{ijk} \int_{\R^d} u^i\, X^j\, \Dso X^k\\
=&\epsilon_{ijk} \int_{\R^d} u^i\, X^j\, \Dso X^k - u^i\, X^k\, \Dso X^j\\
=&\epsilon_{ijk} \int_{\R^d} \brac{\Dso \brac{u^i\, X^j} - u^i\, \Dso X^j}\, X^k\\
\end{split}
\]
and thus with Coifman-Meyer/Kato-Poince estimates (cf. \cite[arxiv version (6.1)]{LScomm})
\[
\begin{split}
 \abs{\int_{\R^d} \vec{X} \cdot \brac{\vec{u} \wedge \Dso \vec{X}}} \aleq& \|\vec{X}\|_{L^2(\R^d)}\, \|[\Dso, \vec{u}](\vec{X})\|_{L^2(\R^d)}\\
 \aleq& \|\nabla \vec{u}\|_{L^\infty(\R^d)}\, \|X\|_{L^2(\R^d)}^2\\
 \aleq& \Gamma \|X\|_{L^2(\R^d)}^2.
 \end{split}
\]
That is, we have shown
\[
\frac{d}{dt} \tilde{E}(t) \aleq \Gamma E(t),
\]
and since $E(0) = 0$ we conclude from Gronwall's lemma $E(t) = 0$ for all $t \in (-T,T)$. Consequently, $X \equiv 0$, which implies that $\vec{u}$ solves the halfwave map.
\end{proof}

\subsection{Global existence of solutions for smooth initial data with small \texorpdfstring{$H^{\frac{d}{2}} \times H^{\frac{d-2}{2}}$}{Hd2}-norm}
With the the solution $\vec{u}$ of \Cref{pr:existencenotarget} mapping into the sphere we can find an estimate somewhat improving \eqref{eq:KSPDEbdsHk}
\begin{proposition}\label{pr:dvecucontrol}
Assume that $\vec{u_0} \in C^2(\R^d,\S^{2})$ and $\vec{u} \in C^2(\R^d\times (T_1,T_2),\S^2)$, solves
\[
\begin{split}
\partial_t \vec{u} = \vec{u} \wedge \laps{1} \vec{u} \quad \text{in $\R^d \times [T_1,T]$}\\
\vec{u}(T_1)= \vec{u_0} \quad \text{in $\R^d$}
\end{split}
\]
such that $\vec{u}$ satisfies \eqref{eq:KSPDEbds}, \eqref{eq:KSPDEbdsHk}, and \eqref{eq:existhalfwave:finitenorm} hold.

Then, whenever the right-hand side is finite we have (for suitably large $K$)
\begin{equation}\label{eq:dvecucontrol}
 \sup_{t \in [-T,T]}\|d\vec{u}(t)\|_{H^{K-1}(\R^d)}^2 \aleq \|d\vec{u_0}\|_{H^{K-1}(\R^d)}^2+ \|d\vec{u}\|_{L^2_t L_x^{2d}(\R^d \times (-T,T))}^2\, \sup_{t \in [-T,T]}\|d\vec{u}\|_{H^{K-1}(\R^d)}^2
\end{equation}

In particular for any $K$ there exist some $\eps > 0$ and $C > 0$ (independent of $T_1,T_2$) such that if
\[
 \|d\vec{u}\|_{L^2_t L_x^{2d}(\R^d \times (-T,T))} < \eps
\]
then
\[
 \sup_{t \in [-T,T]}\|d\vec{u}\|_{H^{K-1}(\R^d)}^2 \leq C \|\vec{u_0}\|_{H^{K}(\R^d)}^2.
\]
\end{proposition}
\begin{proof}
For $s \geq 0$
\[
 E_s(t) := \int_{\R^d} |\partial_t \Ds{s} \vec{u}|^2 + |\nabla \Ds{s}\vec{u}|^2.
\]
We have
\[
 \frac{d}{dt} E_s(t) := \int_{\R^d} \langle\Ds{s} (\partial_{tt} - \lap) \vec{u}, \Ds{s} \partial_t \vec{u}\rangle
\]

Let
\[
\begin{split}
H(\vec{u}) :=& \vec{u} \brac{|\nabla \vec{u}|^2-|\partial_t \vec{u}|^2} + \Pi_{\vec{u}_\perp} \Dso \vec{u}\, \langle \vec{u} \cdot \Dso \vec{u} \rangle \\
&+\vec{u} \wedge \Dso{} \brac{ \vec{u} \wedge\Dso{} \vec{u}}- \vec{u} \wedge \brac{\vec{u} \wedge (-\lap) \vec{u}}.
\end{split}
\]
Then we know from the Krieger-Sire result that
\[
 \begin{cases}
    (\partial_{tt} -\lap)\vec{u} = H(\vec{u}) \quad \text{in $\R^d \times(T_1,T_2)$}\\
    \vec{u}(0) = \vec{u_0}\quad& \text{in $\R^d$}\\
    \partial_t \vec{u}(0) = \vec{u_0} \wedge \Dso{} \vec{u_0} \quad& \text{in $\R^d$}
 \end{cases}
\]

If $s \geq 0$ (observe that $\partial_t \vec{u} \cdot \vec{u} = 0$),
\[
\begin{split}
&\abs{\int_{\R^d} \partial_t \Ds{s} \vec{u} \cdot \Ds{s} (\vec{u} |d \vec{u}|^2)} \\
\aleq&\abs{\int_{\R^d} \langle\Ds{s} \partial_t \vec{u},\vec{u}\rangle \Ds{s} |d \vec{u}|^2}+\|\Ds{s} d\vec{u}\|_{L^2(\R^d)}\, \cdot\\
&\quad \cdot \max_{s_1+s_2+s_3 = s, s_1 > 3/4} \|\Ds{s_1} \vec{u}\|_{L^{\frac{2d}{d-2(s-s_1+1)-(1-\sigma_1)(d-1-2s)}}(\R^d)}\, \|\Ds{s_2} d\vec{u}\|_{{L^{\frac{2d}{d-2(s-s_2)-(1-\sigma_2)(d-1-2s)}}(\R^d)}}\\
&\quad \cdot \|\Ds{s_3} d\vec{u}\|_{L^{\frac{2d}{d-2(s-s_3)-(1-\sigma_3)(d-1-2s)}}(\R^d)}\\
\overset{\eqref{eq:GN:1v2}}{\aleq}&\abs{\int_{\R^d} \langle\Ds{s} \partial_t \vec{u},\vec{u}\rangle \Ds{s} |d \vec{u}|^2}\\
&+\|\Ds{s} d\vec{u}\|_{L^2(\R^d)}\, \|\Dso \vec{u}\|_{L^{2d}(\R^d)}^{2}\, \|\Ds{s}d \vec{u}\|_{L^{2d}(\R^d)}
\end{split}
\]
where $\sigma_1+\sigma_2+\sigma_3 = 1$, and we needed to ensure that
\[
 \sigma_1 \in [\frac{(s_1-1)_+}{s},1], \quad \sigma_2 \in [\frac{s_2}{s},1], \quad \sigma_3 \in [\frac{s_3}{s},1],
\]
which is possible since
\[
 \frac{(s_1-1)_+}{s} + \frac{s_2}{s}+\frac{s_3}{s} = \frac{s+(s_1-1)_-- s_1}{s}<1.
\]

For the other term, by the same argument, using that $|\vec{u}| \equiv 1$ we also get
\[
 \abs{\int_{\R^d} \langle\Ds{s} \partial_t \vec{u},\vec{u}\rangle \Ds{s} |d \vec{u}|^2} \aleq \|\Ds{s} d\vec{u}\|_{L^2(\R^d)}^2\, \|\Dso \vec{u}\|_{L^{2d}(\R^d)}^{2}.
\]
So we have
\[
 \frac{d}{dt} E_s(t) \aleq \|d\vec{u}\|_{L^{(2d,2)}(\R^d)}^2\, E_s(t).
\]
Integrating this implies for any $\tilde{T} \in [0,T]$
\[
 E_s(\tilde{T}) \leq E_s(0) \aleq \|d\vec{u}\|_{L^2_t L^{(2d,2)}(\R^d)}^2\, \max_{t \in [0,\tilde{T}]} E_s(t)
\]
taking the supremum of $\tilde{T}$ on the left-hand side, we conclude
\[
 \sup_{t \in [-T,T]} E_s(t) \aleq \|d\vec{u}\|_{L^2_t L^{(2d,2)}(\R^d)}^2\, \sup_{t \in [-T,T]} E_s(t).
\]
Recall that this holds for any $s \geq 0$.
We have established
\eqref{eq:dvecucontrol} and can conclude.

\end{proof}

As a consequence of \Cref{th:existsmoothhalfwave}, \Cref{pr:dvecucontrol} and the uniqueness result of \cite{ERFS22}, and the a priori estimates \Cref{th:aprioriest} we obtain global existence.

\begin{proof}[Proof of \Cref{co:wellposedness}]
Take $K$ sufficiently large, $\eps$ from \Cref{pr:dvecucontrol}.
We first assume that additionally
\begin{equation}\label{eq:largeinitialcontrol}
 \|\nabla u_0\|_{H^{K-1}} < \infty
\end{equation}
Take $T$ from \Cref{th:existsmoothhalfwave} for $\Gamma := C \|\nabla u_0\|_{H^{K-1}}$, where $C>1$ is from \Cref{pr:dvecucontrol}.

Then there exists a uniform $T_1 := T$ such that we can solve
\[
 \begin{cases} \partial_t \vec{u_1} = \vec{u_1} \wedge \laps{1} \vec{u_1} \quad &\text{in $\R^d \times (0,T)$}\\
  \vec{u_1}(0) = \vec{u_0}\quad &\text{in $\R^d$}
 \end{cases}
\]
For $k \geq 2$ set $T_k := (k-\frac{3}{2})T$ and $u_k := u_{k-1} \Big |_{T_k}$, which solves
\[
 \begin{cases} \partial_t \vec{u_k} = \vec{u_k} \wedge \laps{1} \vec{u_k} \quad &\text{in $\R^d \times (T_k-T,T_k+T)$}\\
  \vec{u_{k}}(T_k) = \vec{u_{k-1}}\quad &\text{in $\R^d$}
 \end{cases}
\]
Since by \Cref{pr:dvecucontrol} we know $\vec{u_1}(T_1-T)$ satisfies uniform bounds, namely
\[
 \|\nabla u_1(T_2)\|_{H^{K-1}} \leq C \|\nabla u_0\|_{H^{K-1}} \leq \Gamma
\]
we $u_2$ exist for the same $T$ as before, by \Cref{th:existsmoothhalfwave}. By uniqueness proven in \cite{ERFS22} (observe that we have \eqref{eq:KSPDEbdsHk}) we conclude that $u_2 \equiv u_1$ where the domain of definition overlap. In particular we have from \Cref{pr:dvecucontrol}
\[
 \|\nabla u_2(T_3)\|_{H^{K-1}} \leq C \|\nabla u_0\|_{H^{K-1}} \leq \Gamma
\]
We repeat this argument inductively to obtain $u_k$, and we have that $u_k = u_{k-1}$ where their domains of definitions overlap.

Thus we may set $u(x,t) := u_k(x,t)$ whenever $t \in (T_k-T,T_k+T)$ we have found a solution in $[0,\infty)$ (and similarly we argue to obtain a solution in $(-\infty,0]$.

So we have established global existence under the additional assumption \eqref{eq:largeinitialcontrol} (and from the a priori estimates \Cref{th:aprioriest} we obtain
\begin{equation}\label{eq:estpde2}
\begin{split}
\|\Ds{\frac{d-2}{2}} d\vec{u}\|_{C_t^0 L^2(\R^d\times (-\infty,\infty))} + \|\Ds{s-1}d \vec{u}\|_{L^2_t L^{(\frac{2d}{2s-1},2)}_x(\R^d \times (-\infty,\infty))} \aleq  \|\Ds{\frac{d}{2}} \vec{u_0}\|_{L^2(\R^d)}
\end{split}
\end{equation}
which in particular implies \eqref{eq:Wellposedest})

In case that $\vec{u_0}$ does not satisfy \eqref{eq:largeinitialcontrol}. Since $\|\Ds{\frac{d}{2}}\vec{u}_0\|_{L^2(\R^d)} < \eps$
by density of smooth maps in $\dot{H}^{\frac{d}{2}}(\R^d,\S^2)$, we find $\vec{u_0}$ by $\vec{u_{0;\ell}} \in C^\infty(\R^d,\S^2)$ which still satisfies
\[
 \|\Ds{\frac{d}{2}}\vec{u_{0;\ell}}\|_{L^2(\R^d)} < 2\eps
\]
Since we have the support condition that $\supp \vec{u_0}-\vec{q}$ is compact we can also obtain that $\vec{u_{0;\ell}} \in H^{K}$ for any $K$.

We then find a global solution $\vec{u}_{\ell}$ a global solution with initial data $\vec{u_{0;\ell}}$ with the estimates \eqref{eq:estpde2}. We pass to the limit and by lower semicontinuity of the involved norms we obtain
\begin{equation}\label{eq:estpde3}
\begin{split}
\|\Ds{\frac{d-2}{2}} d\vec{u}\|_{L^\infty_t L^2(\R^d\times (-\infty,\infty))} + \|\Ds{s-1}d \vec{u}\|_{L^2_t L^{(\frac{2d}{2s-1},2)}_x(\R^d \times (-\infty,\infty))} \aleq  \|\Ds{\frac{d}{2}} \vec{u_0}\|_{L^2(\R^d)}
\end{split}
\end{equation}

Again this implies \eqref{eq:Wellposedest}. As mentioned before the uniqueness result was proven in \cite{ERFS22}.
\end{proof}

\appendix

\section{Leibniz Rule and Sobolev-Gagliardo-Nirenberg embeddings}
In this section we formulate several estimates for the Leibnizrule for operators like $\Ds{\alpha}$ as well as Gagliardo-Nirenberg-type inequalities. These estimates are known at least to experts and we make no claim of originality.

\subsection{Leibniz rule}
We begin by the basic Leibniz rule, which can be proven, e.g. \cite[Theorem 7.1.]{LScomm}, see also the presentation in \cite[Theorem 3.4.1]{Ingmanns20}.

Here, and henceforth the Leibniz-rule operator is denoted by
\[
 H_{\Ds{\alpha}}(f,g) = \Ds{\alpha}(fg) - (\Ds{\alpha} f)\, g-f\, \Ds{\alpha}g
\]
A simple, yet extremely useful observation by many people is that for $\alpha \in (0,2)$
\begin{equation}\label{eq:HDsalpharep}
 H_{\Ds{\alpha}}(f,g)(x) = c\int_{\R^d} \frac{\brac{f(x)-f(y)}\, \brac{g(x)-g(y)}}{|x-y|^{d+\alpha}}\, dy.
\end{equation}

\begin{lemma}[Leibniz rule estimate]\label{la:basicLeibniz}
Let $\alpha \in (0,1)$ and $\beta \in (0,2)$. Pick any $\gamma \in (0,1)$ such that $\alpha+\beta-\gamma \in (0,1)$, and $p,p_1,p_2 \in (1,\infty)$ such that
 \[
  \frac{1}{p} = \frac{1}{p_1} + \frac{1}{p_2}.
 \]
Then
  \begin{equation}\label{eq:comm:123}
  \|\Ds{\alpha} H_{\Ds{\beta}} (f,g)\|_{L^{p}(\R^d)} \aleq \|\Ds{\gamma} f\|_{L^{p_1}}\, \|\Ds{\alpha +\beta-\gamma} g\|_{L^{p_2}}.
 \end{equation}
\end{lemma}

As a consequence of repeated Leibniz-type arguments we also have
\begin{lemma}\label{la:reallyworstterm}
Assume that $\vec{u}: \R^d \to \R^N$ such that $|\vec{u}|_{\R^N}^2 \equiv 1$ a.e. in $\R^d$. Then,
\[
 \norm{\sum_{j} u^j H_{\Dso}(u^j,F)}_{L^2(\R^d)} \aleq \|\Dso \vec{u}\|_{L^{2d}(\R^d)}^2\, \|F\|_{L^2(\R^d)}
\]
\end{lemma}
\begin{proof}
We are going to use a well-known facts, such as
\[
\frac{ |f(x)-f(y)|}{|x-y|^{\alpha}} \aleq \mathcal{M}|\Ds{\alpha} f|(x)+\mathcal{M}|\Ds{\alpha} f|(y) \quad \text{a.e. $x,y \in \R^n$}
\]
which can be found e.g. in \cite[Proposition 6.6.]{Schikorra18}; Here $\mathcal{M}$ is the Hardy-Littlewood maximal function.
Also, we observe that since $|\vec{u}(x)|^2=|\vec{u}(y)|^2=1$
\[
 \vec{u}(x) \cdot (\vec{u}(x)-\vec{u}(x))= \frac{1}{2} \abs{\vec{u}(x)-\vec{u}(x)}^2.
\]
Then, by duality, for some $\|G\|_{L^2(\R^d)} \aleq 1$,  and for any choice of $\alpha \in (1/2,1)$ such that $d-(2\alpha-1)2 > 0$,
\[
\begin{split}
 &\|\sum_{j} u^j H_{\Dso}(u^j,F)\|_{L^2(\R^d)}\\
 \aleq &\int_{\R^d} \sum_{j} u^j H_{\Dso}(u^j,F)\, G\\
 \overset{\eqref{eq:HDsalpharep}}{=}&\frac{c}{2}\int_{\R^d}\int_{\R^d} \frac{|\vec{u}(x)-\vec{u}(y)|^2\, \brac{F(x)-F(y)} G(x)}{|x-y|^{d+1}}\, dy dx\\
 \aleq & \abs{\int_{\R^d} F(x)G(x)\int_{\R^d} \frac{|\vec{u}(x)-\vec{u}(y)|^2}{|x-y|^{d+1}}\, dy dx}\\
&+\abs{\int_{\R^d} G(x)\int_{\R^d} \frac{|\vec{u}(x)-\vec{u}(y)|^2 F(y)}{|x-y|^{d+1}}\, dy dx}\\
 = & \abs{\int_{\R^d} F(x)G(x) H_{\Dso}(\vec{u}\cdot,\vec{u})}\\
&+\abs{\int_{\R^d} G(x)\int_{\R^d} \frac{|\vec{u}(x)-\vec{u}(y)|^2 F(y)}{|x-y|^{d+1}}\, dy dx}\\
\aleq& \|F\|_{L^2(\R^d)}\, \|G\|_{L^2(\R^d)}\, \|H_{\Dso}(\vec{u}\cdot,\vec{u})\|_{L^\infty}\\
&+\int_{\R^d} \abs{G(x)}\int_{\R^d} \frac{\brac{ \brac{\mathcal{M}\Ds{\alpha} \vec{u}(x)}^2 + \brac{\mathcal{M}\Ds{\alpha} \vec{u}(y)}^2}\abs{F(y)}}{|x-y|^{d+1-2\alpha}}\, dy dx\\
\overset{\text{\cite[(A.7)]{ERFS22}}}{\aleq}& \|F\|_{L^2(\R^d)}\, \, \|\Dso \vec{u}\|_{L^{(2d,2)}(\R^d)}^2\\
&+\int_{\R^d} \abs{G(x)}\, \brac{\mathcal{M}\Ds{\alpha} \vec{u}(x)}^2  \int_{\R^d} |x-y|^{2\alpha-1-d} \abs{F(y)}\, dy dx\\
&+\int_{\R^d} \brac{\mathcal{M}\Ds{\alpha} \vec{u}(y)}^2\, \abs{F(y)}\int_{\R^d} |x-y|^{2\alpha-1-d} \abs{G(x)} dx dy\\
=& \|F\|_{L^2(\R^d)}\, \, \|\Dso \vec{u}\|_{L^{(2d,2)}(\R^d)}^2\\
&+\int_{\R^d} \abs{G(x)}\, \brac{\mathcal{M}\Ds{\alpha} \vec{u}(x)}^2\, \lapms{(2\alpha-1)} \abs{F}(x) dx\\
&+\int_{\R^d} \brac{\mathcal{M}\Ds{\alpha} \vec{u}(y)}^2\, \abs{F(y)} \lapms{(2\alpha-1)} \abs{G}(y)  dy\\
\aleq& \|F\|_{L^2(\R^d)}\, \|\Dso \vec{u}\|_{L^{(2d,2)}(\R^d)}^2\\
&+\|G\|_{L^2(\R^d)}\, \|\brac{\mathcal{M}\Ds{\alpha} \vec{u}}^2\|_{L^{\frac{2d}{2(2\alpha-1)}}(\R^d)}  \|\lapms{(2\alpha-1)} \abs{F}\|_{L^{\frac{2d}{d-(2\alpha-1)2}}(\R^d)}\\
&+\|\brac{\mathcal{M}\Ds{\alpha} \vec{u}}^2\|_{L^{\frac{2d}{2(2\alpha-1)}}(\R^d)}\, \|F\|_{L^2(\R^d)} \|\lapms{(2\alpha-1)} \abs{G} \|_{L^{\frac{2d}{d-(2\alpha -1)2}}(\R^d)}\\
\aleq& \|F\|_{L^2(\R^d)}\, \|\Dso \vec{u}\|_{L^{(2d,2)}(\R^d)}^2\\
&+\|G\|_{L^2(\R^d)}\, \|\mathcal{M}\Ds{\alpha} \vec{u}\|_{L^{\frac{2d}{2\alpha-1}}(\R^d)}^2  \|F\|_{L^{2}(\R^d)}\\
&+\|\mathcal{M}\Ds{\alpha} \vec{u}\|_{L^{\frac{2d}{2\alpha-1}}(\R^d)}^2\, \|F\|_{L^2(\R^d)} \|G \|_{L^{2}(\R^d)}\\
\aleq& \|F\|_{L^2(\R^d)}\, \|\Dso \vec{u}\|_{L^{(2d,2)}(\R^d)}^2\\
&+\|G\|_{L^2(\R^d)}\, \|\Dso \vec{u}\|_{L^{2d}(\R^d)}^2  \|F\|_{L^{2}(\R^d)}\\
&+\|\Dso \vec{u}\|_{L^{2d}(\R^d)}^2\, \|F\|_{L^2(\R^d)} \|G \|_{L^{2}(\R^d)}\\
 \end{split}
\]
We can conclude.
\end{proof}

\begin{lemma}\label{la:reallyworsttermv2}
Set
\[
H_{\Dso,3}(f,g,h)(x) := c\int_{\R^d} \frac{ \brac{f(x)-f(y)}\, (g(x)-g(y))\, \brac{h(x)-h(y)}}{|x-y|^{d+1}}\, dy
\]
Denote by $\tilde{D}^s$ either $\Ds{s}$ or $D^{\lfloor s\rfloor} \Ds{s-\lfloor s\rfloor}$. Then if $d \geq 4$ is even and $t=0$ or if $d\geq 5$ and $t=\frac{1}{2}$, assuming that $s_1,s_2,s_3 \in [0,\frac{d}{2}-t)$ and
\[
 s_1+s_2+s_3 + t = \frac{d}{2}
\]
Then
\[
\|\Ds{t} H_{\Dso,3}(\tilde{D}^{s_1}u^i,\tilde{D}^{s_2} u^j,\tilde{D}^{s_3}u^k)\|_{L^2(\R^d)} \aleq \|\Dso \vec{u}\|_{L^{2d}(\R^d)}^2\, \|\Ds{\frac{d}{2}} \vec{u}\|_{L^{2}(\R^d)}
\]
\end{lemma}
\begin{proof}
If $d$ is even, and thus $t=0$, from \cite[Lemma A.6.]{ERFS22}, for $\sigma_1,\sigma_2,\sigma_3 \in [0,1]$ such that $\sigma_1+\sigma_2+\sigma_3 = 1$ we have
\[
\begin{split}
& \|\Ds{t} H_{\Dso,3}(\tilde{D}^{s_1}u^i,\tilde{D}^{s_2} u^j,\tilde{D}^{s_3}u^k)\|_{L^2(\R^d)}\\
\aleq&\|\Ds{s_1+\alpha_1} u^i\|_{L^{\frac{2d}{1+\sigma_1+2(s_1+\alpha_1-1)}}(\R^d)}\, \|\Ds{s_2 + \alpha_2} u^j \|_{L^{\frac{2d}{1+\sigma_2+2(s_2+\alpha_2-1)}}(\R^d)}\, \|\Ds{s_3+\alpha_3} u^k\|_{L^{\frac{2d}{1+\sigma_3+2(s_3+\alpha_3-1)}}(\R^d)}\\
 \overset{\eqref{eq:GN:1}}{\aleq}&\|\Dso \vec{u}\|_{L^{2d}(\R^d)}^2\, \|\Ds{\frac{d}{2}} \vec{u}\|_{L^2(\R^d)}.
 \end{split}
\]
For this to be correct, we need some assumptions.

For the application of  \cite[Lemma A.6.]{ERFS22} observe that \cite[(A.4)]{ERFS22} is naturally satisfied since $s_1,s_2,s_3 < \frac{d}{2}$.

For the application of \eqref{eq:GN:1} we need $s_j+\alpha_j-1 < \frac{d-2}{2}$ which is easy to establish since $\alpha_1+\alpha_2+\alpha_3 = 1$, $s_1+s_2+s_3 = \frac{d}{2}$, and each $s_i < \frac{d}{2}$ (and $d \geq 4$).

So the conditions we need to ensure are
\[
\sigma_j \geq \max\{\frac{2(s_j+\alpha_j-1)}{d-2},0\}, \quad j=1,2,3.
\]
If all $s_j+\alpha_j-1 > 0$ for $j=1,2,3$ then we see that
\[
\sum_{j=1}^3 \max\{\frac{2(s_j+\alpha_j-1)}{d-2},0\} = \frac{d-4}{d-2} <1
\]
so we can easily choose $\sigma_j$ with the requirements above.

If (say) $s_1+\alpha_1-1 <0$ and $s_j+\alpha_j-1 >0$, $j=2,3$ then we can assume that $\alpha_1$ is as close to $1$ as possible, which means
\[
 s_2+\alpha_2-1+s_3+\alpha_3-1 \aeq  \sum_{j=1}^3 \frac{2(s_j+\alpha_j-1)}{d-2} = \frac{d-4}{d-2} < 1.
\]
Then we choose
\[
\sigma_j \geq \max\{\frac{2(s_j+\alpha_j-1)}{d-2},0\}, \quad j=2,3
\]
such that $\sigma_2+\sigma_3 < 1$, and pick $\sigma_3 = 1-\sigma_2-\sigma_3$. Clearly any permutation of indices does not affect the argument.

The last case is that (say) $s_1+\alpha_1-1 <0$ and $s_2+\alpha_2-1 <0$. In this case we can assume (by choice of $\alpha_1,\alpha_2$ as large as possible)
\[
s_j+\alpha_j-1 > -\frac{1}{4},
\]
so that
\[
1+2(s_j+\alpha_j-1) > 0 \quad j=1,2
\]
Then we simply choose
\[
 \sigma_3 \in (\frac{2(s_3+\alpha_3-1)}{d-2},1)
\]
and for any choice of $\sigma_1$, $\sigma_2$ such that $\sigma_1+\sigma_2+\sigma_3$ we found an admissible tuple for the application above of \eqref{eq:GN:1}.

The same principle can be applied for $t=\frac{1}{2}$ and $d\geq 5$ odd. In this case we need to replace  \cite[Lemma A.6.]{ERFS22} to show that
\[
\begin{split}
\|\Ds{\frac{1}{2}} H_{\Dso,3}(f,g,h)\|_{L^2(\R^d)}
 \aleq &\|\Ds{\alpha_1+\frac{1}{2}} f\|_{L^{p_{1;1}}(\R^d)}\, \|\Ds{\alpha_2} g\|_{L^{p_{2;1}}(\R^d)}\, \|\Ds{\alpha_3} h\|_{L^{p_{3;1}}(\R^d)}\\
 &\|\Ds{\alpha_1} f\|_{L^{p_{1;2}}(\R^d)}\, \|\Ds{\alpha_2+\frac{1}{2}} g\|_{L^{p_{2;2}}(\R^d)}\, \|\Ds{\alpha_3} h\|_{L^{p_{3;2}}(\R^d)}\\
 &\|\Ds{\alpha_1} f\|_{L^{p_{1;3}}(\R^d)}\, \|\Ds{\alpha_2} g\|_{L^{p_{2;3}}(\R^d)}\, \|\Ds{\alpha_3+\frac{1}{2}} h\|_{L^{p_{3;3}}(\R^d)}.
 \end{split}
\]
The proof is left to the reader, it follows from the arguments \cite[Lemma A.6.]{ERFS22} combined with the estimates in \cite[Section 3]{MSY21} after representing
\[
\begin{split}
H_{\Dso,3}(f,g,h)(x) := c\int_{\R^{3d}}  K(x,y,z_1,z_2,z_3) \Ds{\alpha_1}f(z_1)\, \Ds{\alpha_2}g(z)\, \Ds{\alpha_3}h(z)\, dz_1 dz_2 dz_3,
 \end{split}
\]
where
\[
\begin{split}
& K(x,y,z_1,z_2,z_3)\\
 =& \int_{\R^d} \frac{ \brac{|x-z_1|^{\alpha_1-d} -|y-z_1|^{\alpha_1-d}}\, \brac{|x-z_2|^{\alpha_2-d} -|y-z_2|^{\alpha_2-d}}\, \brac{|x-z_3|^{\alpha_3-d} -|y-z_3|^{\alpha_3-d}}}{|x-y|^{d+1}}\, dy
 \end{split}
\]
and estimating with the arguments of \cite[Section 3]{MSY21}
\[
 \int_{\R^n} \frac{K(x_1,y,z_1,z_2,z_3)-K(x,y,z_1,z_2,z_3)}{|x_1-x|^{d+\frac{1}{2}}}\, dx_1
\]
We leave the details to the reader.
\end{proof}

\subsection{Gagliardo-Nirenberg type inequalities}
We also will use versions of the Gagliardo-Nirenberg inequality, which we collect in the following lemma

\begin{lemma}[Higher order Gagliardo-Nirenberg]
Throughout this lemma we assume $d \geq 3$.
\begin{enumerate}
\item
For $s \in [0,\frac{d-2}{2}]$ and $\sigma \in [\frac{2s}{d-2},1]$
we have
\begin{equation}\label{eq:GN:1}
 \|\Ds{s} f\|_{L^{\frac{2d}{1+\sigma+2s}}(\R^d)} \aleq \|f\|_{L^{2d}(\R^d)}^{1-\sigma} \ \|\Ds{\frac{d-2}{2}} f\|_{L^2(\R^d)}^{\sigma}.
\end{equation}
The same holds for $s < 0$ as long as $\sigma \in [0,1]$ and we have $\frac{2d}{1+\sigma+2s} \in (1,\infty)$.

\item For $s \geq 0$, $t \in [0,s]$ and $\sigma \in [\frac{t}{s},1]$
\begin{equation}\label{eq:GN:1v2}
 \|\Ds{t} f\|_{L^{\frac{2d}{d-2(s-t)-(1-\sigma)(d-1-2s)}}(\R^d)} \aleq \|f\|_{L^{2d}(\R^d)}^{1-\sigma} \ \|\Ds{s} f\|_{L^2(\R^d)}^{\sigma}.
\end{equation}

\item For  any $\sigma \in [\frac{2\alpha}{d-2},1]$, $\alpha \in \R$ with $\frac{2d}{2+2\alpha+\sigma} \in (1,\infty)$ we have
\begin{equation}\label{eq:GN:234}
 \|\Ds{\alpha} (fg)\|_{L^{\frac{2d}{2+2\alpha+\sigma}}(\R^d)} \aleq \max_{\tilde{f},\tilde{g} \in \{f,g\}}\|\tilde{g}\|_{L^{2d}(\R^d)}^{2-\sigma} \ \|\Ds{\frac{d-2}{2}} \tilde{f}\|_{L^2(\R^d)}^{\sigma}
\end{equation}

\item We have
\begin{equation}\label{eq:gagliardo5}
\begin{split}
 &\|\Ds{\frac{d-2}{2}} \brac{{Q} f\, H_{\Dso}(\Ds{-1}g,\Ds{-1}h)}\|_{L^2}\\
 \aleq &
 \brac{1+\|Q\|_{L^\infty(\R^d)}}\, \max_{\tilde{f} \in \in \{f,g,h, {\Dso \tilde{Q}}\}} \brac{\|\Ds{\frac{d-2}{2}} \tilde{f}\|_{L^2(\R^d)}+\|\Ds{\frac{d-2}{2}} \tilde{f}\|_{L^2(\R^d)}}\, \max_{\tilde{g}\in \{f,g,h, {\Dso \tilde{Q}}\}} \|\tilde{g}\|_{L^{(2d,2)}(\R^d)}^{{2}}
 \end{split}
 \end{equation}
\item Recall the commutator notation
\[
 [T,f](g) = T(fg)-fT(g).
\]
Let $s \leq \frac{d-2}{2}$, $\sigma \in [0,1]$ such that $1+2s  -\sigma > 0$. If $s \geq 1$ assume additionally $(1-\sigma) \in {(}\frac{2(s-1)}{d-2},1]$ (observe the strict inequality!)

Then
\begin{equation}\label{eq:GN:gagliardo6}
\|\Ds{s} \brac{[\Rz,\Ds{-1} f] (g)}\|_{L^{\frac{2d}{1+2s -\sigma}}} \aleq \max_{\tilde{f}, \tilde{g} \in \{f,g\}}\|\tilde{g}\|_{L^{2d}(\R^d)}^{1+\sigma} \ \|\Ds{\frac{d-2}{2}} \tilde{f}\|_{L^2(\R^d)}^{1-\sigma}
\end{equation}
\end{enumerate}
\end{lemma}

\begin{proof}
\begin{enumerate}
\item We prove \eqref{eq:GN:1}.

If $s <0$ then from the Sobolev inequality,
\[
\|\Ds{s} f\|_{L^{\frac{2d}{1+\sigma+2s}}(\R^d)} \aleq
\|f\|_{L^{\frac{2d}{1+\sigma}}(\R^d)}.
\]
Then the claim follows from the case $s=0$ (observe that the assumptions of $\sigma$ are still satisfied).

So, from now on we assume that $s \geq 0$.

For $\sigma = 1$ \eqref{eq:GN:1} is simply Sobolev inequality. For $\sigma = 0$, the assumption $\sigma \geq \frac{2s}{d-2}$ implies that $s \leq 0$, i.e. $s=0$, and then \eqref{eq:GN:1} is trivial.

So from now on assume $\sigma \neq 0,1$ and $s \geq 0$.

If $\sigma \geq \frac{2s}{d-2}$ then $t$ defined by
\[
\sigma= \frac{2t}{d-2} \quad \Leftrightarrow \quad t=\frac{d-2}{2}\sigma
\]
satisfies $t \in [s,\frac{d-2}{2})$.

Moreover, from Sobolev embedding
\[
 \|\Ds{s} f\|_{L^{\frac{2d}{1+\sigma+2s}}(\R^d)} \aleq  \|\Ds{t} f\|_{L^{\frac{2d}{1+\sigma+2t}}(\R^d)}.
\]
By this argument, otherwise replacing $s$ with $t$, we can reduce the claim to the case
\[
 \sigma =\frac{2s}{d-2}.
\]
Then,
\[
 s = (1-\sigma) 0 + \sigma \frac{d-2}{2},
\]
and
\[
 \frac{1+\sigma+2s}{2d} = \brac{1-\sigma} \frac{1}{2d} + \sigma \frac{1}{2}.
 \]
Consequently, classical Gagliardo-Nirenberg (see, e.g., \cite[Lemma 3.1.]{BM01})  implies
\[
 \|\Ds{s} f\|_{L^{\frac{2d}{1+\sigma+2s}}(\R^d)} \aleq \|f\|_{L^{2d}(\R^d)}^{1-\sigma} \ \|\Ds{\frac{d-2}{2}} f\|_{L^2(\R^d)}^{\sigma}.
\]

\item The proof of  \eqref{eq:GN:1v2} is very similar to \eqref{eq:GN:1}.

For $\sigma = 1$ it is Sobolev embedding. For $\sigma = 0$ we must have $t=s=0$ and then \eqref{eq:GN:1v2} is trivial. So from now on $\sigma \neq 0,1$.

If we set $\tilde{t}$ defined by
\[
\frac{\tilde{t}}{s} := \sigma \quad \Leftrightarrow \tilde{t} = s\sigma \geq t
\]
then from Sobolev embedding we have
\[
 \|\Ds{t} f\|_{L^{\frac{2d}{d-2(s-t)-(1-\sigma)(d-1-2s)}}(\R^d)} \aleq
  \|\Ds{\tilde{t}} f\|_{L^{\frac{2d}{d-2(s-\tilde{t})-(1-\sigma)(d-1-2s)}}(\R^d)}
\]
Now we have
\[
 \tilde{t} = (1-\sigma) 0  + \sigma s
\]
and
\[
\frac{d-2(s-\tilde{t})-(1-\sigma)(d-1-2s)}{2d} = (1-\sigma) \frac{1}{2d} + \sigma \frac{1}{2}
\]
so we have
\[
 \|\Ds{\tilde{t}} f\|_{L^{\frac{2d}{d-2(s-\tilde{t})-(1-\sigma)(d-1-2s)}}(\R^d)} \aleq \|f\|_{L^{2d}(\R^d)}^{1-\sigma} \ \|\Ds{s} f\|_{L^2(\R^d)}^{\sigma},
\]
from the usual Gagliardo-Nirenberg inequality.

\item Proof of \eqref{eq:GN:234}: The case $\alpha < 0$ can be reduced to $\alpha = 0$ by means of Sobolev embedding
\[
 \|\Ds{-\alpha} (fg)\|_{L^{\frac{2d}{2-2\alpha+\sigma}}(\R^d)} \aleq
 \|fg\|_{L^{\frac{2d}{2+\sigma}}(\R^d)}
 \]
So assume from now on $\sigma \in [\frac{2\alpha}{d-2},1]$ and $\alpha \in [0,\frac{d}{2}-1]$.

For each $\beta \in [0,\alpha]$ we pick $\sigma_1 \in [\frac{2\beta}{d-2},1]$, $\sigma_2 \in [\frac{2(\alpha-\beta)}{d-2},1]$ with $\sigma_1 + \sigma_2 = \sigma$.

Then, from the fractional Leibniz rule and \eqref{eq:GN:1},
\[
\begin{split}
 \|\Ds{\alpha} (fg)\|_{L^{\frac{2d}{2+2\alpha+\sigma}}(\R^d)} \aleq&  \max_{\beta \in [0,\alpha]} \|\Ds{\beta} f\|_{L^{\frac{2d}{1+\sigma_1+2\beta}}(\R^d)}\, \|\Ds{\alpha-\beta} g \|_{L^{\frac{2d}{1+\sigma_2+2(\alpha-\beta)}}(\R^d)} \\
 \aleq&\|f\|_{L^{2d}(\R^d)}^{1-\sigma_1} \ \|\Ds{\frac{d}{2}-1} f\|_{L^2(\R^d)}^{\sigma_1}\ \|g\|_{L^{2d}(\R^d)}^{1-\sigma_2} \ \|\Ds{\frac{d}{2}-1} g\|_{L^2(\R^d)}^{\sigma_2}\\
 =&\max_{\tilde{f},\tilde{g} \in \{f,g\}}\|\tilde{g}\|_{L^{2d}(\R^d)}^{2-\sigma} \ \|\Ds{\frac{d-2}{2}} \tilde{f}\|_{L^2(\R^d)}^{\sigma}
 \end{split}
\]
Here and throughout the paper we mean by an abuse of notation with $\max_{\beta \in [0,\alpha]}$ the maximum over finitely many specific $\beta \in \{0, \beta_1 ,\beta_2 ,\beta_{[a]} , \alpha\} \subset [0,\alpha]$ (that can be computed from the Leibniz rule).

\item In preparation of the proof of \eqref{eq:gagliardo5} we first prove for any $\alpha \geq 0$ such that $\frac{2d}{d-2\alpha} \in (1,\infty)$,
 \begin{equation}\label{eq:gagliardo6v2}
  \|\Ds{\frac{d-2}{2}-\alpha}\brac{f\, H_{\Dso} (\Ds{-1}g,\Ds{-1}h)}\|_{L^{\frac{2d}{d-2\alpha}}(\R^d)} \aleq \max_{\tilde{f},\tilde{g} \in \{f,g,h\}} \|\tilde{f}\|_{L^2}\, \|\tilde{g}\|_{L^{(2d,2)}(\R^d)}^2
 \end{equation}
By Sobolev inequality
\[
\begin{split}
  &\|\Ds{\frac{d-2}{2}-\alpha}\brac{f\, H_{\Dso} (\Ds{-1}g,\Ds{-1}h)}\|_{L^{\frac{2d}{d-2\alpha}}(\R^d)}\\
  \aleq&\|\Ds{\frac{d-2}{2}}\brac{f\, H_{\Dso} (\Ds{-1}g,\Ds{-1}h)}\|_{L^{2}(\R^d)}\\
  \end{split}
\]
so in order to prove \eqref{eq:gagliardo6v2} we may assume w.l.o.g. $\alpha = 0$.

We have from the Leibniz rule
\begin{equation}\label{eq:asuioxcvxuiocv23}
\begin{split}
& \|\Ds{\frac{d-2}{2}} \brac{f\, H_{\Dso}(\Ds{-1}g,\Ds{-1}h)}\|_{L^2} \\
\aleq & \|\Ds{\frac{d-2}{2}} f\|_{L^{2}(\R^d)}\ \|H_{\Dso}(\Ds{-1}g,\Ds{-1}h)\|_{L^{\infty}(\R^d)}\\
&+\max_{\gamma_1+\gamma_2+\gamma_3 = \frac{d-2}{2}} \|\Ds{\gamma_1} f\|_{L^{p_1}} \|\Ds{\gamma_2} g\|_{L^{p_2}} \|\Ds{\gamma_3} h\|_{L^{p_3}}
\end{split}
\end{equation}
Here, in the max we have $\gamma_1 \in [0,\frac{d}{2}-1)$, $\gamma_2 > -1$, $\gamma_3 > -1$. Indeed, by the fractional Leibniz rule we can ensure for any fixed $\eps_0 > 0$
\[
 \gamma_2,\gamma_3 > -\eps_0
\]
Moreover $p_1,p_2,p_3$ can be chosen differently for each tuple $(\gamma_1,\gamma_2,\gamma_3)$.

The first term in \eqref{eq:asuioxcvxuiocv23} can be estimated by \cite[Lemma A.8, (A.7)]{ERFS22}
\[
 \|\Ds{\frac{d-2}{2}} f\|_{L^2} \|H_{\Dso}(\Ds{-1}g,\Ds{-1}h)\|_{L^\infty} \aleq  \|\Ds{\frac{d-2}{2}} f\|_{L^2} \|g\|_{L^{(2d,2)}(\R^d)}\, \|h\|_{L^{(2d,2)}(\R^d)}
\]

For the second term in \eqref{eq:asuioxcvxuiocv23} we first consider the case where \underline{$\gamma_3 \leq 0$ and $\gamma_2 \leq 0$}. Observe that $\gamma_2,\gamma_3 > -1$ so the following makes sense. We estimate the second line of \eqref{eq:asuioxcvxuiocv23} by
\[
\begin{split}
\aleq&\|\Ds{\gamma_1} f\|_{L^{\frac{d2}{2+2\gamma_1}}(\R^d)}\ \|\Ds{\gamma_2} g\|_{L^{\frac{2d}{1+\gamma_2 2}}(\R^d)}\, \|\Ds{\gamma_3} h\|_{L^{\frac{2d}{1+\gamma_3 2}}(\R^d)}\\
 \overset{\gamma_2,\gamma_3 \leq 0}\aleq& \|\Ds{\frac{d-2}{2}} f\|_{L^2(\R^d)}\, \|g\|_{L^{2d}(\R^d)}\, \|h\|_{L^{2d}(\R^d)}
 \end{split}
\]
Now we assume that \underline{$\gamma_3 \leq 0$ and $\gamma_2 > 0$}. Then we estimate the second line of \eqref{eq:asuioxcvxuiocv23} by
\[
\begin{split}
 \aleq&\|\Ds{\gamma_1} f\|_{L^{\frac{2d}{1+\sigma_1+2\gamma_1}}(\R^d)}\ \|\Ds{\gamma_2} g\|_{L^{\frac{2d}{1+\sigma_2+2\gamma_2}}(\R^d)}\, \|\Ds{\gamma_3} h\|_{L^{\frac{2d}{1+\gamma_3 2}}(\R^d)}\\
 \overset{\gamma_3 < 0}{\aleq}&\|\Ds{\gamma_1} f\|_{L^{\frac{2d}{1+\sigma_1+2\gamma_1}}(\R^d)}\ \|\Ds{\gamma_2} g\|_{L^{\frac{2d}{1+\sigma_2+2\gamma_2}}(\R^d)}\, \|h\|_{L^{2d}(\R^d)},
 \end{split}
\]
where we choose $\sigma_1+\sigma_2 = 1$ so that
\[
 \frac{1+\sigma_1+2\gamma_1}{2d}+\frac{1+\sigma_2+2\gamma_2}{2d}=\frac{1}{2}
\]
We need to ensure that
\[
 \sigma_i \geq \frac{2\gamma_i}{d-2} \quad i=1,2
\]
Observe that since $\gamma_3$ is negative, but very close to zero,
\[
\sum_{i=1}^2 \frac{2\gamma_i}{d-2} =\frac{d-4}{d-2} +\frac{-2\gamma_3}{d-2} \overset{|\gamma_3| \ll 1}{<} 1 =\sigma_1 + \sigma_2
\]
so the conditions on $\sigma_i$ can be satisfied, and we have
\[
\begin{split}
&\|\Ds{\gamma_1} f\|_{L^{\frac{2d}{1+\sigma_1+2\gamma_1}}(\R^d)}\ \|\Ds{\gamma_2} g\|_{L^{\frac{2d}{1+\sigma_2+2\gamma_2}}(\R^d)}\, \|h\|_{L^{2d}(\R^d)}\\
\aleq&\|f\|_{L^{2d}(\R^d)}^{1-\sigma_1}\, \|\Ds{\frac{d-2}{2}} f\|_{L^2}^{\sigma_1}\
\|g\|_{L^{2d}(\R^d)}^{1-\sigma_2}\, \|\Ds{\frac{d-2}{2}} g\|_{L^2}^{\sigma_2}\ \|h\|_{L^{2d}(\R^d)}\\
\aleq&\max_{\tilde{f}} \|\Ds{\frac{d-2}{2}} \tilde{f}\|_{L^2(\R^d)} \, \max_{\tilde{g}\in \{f,g,h\}} \|\tilde{g}\|_{L^{(2d,2)}(\R^d)}^2
 \end{split}
\]

The case \underline{$\gamma_3 \leq 0$ and $\gamma_2 > 0$} is almost verbatim to the previous one.

Now we assume \underline{$\gamma_1,\gamma_2,\gamma_3 \geq 0$}. In this case we estimate the second line of \eqref{eq:asuioxcvxuiocv23} by
\[
\begin{split}
\aleq&\max_{\gamma_1+\gamma_2+\gamma_3 = \frac{d-2}{2}}\|\Ds{\gamma_1} f\|_{L^{\frac{2d}{1+\sigma_1+2\gamma_1}}}\, \|\Ds{\gamma_2} g\|_{L^{\frac{2d}{1+\sigma_2+2\gamma_2}}}\, \|\Ds{\gamma_3} h\|_{L^{\frac{2d}{1+\sigma_3+2\gamma_3}}}\\
 \end{split}
\]
where and we choose $\sigma_1 + \sigma_2 + \sigma_3=1$. In order to conclude with the help of \eqref{eq:GN:1} we set
\[
 \sigma_i := \frac{2\gamma_i}{d-2}, \quad i=1,2,3.
\]
We have established \eqref{eq:gagliardo6v2}.

\underline{We now prove \eqref{eq:gagliardo5}}
Using the fractional Leibniz rule we have
\[
\begin{split}
 &\|\Ds{\frac{d-2}{2}} \brac{{Q} f\, H_{\Dso}(\Ds{-1}g,\Ds{-1}h)}\|_{L^2}\\
 \aleq&\|Q\|_{L^\infty}\, \|\Ds{\frac{d-2}{2}} \brac{f\, H_{\Dso}(\Ds{-1}g,\Ds{-1}h)}\|_{L^2}\\
 &+\|\Ds{\frac{d}{2}-1}Q\|_{L^\frac{2d}{d-2}(\R^d)}\, \|f\|_{L^d} \|H_{\Dso}(\Ds{-1}g,\Ds{-1}h)\|_{L^{\infty}}\\
 &+\max_{\alpha \in (0,\frac{d}{2}-1)}\|\Ds{\alpha}Q\|_{L^{\frac{d}{\alpha}}(\R^d)}\, \|\Ds{\frac{d}{2}-1-\alpha} \brac{f\, H_{\Dso}(\Ds{-1}g,\Ds{-1}h)}\|_{L^{\frac{2d}{d-2\alpha}}}.
 \end{split}
\]
The first line is estimated by \eqref{eq:gagliardo6v2} ($\alpha = 0$). The second line is estimated by Sobolev embedding and  \cite[Lemma A.8, (A.7)]{ERFS22},
\[
\begin{split}
 &\|\Ds{\frac{d}{2}-1}Q\|_{L^\frac{2d}{d-2}(\R^d)}\, \|f\|_{L^d} \|H_{\Dso}(\Ds{-1}g,\Ds{-1}h)\|_{L^{\infty}}\\
 \aleq &\|\Ds{\frac{d}{2}} Q\|_{L^2(\R^d)}\, \|\Ds{\frac{d-2}{2}} f\|_{L^2(\R^d)}\ \|f\|_{L^{(2d,2)}(\R^d)}\, \|g\|_{L^{(2d,2)}(\R^d)}.
\end{split}
 \]
For the third line we use again \eqref{eq:gagliardo6v2} and Sobolev inequality and conclude
\[
 \begin{split}
 &\|\Ds{\alpha}Q\|_{L^{\frac{d}{\alpha}}(\R^d)}\, \|\Ds{\frac{d}{2}-1-\alpha} \brac{f\, H_{\Dso}(\Ds{-1}g,\Ds{-1}h)}\|_{L^{\frac{2d}{d-2\alpha}}}\\
 \aleq&\|\Ds{\frac{d}{2}}Q\|_{L^{2}(\R^d)}\, \max_{\tilde{f} \in \{f,g,h\}} \|\Ds{\frac{d-2}{2}} \tilde{f}\|_{L^2(\R^d)}\, \max_{\tilde{g}\in \{f,g,h\}} \|\tilde{g}\|_{L^{(2d,2)}(\R^d)}^{{2}}\end{split}
 \]

\item Now we prove \eqref{eq:GN:gagliardo6}:

The case \underline{$s < 0$} can be reduced to the case $s=0$, since by assumption, $\frac{2d}{1+2s -\sigma} \in (1,\infty)$, and by Sobolev embedding,
\[
 \|\Ds{s} \brac{[\Rz,\Ds{-1} f] (g)}\|_{L^{\frac{2d}{1+2s -\sigma}}} \aleq
 \|\brac{[\Rz,\Ds{-1} f] (g)}\|_{L^{\frac{2d}{1 -\sigma}}}.
\]
So assume now \underline{$s \in [0,1)$, $\sigma \in [0,1]$, and either $\sigma < 1$ or $s>0$}. From commutator estimates, \cite[(2.3.3)]{Ingmanns20}, for $\sigma_1,\sigma_2 \in [0,1]$, $\sigma_1+\sigma_2 = 1-\sigma$, and
for any $\beta \in (0,1)$.
\[
\begin{split}
& \|\Ds{s} \brac{[\Rz,\Ds{-1} f] (g)}\|_{L^{\frac{2d}{1+2s -\sigma}}}\\
\aleq &\| \Ds{\beta-1} f\|_{L^{\frac{2d}{1+\sigma_1+2(\beta-1)}}} \|\Ds{s-\beta} g\|_{L^{\frac{2d}{1+\sigma_2+2(s-\beta)}}}\\
\end{split}
\]
Since $s < 1$ we can choose $s<\beta < 1$, and then with \eqref{eq:GN:1},
\[
\aleq \| f\|_{L^{2d}(\R^d)}^{1-\sigma_1} \|\Ds{\frac{d-2}{2}}f\|_{L^2}^{\sigma_1} \| g\|_{L^{2d}(\R^d)}^{1-\sigma_2} \|\Ds{\frac{d-2}{2}}f\|_{L^2}^{\sigma_2}.
\]

If \underline{$s=1$} we can make the same argument, only that we have $\beta<1=s$. However we can make $s-\beta$ arbitrarily small. So if $1-\sigma>0$ then we can take $\beta$ suitably close to $1$ so that $\sigma_2$ additionally satisfies $\sigma_2 \geq \frac{2(s-\beta)}{d-2}$, and then the same argument as above is true.

We argue similarly for \underline{$s > 1$}, $\sigma \in [0,1]$ and $(1-\sigma) \in {(}\frac{2(s-1)}{d-2},1]$. From commutator estimates, \cite[(2.3.3)]{Ingmanns20}, for any $\delta > 0$
\[
\begin{split}
&\|\Ds{s} \brac{[\Rz,\Ds{-1} f] (g)}\|_{L^{\frac{2d}{1+2s  -\sigma}}}\\
=&\|\Ds{s} \brac{[\Rz,\Ds{-1} f] (g)}\|_{L^{\frac{2d}{2+2(s-1) + 1-\sigma}}}\\
\aleq &\max_{t \in [-\delta,s-1]} \|\Ds{t} f\|_{L^{\frac{2d}{1+\sigma_1+2t}}(\R^d)} \|\Ds{s-1-t} g\|_{L^{\frac{2d}{1+\sigma_2+2(s-1-t)}}(\R^d)}\\
\overset{\eqref{eq:GN:1}}{\aleq}&\|f\|_{L^{2d}(\R^d)}^{1-\sigma_1} \ \|\Ds{\frac{d}{2}-1} f\|_{L^2(\R^d)}^{\sigma_1}\ \|g\|_{L^{2d}(\R^d)}^{1-\sigma_2} \ \|\Ds{\frac{d}{2}-1} g\|_{L^2(\R^d)}^{\sigma_2}\\
 =&\max_{\tilde{f},\tilde{g} \in \{f,g\}}\|\Dso \tilde{g}\|_{L^{2d}(\R^d)}^{1+\sigma} \ \|\Ds{\frac{d}{2}} \tilde{f}\|_{L^2(\R^d)}^{1-\sigma}.
 \end{split}
\]
For this to make sense we must choose $\sigma_1 + \sigma_2 = 1-\sigma$, $\sigma_1,\sigma_2 \in [0,1]$ such that additionally
\[
 \sigma_1 \in [\frac{2t}{d-2},1], \quad
 \sigma_2 \in [\frac{2(s-1-t)}{d-2},1].
\]
If $t \geq 0$ this is no problem, by the assumptions, since
\[
 (1-\sigma) \in {(}\frac{2(s-1)}{d-2},1]
\]
For the case $t < 0$, we need to assume that $\delta \ll 1$ so that
\[
 (1-\sigma)  >\frac{2(s-1+\delta)}{d-2}.
\]
Then pick $\sigma_1 = 0$ and we observe
\[
 \sigma_2 =  (1-\sigma) \in [\frac{2(s-1-t)}{d-2},1].
\]
\end{enumerate}
\end{proof}

\section{Strichartz estimates}\label{s:strichartz}
The following estimate is a consequence of a careful inspection of the arguments in \cite[p.566]{SS02}. See also \cite[Corollary 1.3]{KT98}. It follows from Duhamel principle and the semigroup-representation for the homogeneous wave equation
\begin{equation}\label{eq:strich:veq}
\begin{cases}
 (\partial_{tt} - \lap )v =0 \quad &\text{in $\R^d \times [0,T] $}\\
 v = f\quad &\text{in $\R^d \times \{0\}$}\\
 \partial_t v = g\quad &\text{in $\R^d \times \{0\}$}.
 \end{cases}
\end{equation}
Namely we have
\[
 v(x,t) = \frac{e^{\i t \sqrt{-\lap}} + e^{-\i t \sqrt{-\lap}}}{2} f(x)+\frac{e^{\i t \sqrt{-\lap}} - e^{-\i t \sqrt{-\lap}}}{2\i} \sqrt{-\lap}^{-1}g(x)
\]
and thus
\[
  \sqrt{-\lap}^{-1} \partial_t v(x,t) = -\frac{e^{\i t \sqrt{-\lap}} {-}  e^{-\i t \sqrt{-\lap}}}{2{\i}}  f(x)+\frac{e^{\i t \sqrt{-\lap}} {+}  e^{-\i t \sqrt{-\lap}}}{{2}} \sqrt{-\lap}^{-1} g(x)
\]

\begin{lemma}\label{la:strichartz}
Let $d \geq 4$. Assume for $T > 0$
\begin{equation}\label{eq:strich:ueq}
\begin{cases}
 (\partial_{tt} - \lap )u =h \quad &\text{in $\R^d \times [-T,T]$}\\
 u  = f\quad &\text{in $\R^d\times\{0\}$}\\
 \partial_t u = g\quad &\text{in $\R^d \times \{0\}$}\\
 \end{cases}
\end{equation}
then for all $s \in (\frac{1}{2},\frac{d^2-4d+1}{2(d-1)}+1]$,
\[
\begin{split}
 &\|\Ds{\frac{d-2}{2}}du(t)\|_{C^0_t L^{2}_x(\R^d) \times (-T,T)}  + \|\Ds{s} u(t)\|_{L^2_t L^{(\frac{2d}{2s-1},2)}_x(\R^d \times (-T,T)}\\
 \aleq& \|\Ds{\frac{d}{2}} f\|_{L^2(\R^d)} + \|\Ds{\frac{d-2}{2}} g\|_{L^2(\R^d)}\\
 &+\|\Ds{\frac{d-2}{2}}h\|_{L^1_t L^2_x (\R^d \times (-T,T))}.
 \end{split}
\]
as well as
\[
\begin{split}
 \|\Ds{s-1} du(t)\|_{L^2_t L^{(\frac{2d}{2s-1},2)}_x(\R^d \times (-T,T))} \aleq &\|\Ds{\frac{d}{2}} f\|_{L^2(\R^d)} + \|\Ds{\frac{d-2}{2}} g\|_{L^2(\R^d)}\\
 &+\|{\Ds{\frac{d-2}{2}}h}\|_{L^1_t L^2_x (\R^d \times (-T,T))}.
 \end{split}
\]
In particular,
\[
\begin{split}
 &\|\Ds{\frac{d-4}{2}}du(t)\|_{C^0_t L^{2}_x(\R^d) \times (-T,T)}  + \|\Ds{s -1}u(t)\|_{L^2_t L^{(\frac{2d}{2s-1},2)}_x(\R^d \times (-T,T)} \\
 \aleq& \|\Ds{\frac{d-2}{2}} f\|_{L^2(\R^d)} + \|\Ds{\frac{d-4}{2}} g\|_{L^2(\R^d)}\\
 &+\|\Ds{\frac{d-4}{2}}h\|_{L^1_t L^2_x (\R^d \times (-T,T))}.
 \end{split}
\]

\end{lemma}

\bibliographystyle{abbrv}%
\bibliography{bib}%

\end{document}